%% file: motif.tex
\documentclass[11pt]{article}

\usepackage{amssymb,amsbsy,amsmath,amscd,amsthm,amsfonts,bbm}
\usepackage{cite}
\usepackage{enumerate}
\usepackage{dsfont}
\usepackage{graphicx}
\usepackage{longtable}
\usepackage{tablefootnote}
\usepackage[margin=1in]{geometry}

\usepackage{subfigure}

\usepackage[
            CJKbookmarks=true,
            bookmarksnumbered=true,
            bookmarksopen=true,
            colorlinks=true,
            citecolor=red,
            linkcolor=blue,
            anchorcolor=red,
            urlcolor=blue
            ]{hyperref}

\newtheorem{theorem}{Theorem}

\newtheorem{lemma}{Lemma}
\newtheorem{remark}{Remark}

\newcommand{\Binom}{\mathrm{Bin}}

\usepackage{tikz}
\usetikzlibrary{matrix,arrows,calc,shapes,backgrounds}
\usetikzlibrary{shapes.callouts,decorations.text}

\usetikzlibrary{shapes.misc}

\tikzset{cross/.style={cross out, draw=black, minimum size=2*(#1-\pgflinewidth), inner sep=0pt, outer sep=0pt},
cross/.default={1pt}}

\tikzstyle{int}=[draw, fill=blue!20, minimum size=2em]
\tikzstyle{dot}=[circle, draw, fill=blue!20, minimum size=2em]
\tikzstyle{dotred}=[circle, draw, fill=red!20, minimum size=2em]
\tikzstyle{init} = [pin edge={to-,thin,black}]
\tikzstyle{initred} = [pin edge={to-,thin,red}]
\tikzstyle{plan}=[draw, fill=blue!20, minimum size=2em, text width=5em, rounded corners,align=center]
\tikzstyle{planwide}=[draw, fill=blue!20, minimum size=2em, text width=8em, rounded corners,align=center]
\tikzstyle{vertexdot}=[circle, draw, fill=white, minimum size=3,inner sep=0pt]
\tikzstyle{vertexdotsolid}=[circle, draw, fill=black, minimum size=3,inner sep=0pt]

\newcommand{\Vertex}{\tikz[scale=0.75]{\draw (0,0) node (zero) [vertexdot] {};}}
\newcommand{\Edge}{\tikz[scale=0.75,baseline={([yshift=-3pt]zero.base)}]{\draw (0,0) node (zero) [vertexdot] {} -- (0.5,0) node (zero) [vertexdot] {};}}
\newcommand{\Halfedge}{\tikz[scale=0.75,baseline={([yshift=-3pt]zero.base)}]{\draw (0,0) node (zero) [vertexdotsolid] {} -- (0.5,0) node (zero) [vertexdot] {};}}
\newcommand{\Fulledge}{\tikz[scale=0.75,baseline={([yshift=-3pt]zero.base)}]{\draw (0,0) node (zero) [vertexdotsolid] {} -- (0.5,0) node (zero) [vertexdotsolid] {};}}

\newcommand{\PathThree}{\tikz[scale=0.75]{\draw (0,0) node (zero) [vertexdot] {} -- (0.5,0) node (zero) [vertexdot] {} -- (1.0,0) node (zero) [vertexdot] {};}}
\newcommand{\PathFour}{\tikz[scale=0.75]{\draw (0,0) node (zero) [vertexdot] {} -- (0.5,0) node (zero) [vertexdot] {} -- (1.0,0) node (zero) [vertexdot] {}-- (1.5,0) node (zero) [vertexdot] {};}}

\newcommand{\Square}{\tikz[scale=0.75,baseline=(zero.base)]{\draw (0,0) node (zero) [vertexdot] {} -- (0.5,0) node[vertexdot] {} -- (0.5,0.5) node[vertexdot] {} -- (0,0.5) node[vertexdot] {} -- cycle;}}

\newcommand{\Triangle}{\tikz[scale=0.75,baseline=(zero.base)]{\draw (0,0) node (zero) [vertexdot] {} -- (0.5,0) node[vertexdot] {} -- (0.25,{0.5*sin(60)}) node[vertexdot] {} -- cycle;}}

\newcommand{\Paw}{\tikz[scale=0.75,baseline=(zero.base)]{\draw (0,0) node (zero) [vertexdot] {} -- (0.5,0) node[vertexdot] {} -- (0.25,{0.5*sin(60)}) node[vertexdot] {} -- cycle; \draw (0.5,0) node[vertexdot] {} -- (1,0) node[vertexdot] {};}}

\newcommand{\Bowtie}{\tikz[scale=0.75,baseline=(zero.base)]{\draw (0,0) node (zero) [vertexdot] {} -- (30:0.5) node[vertexdot] {} -- (-30:0.5) node[vertexdot] {}
-- (150:0.5) node[vertexdot] {} -- (-150:0.5) node[vertexdot] {} -- cycle;}}

\newcommand{\Claw}{\tikz[scale=0.75,baseline=(zero.base)]{\draw (0,0) node (zero) [vertexdot] {} -- (0.5,0) node[vertexdot] {} -- (0.5,0.5) node[vertexdot] {}; \draw (0.5,0) node[vertexdot] {} -- (1,0) node[vertexdot] {};}}

\newcommand{\Wedge}{\tikz[scale=0.75,baseline=(zero.base)]{\draw (0,0) node (zero) [vertexdot] {} -- (0.25,{0.5*sin(60)}) node[vertexdot] {} -- (0.5,0) node[vertexdot] {};}}
\newcommand{\Kfour}{\tikz[scale=0.75,baseline=(zero.base)]{\draw (0,0)--(0.5,0.5); \draw (0.5,0)--(0,0.5); \draw (0,0) node (zero) [vertexdot] {} -- (0.5,0) node[vertexdot] {} -- (0.5,0.5) node[vertexdot] {} -- (0,0.5) node[vertexdot] {} -- cycle; }}

\renewcommand{\Diamond}{\tikz[scale=0.75,baseline=(zero.base)]{\draw (0.5,0)--(0,0.5); \draw (0,0) node (zero) [vertexdot] {} -- (0.5,0) node[vertexdot] {} -- (0.5,0.5) node[vertexdot] {} -- (0,0.5) node[vertexdot] {} -- cycle; }}

\newcommand{\inj}{\mathsf{inj}}

\newcommand{\aut}{\mathsf{aut}}

\newcommand{\Kempty}{K^{\circ}}
\newcommand{\Kfull}{K^{\bullet}}

\renewcommand{\hat}{\widehat}
\renewcommand{\tilde}{\widetilde}

\newcommand{\ehat}{\hat{\sfe}}
\newcommand{\HT}{\mathsf{HT}}

\newcommand{\floor}[1]{{\left\lfloor {#1} \right \rfloor}}
\newcommand{\ceil}[1]{{\left\lceil {#1} \right \rceil}}

\newcommand{\naturals}{\mathbb{N}}
\newcommand{\integers}{\mathbb{Z}}
\newcommand{\Expect}{\mathbb{E}}

\newcommand{\Prob}{\mathbb{P}}
\newcommand{\prob}[1]{\Prob\left[#1\right]}
\newcommand{\pprob}[1]{\Prob[#1]}

\newcommand{\TV}{{\rm TV}}

\newcommand{\iid}{i.i.d.\xspace}

\newcommand{\pth}[1]{\left( #1 \right)}
\newcommand{\qth}[1]{\left[ #1 \right]}
\newcommand{\sth}[1]{\left\{ #1 \right\}}

\newcommand{\ppth}[1]{( #1 )}

\newcommand{\iiddistr}{{\stackrel{\text{\iid}}{\sim}}}

\newcommand{\Var}{\mathsf{Var}}
\newcommand{\Cov}{\mathsf{Cov}}

\newcommand{\Bern}{\mathrm{Bern}}

\newcommand{\Indc}{\mathbbm{1}}
\newcommand{\indc}[1]{\Indc\left\{{#1}\right\}}

\newcommand{\tG}{{\widetilde{G}}}
\newcommand{\tH}{{\widetilde{H}}}

\newcommand{\sfd}{{\mathsf{d}}}
\newcommand{\sfe}{{\mathsf{e}}}

\newcommand{\sfn}{{\mathsf{n}}}

\newcommand{\sft}{{\mathsf{t}}}

\newcommand{\sfv}{{\mathsf{v}}}
\newcommand{\sfw}{{\mathsf{w}}}

\newcommand{\sfN}{{\mathsf{N}}}

\newcommand{\calF}{{\mathcal{F}}}
\newcommand{\calG}{{\mathcal{G}}}

\newcommand{\calK}{{\mathcal{K}}}

\newcommand{\calP}{{\mathcal{P}}}

\usepackage{prettyref,xspace}
\newrefformat{eq}{(\ref{#1})}
\newrefformat{thm}{Theorem~\ref{#1}}
\newrefformat{sec}{Section~\ref{#1}}
\newrefformat{algo}{Algorithm~\ref{#1}}
\newrefformat{fig}{Fig.~\ref{#1}}
\newrefformat{tab}{Table~\ref{#1}}
\newrefformat{rmk}{Remark~\ref{#1}}
\newrefformat{def}{Definition~\ref{#1}}
\newrefformat{cor}{Corollary~\ref{#1}}
\newrefformat{lmm}{Lemma~\ref{#1}}
\newrefformat{prop}{Proposition~\ref{#1}}
\newrefformat{app}{Appendix~\ref{#1}}
\newrefformat{ex}{Example~\ref{#1}}

\newcommand{\s}{\mathsf{s}}

\newcommand{\n}{\mathsf{n}}

\newcommand{\cc}{\mathsf{cc}}

\newcommand{\lcm}{{\mathsf{lcm}}}

\newcommand{\N}{{\mathsf{N}}}

\newcommand{\id}{\mathrm{id}}

\graphicspath {{graphs/}{experiments/}}

\title{Counting Motifs with Graph Sampling}

\date{\today}
\author{Jason M. Klusowski \and Yihong Wu\thanks{The authors are with the Department of Statistics and Data Science, Yale University, New Haven, CT, 06511, emails: \url{jason.klusowski@yale.edu} and \url{yihong.wu@yale.edu}. This research was supported in part by the NSF Grant IIS-1447879, CCF-1527105, and an NSF CAREER award CCF-1651588.}}

\date{\today}

\begin{document}

\maketitle

\input{abstract}

\tableofcontents

\input{main}


\end{document}

%% file: abstract.tex
\begin{abstract}

Applied researchers often construct a network from data that has been collected from a random sample of nodes, with the goal to infer properties of the parent network from the sampled version. Two of the most widely used sampling schemes are \emph{subgraph sampling}, where we sample each vertex independently with probability $p$ and observe the subgraph induced by the sampled vertices, and \emph{neighborhood sampling}, where we additionally observe the edges between the sampled vertices and their neighbors.

In this paper, we study the problem of estimating the number of motifs as induced subgraphs under both models from a statistical perspective. We show that: for parent graph $G$ with maximal degree $d$, for any connected motif $h$ on $k$ vertices, to estimate the number of copies of $h$ in $G$, denoted by $s=\s(h,G)$,  with a multiplicative error of $\epsilon$, 
\begin{itemize}
	\item For subgraph sampling, the optimal sampling ratio $p$ is $\Theta_{k}(\max\{ \ppth{s\epsilon^2}^{-\frac{1}{k}}, \;  \frac{d^{k-1}}{s\epsilon^{2}} \})$, which only depends on the size of the motif but \emph{not} its actual topology. Furthermore, we show that Horvitz-Thompson type estimators are universally optimal for any connected motifs.
	
	\item For neighborhood sampling, we propose a family of estimators, encompassing and outperforming the Horvitz-Thompson estimator and achieving the sampling ratio $O_{k}(\min\{ (\frac{d}{s\epsilon^2})^{\frac{1}{k-1}}, \; \sqrt{\frac{d^{k-2}}{s\epsilon^2}}\})$, which again only depends on the size of $h$. This is shown to be optimal for all motifs with at most $4$  vertices and cliques of all sizes. 
\end{itemize}
The matching minimax lower bounds are established using certain algebraic properties of subgraph counts. These results allow us to quantify how much more informative neighborhood sampling is than subgraph sampling, as empirically verified by experiments on synthetic and real-world data. We also address the issue of adaptation to the unknown maximum degree, and study specific problems for parent graphs with additional structures, e.g., trees or planar graphs.

\end{abstract}

%% file: main.tex
\section{Introduction}

Counting the number of features in a graph is an important statistical and computational problem. These features are typically basic local structures like motifs \cite{Alon2002} or graphlets \cite{Prvzulj2004} (e.g., patterns of small subgraphs). Seeking to capture the interactions and relationships between groups and individuals, applied researchers typically construct a network from data that has been collected from a random sample of nodes. This scenario is sometimes due to resource constraints (e.g., massive social network, surveying a hidden population) or an inability to gain access the full population (e.g., historical data, corrupted data). Most of the problems encountered in practice are motivated by the need to infer global properties of the parent network (population) from the sampled version. For specific motivations and applications of statistical inference on sampled graphs, we refer the reader to the 
\cite{Duffield2014,Leskovec2006, Kolaczyk2017} for comprehensive reviews as well as applications in computer networks and social networks. 

From a computational and statistical perspective, it is desirable to design sublinear time (in the size of the graph) algorithms which typically involves random sampling as a primitive to reduce both time and sample complexities. Various sublinear-time algorithms based on edge and degree queries have been proposed to estimate graph properties such as the average degree \cite{Goldreich2008, Feige2006}, triangles \cite{Eden2015}, stars \cite{aliakbarpour2017sublinear}, and more general subgraph counts \cite{Gonen2011}. In all of these works, however, some form of adaptive queries, e.g.~breadth or depth first search, is performed, which can be impractical or unrealistic in the context of certain applications such as social network analysis \cite{Apicella2012} or econometrics \cite{Chandrasekhar2011}, where a  sampled graph is obtained and statistical analysis is to be conducted on the basis of this dataset alone.
In this work, we focus on data arising from specific sampling models, in particular, subgraph sampling and neighborhood sampling \cite{Lovasz12}, 
two of the most popular and commonly used sampling models in part due to their simplicity and ease of implementation. In subgraph sampling, we sample each vertex independently with equal probability and observe the subgraph induced by these sampled vertices. In neighborhood sampling, we additionally observe the edges between the sampled vertices and their neighbors. 
Despite their ubiquity, theoretical understanding of these sampling models in the context of statistical inference and estimation has been lacking. 

In this paper, we study the problem of estimating the counts of various classes of motifs, such as edges, triangles, cliques, and wedges, from a statistical perspective. Network motifs are important local properties of a graph. Detecting and counting motifs have diverse applications in a suite of scientific applications including gene regulation networks \cite{Alon2002},  protein-protein interaction networks \cite{uetz2000comprehensive}, and social networks \cite{wasserman1994social}.
Throughout this paper, motifs will be viewed as \emph{induced subgraphs} of the parent graph. 
For a subgraph $ H $, the number of copies of $H$ contained in $G$ as induced subgraphs 
is denoted by $ \s(H, G) $. 
Many useful graph statistics can often be expressed in terms of induced subgraph counts, e.g., the global clustering coefficients, which is the density of induced open triangles. It is worth pointing out that in some literature motifs are also understood as (not necessarily induced) subgraphs \cite{Alon2002}. 
In fact, it is well-known that the number of a given subgraph can be expressed as a linear combination of induced subgraph counts. 
For instance, if we denote the number of copies of $H$ contained in $G$ as subgraphs by $\n(H,G)$, then for wedges, we have $\n(\Wedge,G) = \s(\Wedge,G) + 3\s(\Triangle, G) $.\footnote{More generally, we have (cf.~\cite[Eq.~(5.15)]{Lovasz12}):
\begin{equation}
\n(H,G) = \sum_{H'} \n(H,H') \s(H',G),
\label{eq:injind}
\end{equation}
where the summation ranges over all simple graphs $H'$ (up to isomorphisms) obtained from $H$ by adding edges.}
For this reason, we focus on counting motifs as induced subgraphs.
Furthermore, while we make no assumption about the connectivity of the parent graph, we focus on motifs being \emph{connected} subgraphs which is the most relevant case for applications. It is a classical result that subgraph count of disconnected subgraphs can be expressed as a fixed polynomials in terms of the connected ones; cf.~\cite{Whitney1932,Lovasz12}.
Additionally, motifs in directed graphs have also been considered \cite{Alon2002}; in this paper we focus on undirected simple graphs.

The purpose of this paper is to develop a statistical theory for estimating motif counts in sampled graph. We will be concerned with both methodologies as well as their statistical optimality, with focus on large graphs and the sublinear sample regime, where only a vanishing fraction of vertices are sampled.
In particular, a few questions we want to address quantitatively are as follows:
\begin{itemize}
	\item How does the sample complexity depend on the motif itself? For example, 
	is estimating the count of open triangles as easy as estimating the closed triangles?
	How does the sample complexity of counting 4-cycles compare with that of counting 4-cliques?
	\item How much of the graph must be observed to ensure accurate estimation? For example, severe under-coverage issues have been observed in the study of protein-protein interaction networks \cite{Han2005}.
\item How much more informative is neighborhood sampling than subgraph sampling from the perspective of reducing the sample complexity?
\item To what extent does additional structures of the parent graph, e.g., tree or planarity, impact the sample complexity?
\end{itemize}

Finally, let us also mention that motif counts e.g.,~triangles \cite{GaoLafferty2017}, wheels \cite{Bickel2011}, and cycles \cite{MosselNeemanSly2015} have been used as useful test statistics for generative network models such as the stochastic block models. Furthermore, edges counts of similarity and dependency graphs have been used in the context of testing and estimating change-point detection \cite{chen2015graph, chu2017asymptotic}.
In this paper we do not assume any generative network model, and the randomness of the problem comes solely from the sampling mechanism.

\subsection{Sampling model} \label{sec:model}

In this subsection, we formally describe the two graph sampling models we will study in the remainder of the paper.



\paragraph{Subgraph sampling.}

Fix a simple graph $ G = (V, E) $ on $ \sfv(G) $ vertices. For $ S \subset V $, we denote by $ G[S] $ the vertex induced subgraph. If $ S $ represents a collection of vertices that are randomly sampled according to a sampling mechanism, we denote $ G[S] $ by $ \tG $. The first and simplest sampling model we consider is the subgraph sampling model, where each vertex is sampled with equal probability. In particular, we sample each vertex independently with probability $p$, where $p$ is called the \emph{sampling ratio} and can be thought of as the fraction of the graph that is observed. Thus, the sample size $|S|$ is distributed as $\Binom(\sfv(G), p) $, and the probability of observing a subgraph isomorphic to $ H $ is equal to
\begin{equation}
\pprob{ \tG \simeq H } = \s(H, G)p^{\sfv(H)}(1-p)^{\sfv(G)-\sfv(H)} .
\label{eq:pmf-bern}
\end{equation}
There is also a variant of this model where exactly $ n = p\sfv(G) $ vertices are chosen uniformly at random without replacement from the vertex set $ V $. In the sublinear sampling regime where $n \ll \sfv(G)$, they are nearly equivalent. 

\paragraph{Neighborhood sampling.} 
In this model, in addition to observing $ G[S] $, we also observe the labelled neighbors of all vertices in $ S $, denoted by $ G\{ S \} $. That is, $ G\{ S \} $ is equal to $ \tG = (V, \widetilde{E}) $, where $ \widetilde{E} = \cup_{ v \in S} \cup_{ u \in N_G(v) } \{u, v\} $ 
together with the colors $ b_v \in \{0,1\}$ for each $ v \in V(\tG) $, indicating which vertices were sampled. We refer to such bicolored graphs as \emph{neighborhood subgraphs}, which is a union of stars with the root vertex of each star colored. This model is also known in the literature as ego-centric \cite{Handcock2010} or star sampling \cite{Capobianco72,Kolaczyk2009}.

In other words, we randomly sample the rows of the adjacency matrix of $ G $ independently with probability $ p $ and then observe the rows together with the row indices. The graph then consists of unions of star graphs (not necessarily disjoint) together with colors indicating the root of the stars. 
Neighborhood sampling operates like subgraph sampling but neighborhood information is acquired for each sampled vertex. Hence neighborhood sampling is more informative in the sense that, upon sampling the same set of vertices, considerably more edges are observed. 
For an illustration and comparison of both subgraph and neighborhood sampling, see \prettyref{fig:example}.
Thus it is reasonable to expect (and indeed we will prove in the sequel) that for the same statistical task, neighborhood sampling typically has significantly lower sample complexity than the subgraph sampling scheme.
Note that in many cases, neighborhood sampling is more realistic than subgraph sampling (e.g., social network crawling), where sampling a vertex means that its immediate connections (e.g., friends list) are obtained for free.

A more general version of the neighborhood sampling model is described by Lov\'asz in \cite[Section 1.7]{Lovasz12}, where each sample consists of a radius-$r$ (labeled) neighborhood rooted at a randomly chosen vertex. Since from a union of marked stars one can disassemble each star individually, our model is equivalent to this one with $r=1$.

It turns out that the knowledge of the colors provides crucial information about the sampled graph and affects the quality of estimation (see \prettyref{app:nocolor}). In practice, the model with labels is more realistic since the experimenter would know which nodes were sampled.
We henceforth assume that all sampled graphs obtained from neighborhood sampling are bicolored, with black and white vertices corresponding to sampled and non-sampled vertices, respectively. For a neighborhood subgraph $ h $, let $ V_{b}(h) $ (resp. $ \sfv_{b}(h) $) denote the collection (resp. number) of black vertices. Suppose $ H $ is a bicolored subgraph of $ G $. Let $ \N(H, G) $ be the number of ways that $ H $ can appear (isomorphic as a vertex-colored graph) in $ G $ from neighborhood sampling with $ \sfv_{b}(H) $ vertices. Thus,
\begin{equation*} 
\pprob{\tG \cong H} = \N(H, G){p^{\sfv_b(H)}q^{\sfv(G)-\sfv_b(H)}}.
\end{equation*}

\begin{figure}[ht]
	\centering
	\subfigure[Parent graph.]%
	{\label{fig:examplea} \includegraphics[width=0.24\textwidth]{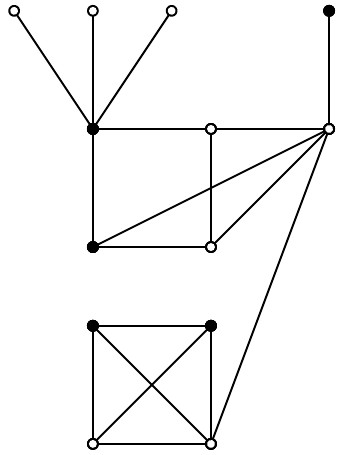}}
         ~~ \hspace{1cm}
	\subfigure[Subgraph sampling.]%
	{\label{fig:exampleb} \includegraphics[width=0.24\textwidth]{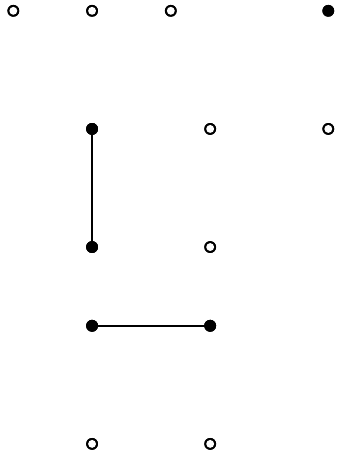}}
         ~~ \hspace{1cm}
	\subfigure[Neighborhood sampling.]%
	{\label{fig:examplec} \includegraphics[width=0.24\textwidth]{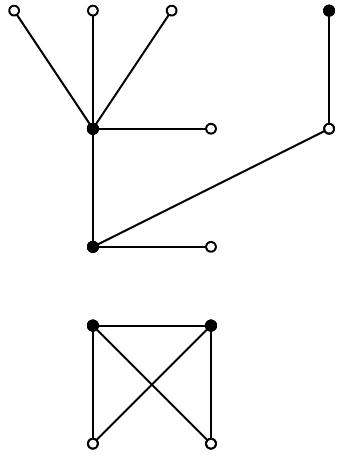}}
\caption{A comparison of subgraph and neighborhood sampling:
Five vertices are sampled in the parent graph, and the observed graph is shown in the middle and right panel for the subgraph and neighborhood sampling, respectively.
}
	\label{fig:example}
\end{figure}


\subsection{Main results}

Let $h$ denote a motif, which is a connected graph on $k$ vertices. As mentioned earlier, we do not assume any generative model or additional structures on the parent graph $G$, except that the maximal degree is at most $d$; this parameter, however, need not be bounded, and one of the goals is to understand how the sample complexity depends on $d$. 
The goal is to estimate the motif count $\s(h,G)$ based on the sampled graph $\tG$ obtained from either subgraph or neighborhood sampling.

Methodologically speaking, Horvitz-Thompson (HT) estimator \cite{HT52} is perhaps the most natural idea to apply here. The HT estimator is an unbiased estimator of the population total by weighting the empirical count of a given item by the inverse of the probability of observing said item. To be precise, consider estimate the edge count in a graph with $m$ edges and maximal degree $d$, the sampling ratio required by the HT estimator to achieve a relative error of $\epsilon$ 
scales as $\Theta(\max\{\frac{1}{\sqrt{m } \epsilon}, \frac{d}{m \epsilon^2} \})$, which turns out to be minimax optimal. 
For $\epsilon$ being a small constant, this yields a sublinear sample complexity when $m$ is large and $m \gg d$. 

For neighborhood sampling, which is more informative than subgraph sampling since more edges are observed, we show that the optimal sampling ratio can be improved to 
$\Theta(\min\{\frac{1}{\sqrt{m } \epsilon}, \frac{d}{m \epsilon^2} \})$, which, perhaps surprisingly, is not always achieved by the HT estimator. 
The main reason for its suboptimality in the high degree regime is the correlation between observed edges. To reduce correlation, we propose a family of  linear estimators encompassing and outperforming the Horvitz-Thompson estimator. The key idea is to use the color information indicating which vertices are sampled. For example, in a neighborhood sampled graph it is possible to observe two types of edges: $\Halfedge$ and $\Fulledge$. The estimator takes a linear combination of the count of these two types of edges with a negative weight on the latter, which, as counterintuitive as it sounds, significantly reduces the variance and achieves the optimal sample complexity. 

For general motifs $h$ on $k$ vertices, for subgraph sampling, it turns out the simple HT scheme for estimating $s = \s(h,G)$ achieves a multiplicative error of $\epsilon$ with the optimal sampling fraction 
\[
\Theta_{k}\pth{\max\sth{ \frac{1}{\pth{s\epsilon^2}^{\frac{1}{k}}}, \;  \frac{d^{k-1}}{s\epsilon^{2}} }},
\]
 which only depends on the size of the motif but \emph{not} its actual topology. 
For neighborhood sampling, the situation is more complicated and the picture is less complete. For general $h$, we propose a family of estimators that achieves the sample ratio:
\[
\Theta_{k}\pth{\min\sth{ \pth{\frac{d}{s\epsilon^2}}^{\frac{1}{k-1}}, \; \sqrt{\frac{d^{k-2}}{s\epsilon^2}}}}
\]
which again only depends on the size of $h$. We conjecture that this is optimal for neighborhood sampling and we indeed prove this for (a) all motifs up to $4$  vertices; (b) cliques of all sizes. 

Let us conclude this part by providing some intuition on proving the impossibility results. The main apparatus is matching subgraph counts: If two graphs have matching subgraphs counts for all induced (resp.~neighborhood) subgraphs up to size $k$, then the total variation of the sampled versions obtained from subgraph (resp.~neighborhood) sampling are at $O(p^k)$. 
At a high level, this idea is akin to the method of moment matching, which haven been widely used to prove statistical lower bound for functional estimation \cite{LNS99,CL11,WY14,JVHW15}; in comparison, in the graph-theoretic context, moments correspond to graph homomorphism numbers which are indexed by subgraphs instead of integers \cite{LS10}. To give a concrete example, consider the triangle motif and take 
\begin{equation}
H=\Paw \qquad H'=\Square
\label{eq:adhocHH1}
\end{equation}
 which have matching subgraph counts up to size two (equal number of vertices and edges) but distinct number of triangles. Then with subgraph sampling, the sampled graph satisfies $\TV(P_{\tH},P_{\tH'}) = O(p^3)$. For neighborhood sampling, we can take
\begin{equation}
H = \tikz[scale=0.75,baseline=(zero.base)]{\draw (0,0) node (zero) [vertexdot] {} -- (0.5,0) node[vertexdot] {} -- (0.25,{0.5*sin(60)}) node[vertexdot] {} -- cycle; \draw (0.5,0) node[vertexdot] {} -- (1,0) node[vertexdot] {} -- (1.5,0) node[vertexdot] {};}
\qquad
H'=
\tikz[scale=0.75,baseline=(zero.base)]{\draw (0,0) node (zero) [vertexdot] {} -- (0.5,0) node[vertexdot] {} -- (0.5,0.5) node[vertexdot] {} -- (0,0.5) node[vertexdot] {} -- cycle; \draw (0.5,0) node[vertexdot] {} -- (1,0) node[vertexdot] {};}
\label{eq:adhocHH2}
\end{equation}
which have \emph{matching degree sequences} $(3,2,2,2,1)$ but distinct number of triangles. 
In general, these pairs of graphs can be either shown to exist by the strong independence of graph homomorphism numbers for connected subgraphs \cite{Erdos1979} or explicitly constructed by a linear algebra argument \cite{Biggs1978}; however, for neighborhood sampling it is significantly more involved as we need to relate the neighborhood subgraph counts to the injective graph homomorphism numbers. Based on these small pairs of graphs, the lower bound in general is constructed by using either $H$ or $H'$ as its connected components.

\subsection{Notations}

We use standard big-$O$ notations,
e.g., for any positive sequences $\{a_n\}$ and $\{b_n\}$, $a_n=O(b_n)$ or $a_n \lesssim b_n$ if $a_n \leq C b_n$ for some absolute constant $C>0$,
$a_n=o(b_n)$ or $a_n \ll b_n$ or if $\lim a_n/b_n = 0$.
Furthermore, the subscript in $a_n=O_{r}(b_n)$ indicates that $a_n \leq C_r b_n$ for some constant $C_r$ depending on $r$ only. For nonnegative integer $ k $, let $ [k] = \{1, \dots, k \} $.

Next we establish some graph-theoretic notations that will be used throughout the paper. Let $ G = (V, E) $ be a simple, undirected graph. Let $ \sfe = \sfe(G) = |E(G)| $ denote the number of edges, $ \sfv = \sfv(G) = |V(G)| $ denote the number of vertices, and $ \cc = \cc(G) $ be the number of connected components in $G$. The open neighborhood of a vertex $ u $ is denoted by $ N_G(u) = \{ v \in V(G) : \{u, v\} \in E(G) \}  $. The closed neighborhood is defined by $ N_G[u] = \{u\}\vee N_G(u) $. Two vertices $ u $ and $ v $ are said to be adjacent to each other, denoted by $ u \sim v $, if $ \{u, v\} \in E(G) $.

Two graphs $ G $ and $ G' $ are isomorphic, denoted by $ G \simeq G' $, if there exists a bijection between the vertex sets of $ G $ and $ G' $ that preserves adjacency, i.e., if there exists a bijective function $ g: V(G) \to V(G') $ such that $ \{g(u), g(v)\} \in E(G') $ whenever $ \{u, v\} \in E(G) $.
If $ G $ and $ G' $ are vertex-colored graphs with colorings $ c $ and $ c' $ (i.e., a function that assigns a color to each vertex), then $ G $ and $ G' $ are isomorphic as vertex-colored graphs, denoted by $ G \cong G' $, if there exist a bijection $ g: V(G) \to V(G') $ such that $ \{g(u), g(v)\} \in E(G') $ whenever $ \{u, v\} \in E(G) $ and $ c(v) = c'(g(v)) $ 
for all vertices $ v \in V(G) $.

Let $ K_{n} $, $ P_{n} $, and $ C_{n} $ denote the complete graph (clique), path graph, and cycle graph on $ n $ vertices, respectively. Let $ K_{n,n'} $ denote the complete bipartite graph (biclique) with $ n n' $ edges and $ n + n' $ vertices. Let $ S_{n} $ denote the star graph $ K_{1, n} $ on $ n+1 $ vertices. 
%
%

Define the following graph operations cf.~\cite{west-book}:
The \emph{disjoint union} of graphs $G$ and $G'$, denoted $ G + G' $, is the graph whose vertex (resp.~edge) set is the disjoint union of the vertex (resp.~edge) sets of $G$ and of $G'$. For brevity, we denote by $ k G $ to the disjoint union of $ k $ copies of $ G $. The \emph{join} of $G$ and $G'$, denoted by $ G \vee G' $, is obtained from the disjoin union $G+G'$ by connecting all $v\in V(G)$ and all $v'\in V(G')$, that is, $G\vee G'= (V(H)\cup V(H'), E(H) \cup E(H') \cup (V(H) \times V(H')) ) $, where $H\simeq G$ and $H'\simeq G'$ and $V(H)$ and $V(H')$ are disjoint.
For example, $n K_1 \vee n' K_1 = K_{n,n'}$.
For $S\subset V(G)$, let $G-S$ denote the resulting graph after deleting all vertices in $S$ and all incident edges, and $G-v \triangleq G-\{v\}$.

We say that $H$ is an (edge-induced) subgraph of $G$, denoted by $H\subset G$, if $V(H)\subset V(G)$ and $E(H)\subset E(G)$.
For any $S \subset V(G)$, the subgraph of $G$ induced by $S$ is denoted by $G[S] \triangleq (S,E(G)\cap S\times S)$.
Let $\s(H,G)$ (resp.~$ \n(H, G) $) be the number of vertex (resp.~edge) induced subgraphs of $G$ that are isomorphic to $H$; in other words,
\begin{align}
\s(H,G) = & ~ \sum_{V\subset V(G)}   \indc{G[V] \simeq H}  \label{eq:ind}\\
\n(H,G) = & ~ \sum_{g \subset G}   \indc{g \simeq H} . \label{eq:inj}
\end{align}
For example, $\s(\PathThree,\; \Diamond) = 2$ and
$\n(\PathThree,\; \Diamond) = 8$.
Let $\omega(G)$ denote the clique number, i.e., the size of the largest clique in $G$.
Let $\sfe(G) = \s(\Edge, \; G) $, $ \sft(G) = \s(\Triangle, \; G) $ and $ \sfw(G) = \s(\Wedge, \; G) $ denote the number of edges, triangles and wedges of $ G $, which are of particular interest.


\subsection{Organization}

The paper is organized as follows. In \prettyref{sec:alg}, we state our positive results in terms of squared error minimax rates and design algorithms that achieve them for subgraph (\prettyref{sec:subgraph}) and neighborhood (\prettyref{sec:neighborhood}) sampling. \prettyref{sec:lower} discusses converse results and states counterpart minimax lower bounds for subgraph (\prettyref{sec:lb-subgraph}) and neighborhood (\prettyref{sec:lb-neighborhood}) sampling. We further restrict the class of graphs to be acyclic or planar in \prettyref{sec:structure} and explore whether such additional structure can be exploited to improve the quality of estimation. In \prettyref{sec:experiments}, we perform a numerical study of the proposed estimators for counting edges, triangles, and wedges on both simulated and real-world data. Finally, in \prettyref{app:appendix}, we prove some of the auxiliary lemmas and theorems that were stated in the main body of the paper.

\section{Methodologies and performance guarantees} \label{sec:alg}

%

\subsection{Subgraph sampling} \label{sec:subgraph}

The motivation for our estimation scheme is based on the observation that any motif count $ \s(h, G) $ can be written as a sum of indicator functions 
as in \prettyref{eq:ind}. 
Note that for a fixed subset of vertices $ T \subset V(G) $, the probability it induces a subgraph in the sampled graph $\tG$ that is isomorphic to $h$ is
\[
 \pprob{\tG[T]\simeq h} = p^{\sfv(h)}\Indc\{ G[T] \simeq h\} .
\] 
In view of \prettyref{eq:ind}, this suggests the following unbiased estimator of $ \s(h, G) $:
\begin{equation}
\hat{\s}_h \triangleq 
 \s(h, \tG)/p^{\sfv(h)}.
\label{eq:HT}
\end{equation}
We refer to this estimator as the Horvitz-Thompson (HT) estimator \cite{HT52} since it also uses inverse probability weighting to achieve unbiasedness.
The next theorem gives an upper bound on the mean-squared error for this simple scheme, 
which, somewhat surprisingly, turns out to be minimax optimal within a constant factor as long as the motif $h$ is \emph{connected}.

\begin{theorem}[Subgraph sampling]
\label{thm:subgraph-main}
Let $h$ be an arbitrary connected graph with $k$ vertices.
Let $ G $ be a graph with maximum degree at most $ d $. 
Consider the subgraph sampling model with sampling ratio $p$.
Then the estimator \prettyref{eq:HT} satisfies
\[
\Expect_G|\widehat{\s}_h-\s(h, G)|^2 \leq  \s(h,G) \cdot k 2^k \left(\frac{1}{p^{k}} \vee  \frac{d^{k-1}}{p}\right).
\]
Furthermore,
\[
\inf_{\widetilde\s} \sup_{\substack{G: ~\sfd(G)\leq d\\ ~~~\s(h,G) \leq s}} \Expect_G|\widetilde{\s} -\s(h,G)|^2 = \Theta_{k} \left( \left(\frac{s}{p^{k}} \vee  \frac{s d^{k-1}}{p}\right) \wedge s^2 \right).
\]
\end{theorem}
%
%


The above result establishes the optimality of the HT estimator for classes of graphs with degree constraints. 
Since the lower bound construction actually uses instances of graphs containing many cycles, it is a priori unclear whether additional assumptions such as tree structures can help. Indeed, for the related problem of estimating the number of connected components with subgraph sampling, it has been shown that for parent graphs that are forests the sample complexity is strictly smaller \cite{KlusowskiWu2017-cc}. 
Nevertheless, for counting motifs such as edges or wedges, in \prettyref{thm:subgraph-rates-forest} and \prettyref{thm:wedge-rates-forest} we show that the HT estimator \prettyref{eq:HT} cannot be improved up to constant factors even if the parent graph is known to be a forest.

The proof of the lower bound of \prettyref{thm:subgraph-main} is given in \prettyref{sec:lb-subgraph}. Below we prove the upper bound of the variance:
\begin{proof}	
	Since $\hat\s$ is unbiased, it remains to bound its variance. 
	Let $b_v \triangleq \indc{v\in S}$, which are iid as $\Bern(p)$. For any $T\subset V(G)$, let $b_T \triangleq \prod_{v\in T} b_v$.
	Then 
	\[
	\hat\s = p^{-k} \sum_{T \subset V(G)} b_T \indc{G[T] \simeq h}. 
	\]
	Hence
	\begin{align*}
	\Var[\hat\s]
	= & ~ p^{-2k}\sum_{T \cap T' \neq \emptyset} \Cov(b_T,b_{T'}) \indc{G[T] \simeq h, G[T'] \simeq h}  \\
	\leq & ~ 	p^{-2k} \sum_{T \cap T' \neq \emptyset} \Expect[b_{T\cup T'}] \indc{G[T] \simeq h, G[T'] \simeq h}  \\
	=  & ~ 	\sum_{t=1}^k p^{-t}  \sum_{|T \cap T'|=t} \indc{G[T] \simeq h, G[T'] \simeq h}  \\
	\leq & ~ 	\sum_{t=1}^k p^{-t}  \s(h,G) \binom{k}{t}  d^{k-t} 
			\leq  \s(h,G) (2d)^k \cdot k \max\{(pd)^{-k},(pd)^{-1}\},
	\end{align*}
	where the penultimate step follows from the fact that the maximum degree of $G$ is $d$ and, crucially, $h$ is connected.	
\end{proof}

\subsection{Neighborhood sampling}  \label{sec:neighborhood}

Our methodology is again motivated by \prettyref{eq:ind} which represents neighborhood subgraph counts as a sum of indicators. In contrast to subgraph sampling, a motif can be observed in the sampled graph by sampling only some, but not all, of its vertices. For example, we only need to sample one vertex of an edge, or two vertices of a triangle to observe the full motif. More generally, for a subset $ T $ vertices in $ G $, we can determine whether $ H \simeq G[T] $ or not if at least $ \sfv(H) - 1 $ vertices from $ T $ are sampled. This reduces the variance but introduces more correlation at the same time.

Throughout this subsection, the neighborhood sampled graph is again denoted by $\tG = G\{S\}$, and $b_v =\indc{v\in S}$ indicates whether a given vertex $v$ is sampled.


\subsubsection{Edges}
We begin by discussing the Horvitz-Thompson type of estimator and why it falls short for the neighborhood sampling model.
Analogously to the estimator \prettyref{eq:HT} designed for subgraph sampling, 
for neighborhood sampling, we can take the observed number of edges and re-weight it according to the probability of observing an edge. 
Note that with neighborhood sampling, a given edge is observed if and only if at least one of the end points is sampled. Thus, the corresponding Horvitz-Thompson type edge estimator is
\begin{equation}
\hat\sfe_{\HT} = \frac{\sfe(\tilde G)}{p^2+2pq},
\label{eq:HTn-edge}
\end{equation}
which is again an unbiased estimator for $\sfe(G)$. To bound the variance, put $\tau=p^2+2pq \in [p,2p]$ and write
\[
\sfe(\tilde G) = \sum_{A\in E(G)} r_A.
\]
where $A=\{u,v\}$ and $r_A \triangleq \indc{b_u=1 \text{ or } b_v = 1} \sim \Bern(\tau)$.
For another edge $A'=\{v,w\}$ intersecting $A$, we have $\Cov[r_A,r_{A'}]  = \prob{b_v=1 \text{ or } b_u = b_{w}=1} \leq 3 p$, by the union bound.
Thus the number of non-zero covariance terms is determined by $\n(\Wedge,G)$, the number of $\Wedge$ contained in $G$ as subgraphs, and we have
\begin{align}
\Var[\sfe(\tilde G)]
\leq & ~ \sfe(G) \tau +  2 \n(\Wedge,G) (3 p)	\leq 2 \sfe(G) p (1+3d), 
\label{eq:varHT-n1}
\end{align}
where we used the fact that $ \n(\Wedge,G) \leq \sfe(G) d$. Therefore, the variance of the Horvitz-Thompson estimator satisfies
\begin{equation}
\Var[\hat\sfe_{\HT}] \lesssim \frac{\sfe(G) d}{p}.
\label{eq:varHT-n}
\end{equation}
However, as we show next, this estimator is suboptimal when $p > \frac{1}{d}$,  or equivalently, when the maximum degree exceeds $\frac{1}{p}$.
In fact, the bound \prettyref{eq:varHT-n} itself is tight which can been seen by considering a star graph $G$ with $d$ leaves, and the suboptimality of the HT estimator is largely due to the \emph{heavy correlation} between the observed edges. For example, for the star graph, the correlation is introduced through the root vertex, since with probability $p$ 
we observe a full star, and with probability $q$ a star with $\Binom(d,p)$ number of black leaves.
Thus, the key observation is to incorporate the colors of the vertices to reduce (or eliminate) correlation.


Next, we describe a class of estimators, encompassing and improving the Horvitz-Thompson estimator. Consider
\begin{equation} \label{eq:edge-estimator}
\widehat{\sfe} = \sum_{A \in E(\tG)} \mathcal{K}_{A},
\end{equation}
where $ \mathcal{K}_{A} $ has the form
\begin{equation}
\mathcal{K}_{A} = b_{u}(1-b_{v})f(d_{u}) + b_{v}(1-b_{u})f(d_{v}) + b_{u}b_{v}g(d_{u}, d_{v});
\label{eq:kernel}
\end{equation}
here $ A = \{u, v\} $ and $ f $ and $ g $ are functions of the degree of sampled vertices. For the neighborhood sampling model, this estimator is well-defined since the degree of any sampled vertex is observed without error.
It is easy to see that
\begin{equation}
\Expect[\widehat{\sfe}] = \sum_{\{u, v\} \in E(G)} [pq(f(d_u) + f(d_v)) + p^2g(d_u, d_v)].
\label{eq:emean}
\end{equation}

For simplicity, next we choose $ f $ and $ g $ to be constant; in other words, we do not use the degree information of the sampled vertices. 
This strategy works as long as the maximal degree $d$ of the parent graph is known. To illustrate the main idea, we postpone the discussion on adapting to the unknown $d$ to \prettyref{sec:adaptive}.
With $ f \equiv \alpha $ and $ g \equiv \beta $, the estimator \prettyref{eq:edge-estimator} reduces to
\begin{equation}
\widehat{\sfe} = \alpha \N(\Halfedge,\tG) + \beta \N(\Fulledge,\tG),
\label{eq:edge-estimator1}
\end{equation}
which is a linear combination of the counts of the two types of observed edges.
In contrast to the HT estimator \prettyref{eq:HTn-edge} which treats the two types of edges equally, the optimal choice will weigh them differently. Furthermore, somewhat counter-intuitively, the weights can be negative, which serves to reduce the correlation.

\renewcommand{\arraystretch}{1.8}
\begin{table}[ht]
\centering
\caption{Probability mass function of $\calK_A\calK_{A'}$ for two distinct intersecting edges (excluding zero values).}
\begin{tabular}{c|c|c|c|c}
\hline
\hline 
Graph 
& \tikz[scale=0.75,baseline=(zero.base)]{\draw (0,0) node (zero) [vertexdot] {} -- (0.25,{0.5*sin(60)}) node[vertexdotsolid] {} -- (0.5,0) node[vertexdot] {};} 
& \tikz[scale=0.75,baseline=(zero.base)]{\draw (0,0) node (zero) [vertexdot] {} -- (0.25,{0.5*sin(60)}) node[vertexdotsolid] {} -- (0.5,0) node[vertexdotsolid] {};}
& \tikz[scale=0.75,baseline=(zero.base)]{\draw (0,0) node (zero) [vertexdotsolid] {} -- (0.25,{0.5*sin(60)}) node[vertexdot] {} -- (0.5,0) node[vertexdotsolid] {};}
& \tikz[scale=0.75,baseline=(zero.base)]{\draw (0,0) node (zero) [vertexdotsolid] {} -- (0.25,{0.5*sin(60)}) node[vertexdotsolid] {} -- (0.5,0) node[vertexdotsolid] {};}
\\
\hline
Probability & $ pq^2 $ & $ 2p^2q $ & $ p^2q $ & $ p^3 $\\
\hline
Value & $ \alpha^2 $ & $ \alpha\beta $ & $ \alpha^2 $ & $ \beta^2 $\\
\hline
\hline
\end{tabular}
\label{tab:edge-cov}
\end{table}

In view of \prettyref{eq:emean}, one way of making $ \widehat{\sfe} $ unbiased is to set
\begin{equation} \label{eq:unbiased-edge}
pq(f(d_u) + f(d_v)) + p^2g(d_u, d_v) = 2pq\alpha + p^2\beta = 1.
\end{equation}
Since the unbiased estimator is not unique, we set out to find the one with the minimum variance.
Similar to \prettyref{eq:varHT-n1}, we have
\begin{align}
\Var[\widehat{\sfe}] & = \sfe(G)\Var[\mathcal{K}_{A}] + 2\n(\Wedge, G)\Cov[\mathcal{K}_{A},  \mathcal{K}_{A'}] 
\leq \sfe(G) (\Var[\mathcal{K}_{A}] + 2d \Cov[\mathcal{K}_{A},  \mathcal{K}_{A'}]),
\label{eq:var-edge}
\end{align}
where $ A = \{u, v\} $ and $ A' = \{v, w\} $ are distinct intersecting edges. Using \prettyref{tab:edge-cov}, we find
\begin{align*}
\Var[\mathcal{K}_{A}] & = 2pq\alpha^2+p^2\beta^2-1 \\
\Cov[\mathcal{K}_{A},  \mathcal{K}_{A'}] & = \alpha^2(pq^2+p^2q) + p^3 \beta^2 + 2 p^2q \alpha \beta - 1.
\end{align*}
In fact, when the unbiased condition \prettyref{eq:unbiased-edge} is met, the covariance simplifies to $ \Cov[\mathcal{K}_{A},  \mathcal{K}_{A'}]  = \frac{q}{p}(1-p\alpha)^2 \geq 0 $.
Finally, optimizing the RHS of \prettyref{eq:var-edge} over $\alpha,\beta$ subject to the constraint \prettyref{eq:unbiased-edge}, we arrive the following 
performance guarantee for $ \widehat{\sfe} $:
\begin{theorem} \label{thm:edge}
	Set 
	\begin{equation}
	\alpha =  \frac{1+dp}{p(2+(d-1)p)} \qquad \beta =  \frac{1-d(1-2p)}{p(2+(d-1)p)}.
	\label{eq:alphabetaopt}
	\end{equation}
Then
\begin{equation} \label{eq:edge-var}
\Expect[(\widehat{\sfe}-\sfe(G))^2] = 
\Var[\widehat{\sfe}] 
\leq \frac{\sfe(G) (d+1)q^2}{p(2+(d-1)p)} 
\lesssim \sfe(G)\left(\frac{1}{p^2} \wedge \frac{d}{p}\right).
\end{equation}
Furthermore, if $p$ is bounded from one, then
\[
\inf_{\widehat{\sfe}} \sup_{\substack{G: ~\sfd(G)\leq d\\ ~~~\sfe(h,G) \leq m}} 
\Expect_G|\widehat{\sfe}-\sfe(G)|^2 = \Theta \left( \left(\frac{m d}{p} \wedge  \frac{m}{p^2}\right) \wedge m^2 \right)
\]
\end{theorem}

The optimal weights in \prettyref{eq:alphabetaopt} appear somewhat mysterious. In fact, the following more transparent choice also achieves the optimal risk within constant factors:
	\begin{itemize}
	\item $p \leq 1/d$: we can set either
	$\alpha=  \beta = \frac{1}{p^2+2pq}$ or
	$\alpha=  \frac{1}{2pq}$ and $\beta = 0$, that is, we can use either the full HT estimator \prettyref{eq:HTn-edge}, or the HT estimator restricted to only edges of type $\Halfedge$, which is the more probable one.


		\item $p > 1/d$: we choose $\alpha =  \frac{1}{p}$ and $\beta = \frac{1-2q}{p^2}$. This is the unique weights that simultaneously kill all covariance terms and, at the same time, achieve zero bias.		Note that although zero covariance is always possible, it is at a price of setting $\beta \approx -\frac{1}{p^2}$, which inflates the variance too much when $p$ is small and hence suboptimal when $p \ll \frac{1}{d}$.
	\end{itemize}



It is a priori unclear whether additional structures such as tree or planarity helps for estimating motif counts with neighborhood sampling. 
Nevertheless, for counting edges, in \prettyref{thm:edge-forest} we show that the Horvitz-Thompson estimator \prettyref{eq:HT} can only be marginally improved, in the sense that the lower bound continues to hold up to a sub-polynomial factor $p^{o(1)}$ where $o(1)$ is uniformly vanishing as $p\to 0$.
Similarly, for planar graphs, \prettyref{thm:triangle-planar} shows a similar statement.


\subsubsection{Cliques and general motifs}

For ease of exposition, we start by developing the methodology for estimating cliques counts. Both the procedure and the performance guarantee readily extend to general motifs.

We now generalize the techniques for counting edges to estimate the number of cliques of size $\omega$ in a given graph.
Note that there are two types of colored cliques one observe:
(a) $\Kempty_\omega$: 
all but one vertex are sampled; (b) $\Kfull_\omega$: all vertices are sampled, with the first one being more probable when the sampling ratio is small.
In the case of triangles, we have 
$\Kempty_3=$
\tikz[scale=0.75,baseline=(zero.base)]{\draw (0,0) node (zero) [vertexdotsolid] {} -- (0.25,{0.5*sin(60)}) node[vertexdot] {} -- (0.5,0) node[vertexdotsolid] {} -- cycle-- cycle;}
and 
$\Kfull_3=$
\tikz[scale=0.75,baseline=(zero.base)]{\draw (0,0) node (zero) [vertexdotsolid] {} -- (0.25,{0.5*sin(60)}) node[vertexdotsolid] {} -- (0.5,0) node[vertexdotsolid] {} -- cycle;}.
Analogous to the estimator \prettyref{eq:edge-estimator1}, 
we take a linear combination of these two types of clique counts as the linear estimator:
\begin{equation}
\hat\s = \alpha \sfN(\Kempty_\omega,\tG) + \beta \sfN(\Kfull_\omega,\tG).
\label{eq:cliqueest}
\end{equation}
Similar to the design principles for counting edges, in the low sampling ratio regime $p < \frac{1}{d}$, we implement the Horvitz-Thompson estimator, so that the coefficients scale like $p^{-\omega}$; in the high sampling ratio regime $p > \frac{1}{d}$, we choose a \emph{negative} $\beta$, which scale as $-p^{-2\omega}$, to reduce the correlation between various observed cliques. However, unlike the case of counting edges, we cannot perfectly eliminate all covariance terms but will be able to remove the leading one.

The following result, which includes \prettyref{thm:edge} as a special case ($\omega=2$), gives the performance guarantee of the estimator \prettyref{eq:cliqueest} and establishes its optimality in the worst case:
\begin{theorem}[Cliques]
\label{thm:neighborhood_clique}
Set
\begin{align}
\begin{cases}
  \alpha =  \frac{1}{p^{\omega-1}}, \quad \beta = \frac{1-\omega q}{p^{\omega}}  & \quad \text{if } p > 1/d \\
  \alpha=  \frac{1}{\omega p^{\omega-1}}, \quad \beta = \frac{1}{p^{\omega}}  & \quad \text{if } p \leq 1/d.
  \end{cases}
\label{eq:alphabeta-general}
\end{align}
Then
\[
\Expect_G|\widehat{\s}-\s(K_{\omega}, G)|^2 = 
\Var_G[\widehat{\s}]
\leq \s(K_{\omega}, G) \cdot \omega^3 2^{\omega+1} \left(  \frac{ \sfd(G)}{p^{\omega-1}} \wedge  \frac{\sfd(G)^{\omega-2}}{p^2} \right).
\]
Furthermore,
\[
\inf_{\widetilde\s} \sup_{\substack{G: ~\sfd(G)\leq d\\ ~~~\s(K_\omega,G) \leq s}} 
\Expect_G|\widehat{\s}-\s(K_{\omega}, G)|^2 = \Theta_{\omega} \left( \frac{s d}{p^{\omega-1}} \wedge  \frac{s d^{\omega-2}}{p^2} \wedge s^2 \right)
\]
\end{theorem}

\begin{proof}
	Let $b_v \triangleq \indc{v\in S} \iiddistr \Bern(p)$. 
	For any $T\subset V(G)$, let $b_T \triangleq \prod_{v\in T} b_v$.
	Write
	\begin{align}
	\hat \s
	= & ~  \sum_{T \subset V(G)} \alpha \indc{\tilde G\{T\} \simeq \Kempty_\omega} + \beta \indc{\tilde G\{T\} \simeq \Kfull_\omega} \nonumber \\
	= & ~  \sum_{T \subset V(G)} f(T)  \indc{G[T] \simeq K_\omega}, \label{eq:neighborhood-est}
	\end{align}
	where 
	\begin{align*}
	f(T) 	\triangleq \alpha \sum_{v\in T} b_{T\backslash\{v\}} (1-b_v) + \beta b_T.
	\end{align*}
	Similar to \prettyref{eq:unbiased-edge}, enforcing unbiasedness, we have the constraint $\Expect[f(T)]=1$, i.e.,
	\begin{equation}
	\omega p^{\omega-1}q \alpha + p^{\omega} \beta = 1 
	\label{eq:unbiased-nbhd}
	\end{equation}
	Furthermore, whenever $|T\cap T'| = t \in [\omega]$, we have
	\begin{align}
	\Expect[f(T)f(T')]
	= & ~ \alpha^2 \pth{t q p^{2\omega-t-1} + (\omega-t)^2 q^2p^{2\omega-t-2}} + 2(\omega-t) q p^{2\omega-t-1}\alpha\beta + \beta^2p^{2\omega-t}	\label{eq:fcross0}	\\
	= & ~ p^{-t} \qth{\alpha^2 t q p^{2\omega-1} +  \pth{ \alpha (\omega-t) q p^{\omega-1}  + \beta p^{\omega}}^2	} \nonumber \\
	\overset{\prettyref{eq:unbiased-nbhd}}{=} & ~ p^{-t} \qth{\alpha^2 t q p^{2\omega-1} +  \pth{ 1 - t q \alpha p^{\omega-1}}^2	}
	\label{eq:fcross}
	\end{align}
This follows from evaluating the probability of observing a pair of intersecting cliques with two, one, or zero unsampled vertices. For example, the four summands in \prettyref{eq:fcross0}, in the case of $\omega=4$ and $t=2$, correspond to 
\tikz[scale=0.75,baseline=(zero.base)]{\draw (0,0)--(0.5,0.5); \draw (0.5,0)--(0,0.5); \draw (0.5,0)--(1,0.5); \draw (1,0)--(0.5,0.5); \draw (0.5,0)--(0.5,0.5); 
\draw (0,0) node (zero) [vertexdotsolid] {} -- (0.5,0) node[vertexdotsolid] {} -- (1,0) node[vertexdotsolid] {} -- (1,0.5) node[vertexdotsolid] {} -- (0.5,0.5) node[vertexdot] {} -- (0,0.5) node[vertexdotsolid] {} -- cycle; },
\tikz[scale=0.75,baseline=(zero.base)]{\draw (0,0)--(0.5,0.5); \draw (0.5,0)--(0,0.5); \draw (0.5,0)--(1,0.5); \draw (1,0)--(0.5,0.5); \draw (0.5,0)--(0.5,0.5); 
\draw (0,0) node (zero) [vertexdot] {} -- (0.5,0) node[vertexdotsolid] {} -- (1,0) node[vertexdotsolid] {} -- (1,0.5) node[vertexdot] {} -- (0.5,0.5) node[vertexdotsolid] {} -- (0,0.5) node[vertexdotsolid] {} -- cycle; },
\tikz[scale=0.75,baseline=(zero.base)]{\draw (0,0)--(0.5,0.5); \draw (0.5,0)--(0,0.5); \draw (0.5,0)--(1,0.5); \draw (1,0)--(0.5,0.5); \draw (0.5,0)--(0.5,0.5); 
\draw (0,0) node (zero) [vertexdotsolid] {} -- (0.5,0) node[vertexdotsolid] {} -- (1,0) node[vertexdotsolid] {} -- (1,0.5) node[vertexdot] {} -- (0.5,0.5) node[vertexdotsolid] {} -- (0,0.5) node[vertexdotsolid] {} -- cycle; },
\tikz[scale=0.75,baseline=(zero.base)]{\draw (0,0)--(0.5,0.5); \draw (0.5,0)--(0,0.5); \draw (0.5,0)--(1,0.5); \draw (1,0)--(0.5,0.5); \draw (0.5,0)--(0.5,0.5); 
\draw (0,0) node (zero) [vertexdotsolid] {} -- (0.5,0) node[vertexdotsolid] {} -- (1,0) node[vertexdotsolid] {} -- (1,0.5) node[vertexdotsolid] {} -- (0.5,0.5) node[vertexdotsolid] {} -- (0,0.5) node[vertexdotsolid] {} -- cycle; },
 respectively.

Let $c_t \triangleq \Cov[f(T),f(T')] = p^{-t} \qth{\alpha^2 t q p^{2\omega-1} +  \pth{ \alpha (\omega-t) q p^{\omega-1}  + \beta p^{\omega}}^2	}-1$ for $|T\cap T'|=t$.
	Denote by $T_{\omega,t}$ the subgraph correspond to two intersecting $\omega$-cliques sharing $t$ vertices. Then
	\begin{align}
	\Var[\hat\s]
	= & ~ \sum_{T \cap T' \neq \emptyset} \Cov(f(T),f(T')) \indc{G[T] \simeq K_\omega, G[T'] \simeq K_\omega}  \label{eq:var-nbhd}\\
		=  & ~ 	\sum_{t=1}^{\omega} c_t \sum_{|T \cap T'|=t} \sfn(T_{\omega,t}, G)
	\leq \s(K_\omega,G) d^\omega \sum_{t=1}^{\omega} c_t \binom{\omega}{t}  d^{-t}. \nonumber
	\end{align}
Next consider two cases separately.

\paragraph{Case I: $p \leq \frac{1}{d}$.}
In this case we choose $\alpha = \frac{1}{\omega p^{\omega-1}}$ and $\beta = \frac{1}{p^{\omega}}$.
Then $c_t =  p^{-t} ( \frac{tpq}{\omega^2} +  (1 - \frac{t q }{\omega})^2	) \leq 2p^{-t}$. Furthermore, for the special case of $t=\omega$, we have $c_\omega \leq p^{-(\omega-1)}$. Thus,
\begin{equation}
\Var[\hat\s] \leq \s(K_\omega,G) \pth{ d^\omega \sum_{t=1}^{\omega} \binom{\omega}{t}  (pd)^{-t} + p^{-(\omega-1)}} 
\leq \s(K_\omega,G) \omega 2^{\omega+1} d p^{-(\omega-1)}.
\label{eq:cliquevar1}
\end{equation}

\paragraph{Case II: $p \leq \frac{1}{d}$.}
In this high-degree regime, the pairs of cliques sharing one vertex ($t=1$) dominates (i.e., open triangle for counting edge and bowties for counting triangles). Thus our strategy is to choose the coefficients to eliminate the these covariance terms. In fact, \prettyref{eq:fcross} for $t=1$ simplifies wonderfully to
\[
c_1 = \frac{q}{p}(1-\alpha p^{\omega-1})^2.
\]
Thus we choose $\alpha =  \frac{1}{ p^{\omega-1}}$ and $\beta = \frac{1-\omega q}{p^{\omega}}$. Hence $c_t \leq 2 \omega^2 p^{-t}$ for all $t \geq 2$, and
\begin{equation}
\Var[\hat\s] \leq \s(K_\omega,G) d^\omega 2 \omega^2 \sum_{t=2}^{\omega} \binom{\omega}{t}  (pd)^{-t} 
\leq \s(K_\omega,G)  2^{\omega+1} \omega^3 d^{\omega-2} p^{-2}.
\label{eq:cliquevar2}
\end{equation}
Combining \prettyref{eq:cliquevar1} and \prettyref{eq:cliquevar2} completes the proof.
\end{proof}

To extend to general motif $h$ on $k$ vertices, note that in the neighborhood sampled graph, again it is possible to observe fully sampled or partially sampled (with one unsampled vertices) motifs.
Consider the following estimator analogous to \prettyref{eq:cliqueest}:
\begin{equation}
\hat\s_h = \alpha \sfN(h^{\circ},\tG) + \beta \sfN(h^{\bullet},\tG),
\label{eq:hest}
\end{equation}
where 
$\sfN(h^{\circ},\tG)$ is the count of $h$ with all vertices sampled and $\sfN(h^{\circ},\tG)$ is the total count of $h$ with exactly one unsampled vertices. 
For instance, if $h=\Diamond$, then 
$\sfN(h^{\bullet},\tG) = \N(\tikz[scale=0.75,baseline=(zero.base)]{\draw (0.5,0)--(0,0.5); \draw (0,0) node (zero) [vertexdotsolid] {} -- (0.5,0) node[vertexdotsolid] {} -- (0.5,0.5) node[vertexdotsolid] {} -- (0,0.5) node[vertexdotsolid] {} -- cycle; }, G)$
and
$\sfN(h^{\circ},\tG) = \N(\tikz[scale=0.75,baseline=(zero.base)]{\draw (0.5,0)--(0,0.5); \draw (0,0) node (zero) [vertexdotsolid] {} -- (0.5,0) node[vertexdotsolid] {} -- (0.5,0.5) node[vertexdotsolid] {} -- (0,0.5) node[vertexdot] {} -- cycle; }, G) + \N(\tikz[scale=0.75,baseline=(zero.base)]{\draw (0.5,0)--(0,0.5); \draw (0,0) node (zero) [vertexdotsolid] {} -- (0.5,0) node[vertexdotsolid] {} -- (0.5,0.5) node[vertexdot] {} -- (0,0.5) node[vertexdotsolid] {} -- cycle; }, G)$.
This example shows that in general, for motifs with less symmetry, there exist multiple partially sampled motifs and in principle they can be weighted differently. However, in \prettyref{eq:hest} we elect to treat them equally, which turns out to be optimal for a wide class of motifs.
Let us point out that if the parent graph has more structures, e.g., forest, then distinguishing different partially sampled motifs can lead to strict improvement; see \prettyref{thm:wedge-forest}.

The estimator \prettyref{eq:hest} turns out to satisfy the same bound as in the clique case. To see this, note that in \prettyref{eq:var-nbhd}, the covariance terms are given in \prettyref{eq:fcross} which do not depend on the actual motif $h$. Furthermore, the sum of the indicators satisfies the same bound in terms of maximal degree provided that $h$ is connected. Using the same optimized coefficients as in \prettyref{eq:alphabeta-general}, the guarantee in \prettyref{thm:neighborhood_clique} holds verbatim:
\begin{equation} \label{eq:var-general-graph}
\Expect_G|\widehat{\s}_h-\s(h, G)|^2 = 
\Var_G[\widehat{\s}_h]
\leq \s(h, G) \cdot k^3 2^{k+1} \left(  \frac{ \sfd(G)}{p^{k-1}} \wedge  \frac{\sfd(G)^{k-2}}{p^2} \right).
\end{equation}
We conjecture that, similar to \prettyref{thm:subgraph-main}, this rate is optimal as long as the motif $h$ is connected. So far we are able to prove this for cliques of all sizes (\prettyref{thm:neighborhood_main}) and motifs on at most 4 vertices (\prettyref{app:othermotif}).


\subsubsection{Adaptation to the maximum degree} \label{sec:adaptive}
In practice, the bound on the maximum degree $ d $ is likely unknown to the observer and obtaining a consistent estimate might be difficult if the high-degree vertices are rare. For example, in a star, most of the vertices have degree one expect for the root.
Even if a consistent estimate is obtained, it is unclear how to avoid it correlating with the data used to form $ \widehat{\sfe} $. Because $ \widehat{\sfe} $ has the form of a sum, such correlations increase the number of cross terms in its variance decomposition.

To overcome these difficulties, we weight each observed edge according to the size of the neighborhood of its incident vertices. Once a vertex is sampled, its degree is exactly determined and thus incorporating this information does not introduce any additional randomness. This observations leads to the following adaptive estimator which achieves a risk that is similar to the optimal risk in \prettyref{thm:edge}:
\begin{theorem} \label{thm:adaptive}
Let $\ehat$ be given in \prettyref{eq:edge-estimator} with $ f(x) = \frac{px+q}{p(px+2q)} $ and $ g(x,y) = \frac{1-pq(f(x) + f(y))}{p^2} $. Then for any graph $ G $ on $ N $ vertices and maximum degree bounded by $ d $, $ \widehat{\sfe} $ is an unbiased estimator of $ \sfe(G) $ and
\begin{equation*}
\Var[\widehat{\sfe}] \lesssim \frac{Nd}{p^2} \wedge \frac{\sfe(G)d}{p}.
\end{equation*}
\end{theorem}

\begin{remark}
The variance bound from \prettyref{thm:adaptive} is weaker than \prettyref{thm:edge} in the $ p > 1/d $ regime -- $ \frac{Nd}{p^2} $ versus $ \frac{\sfe(G)}{p^2} $. They have the largest disparity when $ G $ consists of $ N/(d+1) $ copies of the star graph $ S_{d+1} $, in which case $ \sfe(G) = Nd/(d+1) $. This is due to the fact that with high probability $ 1-p $, all sampled vertices from $ S_{d+1} $ have degree one. Ideally, we would like to know the degree of the root of the star; however this is impossible unless the root is sampled. Nonetheless, we can still find a good estimate. More generally, in addition to using the degree $ d_u $ from a sampled vertex $ u $, we may modify the estimator to incorporate degree information from a non-sampled vertex via an unbiased estimate, i.e., $ \widehat{d}_u = \frac{|N_{\tG}(u)|}{p} = \sum_{v\in N_G(u)} \frac{b_v}{p} $. For example, we can redefine $ \mathcal{K}_{A} $ from \prettyref{eq:edge-estimator} as
\begin{equation*}
\mathcal{K}_{A} = b_{u}(1-b_{v})f(d_u \vee \widehat{d}_v) + b_{v}(1-b_{u})f(d_u \vee \widehat{d}_v) + b_{u}b_{v}g(d_{u}, d_{v}).
\end{equation*}
\end{remark}


\section{Lower bounds} \label{sec:lower}
Throughout this section we assume that the sampling ratio $p$ is bounded away from one.


\subsection{Auxiliary results}
	\label{sec:aux}

We start with a result which is the general strategy of proving all lower bounds in this paper. A variant of this result was proved in \cite{KlusowskiWu2017-cc} for the Bernoulli sampling model, however, an examination of the proof reveals that the conclusions also hold for neighborhood sampling. In the context of estimating motif counts, the essential ingredients involve constructing a pair of random graphs whose motif counts have different average values, and the distributions of their sampled versions are close in total variation, which is ensured by matching lower-order subgraphs counts in terms of $ \s $ for subgraph sampling or $ \N $ for neighborhood sampling. The utility of this result is to use a pair of smaller graphs (which can be found in an ad hoc manner) to construct a bigger pair of graphs and produce a lower bound that scales with an arbitrary positive integer $ s $.


\begin{theorem}[Theorem 11 in \cite{KlusowskiWu2017-cc}] \label{thm:mainlb}
Let $ f $ be a graph parameter that is invariant under isomorphisms and additive under disjoint union, i.e., $f(G+H)=f(G)+f(H)$. Fix a subgraph $ h $.
Let $d,s, m $ and $ M = s/m $ be integers. Let $ H $ and $ H' $ be two graphs such that $ \s(h, H) \vee \s(h, H') \leq m $ and $\sfd(H)\vee \sfd(H') \leq d$.
Suppose $ M \geq 300 $ and $ \TV(P,P') \leq 1/300 $,
where $P$ (resp.~$P'$) denote the distribution of the isomorphism class of the (subgraph or neighborhood) sampled graph $\tH$ (resp.~$\tH'$).
Let $\tG$ denote the sampled version of $G$ under the Bernoulli or neighborhood sampling models with probability $p$.
Then
\begin{equation}
\inf_{\widehat{f}} \sup_{\substack{G: ~\sfd(G)\leq d\\ ~~~\s(h,G) \leq s}} \Prob_G\Big[|\widehat{f}\big(\tG\big)-f(G)| \geq \Delta \Big] \geq 0.01.
\label{eq:mainlb}
\end{equation}
where
\begin{equation*}
\Delta = \frac{|f(H)-f(H')|}{8}\left({\sqrt{\frac{s}{m\TV(P,P')}}}  \wedge \frac{s}{m}\right).
\end{equation*}
\end{theorem}

Next we recall the well-known fact \cite{Whitney1932,Kocay1982} that disconnected subgraphs counts are determined by (fixed polynomials of) connected subgraph counts. The following version is 
from \cite[Corollary 1 and Lemma 9]{KlusowskiWu2017-cc}:
\begin{lemma} \label{lmm:counts}
Let $ H $ and $ H' $ be two graphs with $m$ vertices and $v \leq m$. 
Suppose $ \s(h, H) = \s(h, H') $ for all connected $ h $ with $ \sfv(h) \leq v $. Then $ \s(h, H) = \s(h, H') $ for all $ h $ with $ \sfv(h) \leq v $ and, furthermore,
\[
\TV(P,P') \leq \binom{m}{v+1} p^{v+1},
 \]
where $P$ (resp.~$P'$) denote the distribution of the isomorphism class of the subgraph sampled graph $\tH$ (resp.~$\tH'$) with sampling ratio $p$.
\end{lemma}

The following version is for neighborhood sampling, which will be used in the proof of \prettyref{thm:neighborhood_main}. We need to develop an analogous result that expresses disconnected neighborhood subgraph counts as polynomials of the connected cones. This is done in \prettyref{lmm:nbhd} in \prettyref{app:appendix}.

\begin{lemma} \label{lmm:counts}
Let $ H $ and $ H' $ be two graphs with $m$ vertices and $v \leq m$. 
Suppose $ \N(h, H) = \N(h, H') $ for all connected, bicolored $ h $ with $ \sfv_b(h) \leq v $. Then 
\begin{equation} \label{eq:match}
\N(h, H) = \N(h, H')
\end{equation}
for all $ h $ with $ \sfv_b(h) \leq v $ and, furthermore,
\begin{equation} \label{eq:tv-bound}
\TV(P,P') \leq \binom{m}{v+1} p^{v+1},
\end{equation}
where $P$ (resp.~$P'$) denote the distribution of the isomorphism class of the sampled graph $\tH$ (resp.~$\tH'$) generated from neighborhood sampling with sampling ratio $p$.
\end{lemma}

\begin{proof}
The first conclusion \prettyref{eq:match} follows from \prettyref{lmm:nbhd}. For the second conclusion \prettyref{eq:tv-bound}, we note that conditioned on $ \ell $ vertices are sampled, $ \tH $ is uniformly distributed over the collection of all bicolored neighborhood subgraphs $ h $ with $ \sfv_b(h) = \ell $. Thus,
\begin{equation*}
\prob{\tH \cong h \mid \sfv_b(h) = \ell} = \frac{\N(h, H)}{\binom{m}{\ell}}.
\end{equation*}
By \prettyref{eq:match}, we conclude that the isomorphism class of $ \tH $ and $ \tH' $ have the same distribution provided that no more than $ v $ vertices are sampled. Thus, $ \TV(P_{\tH},P_{\tH'}) \leq \prob{\Binom(m, p) \leq v+1} $, and consequently, $ \prob{\Binom(m, p) \leq v+1} \leq \binom{m}{v+1}p^{v+1} $ follows from a union bound.
\end{proof}

\begin{lemma}
\label{lmm:subgraphmatching}	
	For any connected graph $h$ with $k$ vertices, there exists a pair of (in fact, connected) graphs $H$ and $H'$, such that 
$\s(h,H) \neq \s(h,H')$ and $\s(g,H)=\s(g,H')$ for all connected $g$ with $\sfv(g) \leq k-1$.	
\end{lemma}
\begin{proof}
	The existence of such a pair $H$ and $H'$ follows from the strong independence\footnote{This means that the closure of the range of their normalized version (subgraph densities) has nonempty interior.} of connected subgraph counts \cite[Theorem 1]{Erdos1979}.
For example, for $h=\Triangle$, we can take the ab hoc construction in \prettyref{eq:adhocHH1},
 which have equal number of vertices and edges but distinct number of triangles.
Alternatively, next we provide an explicit construction using a linear algebra argument which is similar to that of \cite[Theorem 3]{Erdos1979} and \cite[Section 2]{Biggs1978}. 
Let $\{h_1,\ldots,h_m\}$ denote all distinct (up to isomorphism) induced \emph{connected} subgraph of $h$, ordered in increasing number of edges (arbitrarily among graphs with the same number of edges) so that $h_1$ is an isolated vertex and $h_m=h$. Then the matrix $B=(b_{ij})$ with $b_{ij} = \s(h_i,h_j)$ is upper triangular with strictly positive diagonals. Thus $B$ is invertible and the entries of $B^{-1}$ are rational. Let $x = B^{-1} e_m$, where $e_m=(0,\ldots,0,1)$. Then $x_m = 1$ since $b_{mm}=\s(h,h)=1$. Let 
$w = \alpha x \in\integers^m$, where $\alpha \in \naturals$ is the lowest common denominator of the entries of $x$.
Now define $H$ and $H'$ as the disjoint union with weights given by the vector $w$:
\begin{equation}
H = \sum_{i=1}^m \max\{w_i,0\} h_i, \quad H' = \sum_{i=1}^m \max\{-w_i,0\} h_i.
\label{eq:HH-matching}
\end{equation}
By design, any connected induced subgraph of $H$ and $H'$ with at most $k-1$ vertices belongs to $\{h_1,\ldots,h_{m-1}\}$.
For any $1\leq i\leq m-1$, since $h_i$ is connected, we have $\s(h_i,H) - \s(h_i,H') = \sum_{j=1}^m w_j \s(h_i,h_j) = 0$, and $\s(h,H')=0$ and $\s(h,H) = \alpha \geq 1$.
For example, for $h=\Kfour$, such a solution is given by:
\begin{align*}
H = \Kfour + 6 \times \Edge \qquad H' = 4 \times \Triangle + 4 \times \Vertex
\end{align*}
which have matching subgraphs of order three (vertices, edges and triangles). For a construction for general cliques, see \cite[Eq.~(47)]{KlusowskiWu2017-cc}.
\end{proof}

Next we present graph-theoretic results that are needed for proving lower bound under the neighborhood sampling model. First, we relate the neighborhood subgraph counts $\N$ to the usual subgraph $\n$. Since $\N$ is essentially subgraph counts with prescribed degree for the sampled vertices (cf.~\cite[p.~62]{Lovasz12}), this can be done by inclusion-exclusion principle similar to \prettyref{eq:injind} that expresses the induced subgraph counts $\s$ in terms of the subgraph counts $\n$; however, the key difference here is that the size of the subgraphs that appear in the linear combination is not bounded a priori. For example,
\begin{align*}
\N(\tikz[scale=0.75,baseline=(zero.base)]{\draw (0,0) node (zero) [vertexdot] {} -- (0.25,{0.5*sin(60)}) node[vertexdotsolid] {} -- (0.5,0) node[vertexdot] {};},G)
= & ~  \text{number of degree-$2$ vertices in $G$}\\
= & ~ 	\sum_{k\geq 2} (-1)^{k-2}  \binom{k}{2} \n(S_{k+1},G),
\end{align*}
where $S_{k+1}$ is the star graph with $k$ leaves. 
The following lemma is a general statement:

\begin{lemma}
\label{lmm:incexc}	
	Let $h$ be a bicolored connected neighborhood graph and $h_0$ denote the uncolored version. Then for any  $G$, 
	\begin{equation}
	\N(h,G) = \sum_g c(g,h) \n(g,G)
	\label{eq:Ninj}
	\end{equation}
	where the sum is over all (uncolored) $g$ obtained from $h$ by either adding edges incident to the black vertices in $h$ or adding vertices connected to black vertices in $h$.
	In particular, the coefficients  $c(g,h)$ do not depend on $G$.
\end{lemma}
\begin{proof}
	The proof is by the inclusion-exclusion principle and essentially similar to the argument in Section 5.2, in particular, the proof of Proposition 5.6(b) in \cite{Lovasz12}. 
	
	Recall the definition of the subgraph count $\n(H,G)$ in \prettyref{eq:inj} in terms of counting distinct subsets. It will be convenient to work with the labeled version counting graph homomorphisms. The following definitions are largely from \cite[Chapter 5]{Lovasz12}. We say $\psi$ is an injective homomorphism from $H$ to $G$, if $\psi: V(H) \to V(G)$ is injective, and $(u,v)\in E(H)$ if $(\psi(u),\psi(v)) \in E(G) $.
		Denote by $\inj(H, G) $ the number of injective homomorphisms from $H$ to $G$.		Then 
$\inj(H,G)=\n(H,G) \aut(H)$, where $\aut(H)$ denotes the number of automorphisms (i.e.~isomorphisms to itself) for $H$.
Furthermore, for neighborhood subgraph $h$, $\aut(H)$ denotes the number of automorphisms for $h$ that also preserve the colors.
For example, $\aut(\tikz[scale=0.75,baseline=(zero.base)]{\draw (0,0) node (zero) [vertexdotsolid] {} -- (0.5,0) node[vertexdot] {} -- (0.5,0.5) node[vertexdotsolid] {} -- (0,0.5) node[vertexdot] {} -- cycle; })=2$ and $\aut(\Square)=4$.
Throughout the proof, $\psi$ always denotes an injection.

	We use the following version of the inclusion-exclusion principle \cite[Appendix A.1]{Lovasz12}. Let $S$ be a ground set and let $\{A_i: i\in S\}$ be a collection of sets. For each $I\subset S$, define $A_I \triangleq \cap_{i\in I} A_i$ and $B_I \triangleq A_I \backslash \cup_{i\notin I} A_i$; in words, $B_I$ denotes those elements that belong to exactly those $A_i$ for $i\in I$ and none other.
	Then we have
	\begin{align}
|A_I|	= & ~ \sum_{J \subset I} |B_J| \\
|B_I|	= & ~ \sum_{J \subset I} (-1)^{|J|-|I|}|A_J|. \label{eq:incexc}
	\end{align}
	
	Fix $G$. 
	Let $\calG$ denote the collection of (uncolored) subgraphs that are ``extensions'' of $h$, obtained from $h$ by either adding edges between the black vertices in $h$ or adding vertices attached to black vertices in $h$.	
For example, for $h=\Halfedge$, we have $\calG=\{\Edge, \tikz[scale=0.75]{\draw  (0.5,0) node (zero) [vertexdot] {} -- (1.0,0) node (zero) [vertexdot] {};}, 
\tikz[scale=0.75]{\draw  (0.5,0) node (zero) [vertexdot] {} -- (1.0,0) node (zero) [vertexdot] {}; \draw (0.5,0) node (zero) [vertexdot] {} -- (1.0,0.4) node (zero) [vertexdot] {};},
\tikz[scale=0.75]{\draw  (0.5,0) node (zero) [vertexdot] {} -- (1.0,0) node (zero) [vertexdot] {}; \draw (0.5,0) node (zero) [vertexdot] {} -- (1.0,0.4) node (zero) [vertexdot] {}; \draw (0.5,0) node (zero) [vertexdot] {} -- (1.0,0.8) node (zero) [vertexdot] {};}, \cdots\}$ is the collection of all stars.
	Let the $g^*$ be the maximal subgraph of $G$ that is in $\calG$; in other words, $\n(g,G)=0$, for any other $g\in \calG$ containing $g^*$ as a subgraph.
	
	Now we define the ground set to be the edge set of $g^*$.
	Let $h_0$ be the uncolored version of $h$, then $E(h_0)\subset E(g^*)$. 	
	For every $I \subset E(g^*)$, define 
	$A_I \triangleq \{\psi: V(g^*) \to V(g) \colon (\psi(u),\psi(v))\in E(G)  \text { if }  (u,v) \in I\}$ and 
	$B_I \triangleq \{\psi: V(g^*) \to V(g) \colon (\psi(u),\psi(v))\in E(G)  \text { if and only if }  (u,v) \in I\}$.
	The key observation is that 
	$|B_{E(h_0)}|   = \aut(h) \N(h,G)$, and 
	$|A_{E(g)}| = \inj(g,G) = \aut(g) \n(h,G)$.
	Applying the inclusion-exclusion principle \prettyref{eq:incexc} yields 
	\[
	\aut(h) \N(h,G)	= \sum_{g: g \supset h_0} (-1)^{|E(g)|-|E(h_0)|} \inj(g,G) .
	\]	
	proving the desired \prettyref{eq:Ninj}.
\end{proof}

The next result is the counterpart of \prettyref{lmm:subgraphmatching}, which shows the existence of a pair of graphs with matching lower order neighborhood subgraph counts but contain distinct number of copies of a certain motif; however, unlike \prettyref{lmm:subgraphmatching}, so far we can only deal with the clique motifs.
For example for $\omega=3$, we can use the ad hoc construction in \prettyref{eq:adhocHH1}; both graphs have the same degree sequence but distinct number of triangles.
For $\omega=4$, we can choose
\begin{equation}
\begin{aligned}
H = & ~ \Kfour + 3 \times \Square + 12 \times \Paw + 12 \times \PathThree \\
H' = & ~ 6 \times \Diamond + 12 \times \PathFour + 4 \times \Triangle + 	4 \times \Claw
\end{aligned}
\label{eq:HH-nhbd4}
\end{equation}
It it straightforward (although extremely tedious!) to verify that $ \N(h, H) = \N(h, H') $ for all neighborhood subgraphs $ h $ with at most 2 black vertices.
The general result is as follows:

\begin{lemma} \label{lmm:general_clique_neighborhood}
There exists two graphs $ H $ and $ H' $ such that $ \s(K_{\omega}, H) - \s(K_{\omega}, H') \geq 1 $ and $ \N(h, H) = \N(h, H') $ for all neighborhood subgraphs $ h $ such that $ \sfv_b(h) \leq \omega-2 $.
\end{lemma}
\begin{proof}
First we show that there exist a pair of graphs $H$ and $H'$ such that $\n(g,H)=\n(g,H')$ for all connected graphs $g$ with at most $\omega$ vertices expect for the clique $K_\omega$, and $\n(g,H)=\n(g,H')=0$ for all connected graphs $g$ with more than $\omega$ vertices. 
Analogous to the proof of \prettyref{lmm:subgraphmatching}, this either follows from the strong independence of injective graph homomorphism numbers \cite{Erdos1979}, or from the following linear algebra argument. Let $\{h_1,\ldots,h_m\}$ denote all distinct (up to isomorphism) \emph{connected} graphs of at most $\omega$ vertices. Order the graphs in increasing number of edges (arbitrarily among graphs with the same number of edges) so that $h_1$ is an isolated vertex and $h_m=K_{\omega}$. Then the matrix $B=(b_{ij})$ with $b_{ij} = \n(h_i,h_j)$ is upper triangular with strictly positive diagonals. Then $H$ and $H'$ can be constructed from the vector $x = B^{-1} e_m$ similar to \prettyref{eq:HH-matching}; see \prettyref{eq:HH-nhbd4} for a concrete example for $K_4$.
By design, each connected component of $H$ and $H'$ has at most $\omega$ vertices, we have $\n(g,H)=\n(g,H')=0$ for all connected $g$ with $\sfv(g)>\omega$.

Next we show that the neighborhood subgraph counts are matched up to order $\omega-2 $. 
For each neighborhood subgraph $ h $ with $ \sfv_b(h) \leq \omega-2 $, by \prettyref{lmm:incexc}, we have 
	$\N(h,H) = \sum_{g \in \calG} c(g,h) \n(g,H)$, where the coefficients $c(g,h)$ are independent of $H$, and $\calG$ contains all subgraphs obtained from $h$ by adding edges incident to black vertices in $h$ or attaching vertices to black vertices in $h$.
	The crucial observation is two-fold:
	(a) since $ \sfv_b(h) \leq \omega-2 $, there exists at least a pair of white vertices in $h$, which are not connected. Since no edges are added between white vertices, the collection $\calG$ excludes the full clique $K_\omega$;
	(b) for each $g\in \calG$, if $g$ contains more than $\omega$ vertices, then $\n(g,H)=\n(g,H)=0$; if $g$ contains at most $\omega$ vertices (and not $K_\omega$ by the previous point), then $\n(g,H)=\n(g,H)$ by design.
	Therefore we conclude that 
$ \N(h, H) = \N(h, H') $ for all neighborhood subgraphs $ h $ with $ \sfv_b(h) \leq \omega-2 $. 
\end{proof}

\subsection{Subgraph sampling} \label{sec:lb-subgraph}

Next we prove the lower bound part of \prettyref{thm:subgraph-main}:
\begin{proof}
Throughout the proof, we assume that both $d$ and $s$ are at least some sufficiently large constant that only depends on $k=\sfv(h)$ and we use $c,c',c_0,c_1,\ldots$ to denote constants that possibly depend on $k$ only.
We consider two cases separately.
\paragraph{Case I: $ p \leq 1/d $.}

Let $H$ and $H'$ be the pair of graphs from \prettyref{lmm:subgraphmatching}, such that $s(h,H)-\s(h,H')\geq 1$ and 
$\s(g,H)=\s(g,H')$ for all induced subgraphs $g$ with $\sfv(g) \leq k-1$.
Therefore, by \prettyref{lmm:counts}, we have $ \TV(P_{\tH}, P_{\tH'}) = O_{k}(p^{k-1})$. 
Let $r = \s(h,H)$ which is a constant only depending on $k$.
Applying \prettyref{thm:mainlb} with $M = \floor{s / r}$ yields the lower bound
\begin{equation}
\inf_{\widetilde\s} \sup_{\substack{G: ~\sfd(G)\leq d\\ ~~~\s(h,G) \leq s}} \Expect_G|\widetilde{\s}-\s(h, G)|^2 
= \Omega_{k} \left( \frac{s}{p^{k}} \wedge s^2 \right).
\label{eq:subgraphlb1}
\end{equation}

\paragraph{Case II: $ p > 1/d $.}
To apply \prettyref{lmm:lower-complete}, we construct a pair of graphs $H$ and $H'$  with maximum degree $d$ such that $\TV(P_{\tH}, P_{\tH'}) \leq 1/2$, 
$\s(h,H')=0$ and $c_1 \ell^{k-1}/p \leq \s(h,H) \leq c_2 \ell^{k-1}/p$.
Choosing $\ell = c_3 ((sp)^{\frac{1}{k-1}}  \wedge d)$ for some small constant $c_3$ and applying \prettyref{thm:mainlb}, we obtain
\begin{equation}
\inf_{\widetilde\s} \sup_{\substack{G: ~\sfd(G)\leq d\\ ~~~\s(h,G) \leq s}} \Expect_G|\widetilde{\s}(H)-\s(h, G)|^2 
= \Omega\pth{\frac{\ell^{k-1} s }{p}}  = \Theta_{k} \left( \frac{sd^{k-1}}{p} \wedge s^2 \right).
\label{eq:subgraphlb2}
\end{equation}
Combining \prettyref{eq:subgraphlb1} and \prettyref{eq:subgraphlb2} completes the proof of the lower bound of \prettyref{thm:subgraph-main}. 

It remains to construct $H$ and $H'$. The idea of the construction is to expand each vertex in $h$ into an independent set, which was used in the proof of \cite[Lemma 5]{Erdos1979}. Here, we also need to consider the possibility of expanding into a clique. Next consider two cases:

Suppose $h$ satisfies the ``distinct neighborhood'' property, that is, for each $v \in V(h)$, $N_h(v)$ is a distinct subset of $v(h)$. 
Such $h$ includes cliques, paths, cycles, etc.
Pick an arbitrary vertex $u\in V(h)$. 
Let $\{S_v: v \in V(h)\}$ be a collection of disjoint subsets, so that $|S_u| = \ceil{ c/p}$ and $|S_v| = \ceil{c d}$, where
$c$ is a constant that only depends on $\sfv(h)=k$ such that $c k \leq 1$. 
Define a graph $H$ with vertex set $\cup_{v\in v(h)} S_v$ by connecting each pair of $a \in S_u$ and $b\in S_v$ whenever $(u,v) \in E(h)$. In other words, $H$ is obtained by blowing up each vertex in $h$ into an independent set and each edge into a complete bipartite graph.
Repeating the same construction with $h$ replaced by $h-u$ yields $H'$, in which case $S_u$ consists of isolated vertices.
By construction, the maximum degree of both graph satisfies is at most $d$.
Note that $H-S_u = H'-S_u$. 
Thus the sampled graph of $H$ and $H'$ have the same law provided that none of the vertices in $S_u$ is sampled. 
Applying \prettyref{lmm:lower-complete}, we conclude that $\TV(P_{\tH}, P_{\tH'}) \leq (1-p)^{c/p} \leq c'$ for all $p \leq 1/2$, where $c'$ is a constant depending only on $k$. 

Furthermore, 
\begin{align*}
\s(h,H')
= & ~ \sum_{T \cap S_u = \emptyset}\indc{H'[T] \simeq h}  + \sum_{T \cap S_u \neq \emptyset}\indc{H'[T] \simeq h}  \\
\overset{(a)}{=} & ~ \sum_{T \cap S_u = \emptyset}\indc{H'[T] \simeq h}   = \sum_{T \cap S_u = \emptyset}\indc{H[T] \simeq h},
\end{align*}
where (a) follows from the fact that $H'[T]$ contains isolated vertices whenever $T \cap S_u \neq \emptyset$ while $h$ is connected by assumption.
Note that since $|T|=k$, if $T \cap S_u = \emptyset$, then there exists $t,t' \in T$ such that $t,t'$ belong to the same independent set $S_v$ for some $v$. By construction, $t$ and $t'$ have the same neighborhood, contradicting $H[T]\simeq h$. 
Thus, we conclude that $\s(h,H')=0$. For $H$, we have
\begin{align*}
\s(h,H)
= & ~ \sum_{T \cap S_u \neq \emptyset}\indc{H[T] \simeq h}  
\geq |S_u| \prod_{v \neq u} |S_v| \geq c^k \ell^{k-1}/p,
\end{align*}
and, similarly, $\s(h,H) \leq |S_u| (\sum_{v \neq u} |S_v|)^{k-1} \leq (2ck)^k \ell^{k-1}/p$. 

Next suppose that $h$ does have distinct neighborhoods, thus there exist $\{u_1,\ldots,u_\ell\} \subset V(h)$ with $\ell\geq 2$ such that the neighborhood $N_{u_i}(h)$ are identical, denoted by $T$. Let $g \triangleq h[T] \vee \ell K_1$ is an induced (by $T \cup \{u_1,\ldots,u_\ell\}$) subgraph of $h$. We define $H$ with vertex set $\cup_{v \in v(h)} S_v$ by the same procedure as above, except now all vertices are expanded into a clique, with $|S_{u_1}|=\ceil{c/p}$ and 
$|S_v| = \ceil{cd}$ for $v \neq u$. Finally, as before, we connect each pair of $a \in S_u$ and $b\in S_v$ whenever $(u,v) \in E(h)$. Define $H'$ by repeating the same construction with $h$ replaced by $h-u_1$.
Analogous to the above we have $\TV(P_{\tH},P_{\tH'}) \leq c$ and it remains to show that $\s(h,H')=0$. Indeed, for any set $T$ of $k$ vertices that does not include any vertex from $S_{u_1}$, since $S_{u_i}$ forms a clique and $u_1,\ldots,u_\ell$ form an independent set in $h$, the number of induced $g$ in $H'[T]$ is strictly less than that in $h$. Thus, there exists no $T \subset \cup_{v\neq u_1} S_v$ such that $H'[T]$ is isomorphic to $h$, and hence $\s(h,H')=0$.
Entirely analogously, we have $\s(h,H) = \Theta_k(\ell^{k-1}/p)$. 
\end{proof}

\subsection{Neighborhood sampling} \label{sec:lb-neighborhood}


To illustrate the main idea, we only prove the lower bound cliques. The proof for other motifs (of size up to four) is similar but involves several ad hoc constructions; see \prettyref{app:othermotif}.

\begin{theorem}[Cliques] \label{thm:neighborhood_main}
For neighborhood sampling with sampling ratio $p$,
\[
\inf_{\widehat{\s}} \sup_{\substack{G: ~\sfd(G)\leq d\\ ~~~\s(h,G) \leq s}} 
\Expect_G|\widehat{\s}-\s(K_{\omega}, G)|^2 = \Theta_{\omega} \left( \left(\frac{s d}{p^{\omega-1}} \wedge  \frac{s d^{\omega-2}}{p^2}\right) \wedge s^2 \right)
\]
\end{theorem}

\begin{proof}
For the lower bound, consider two cases. For simplicity, denote the minimax risk on the left-hand side by $R$.
\paragraph{Case I: $ p > 1/d $.}
Applying \prettyref{lmm:lower-complete} with 
$G$ being the complete $(\omega-2)$-partite graph of $(\omega-2)\ell$ vertices, 
$ H_1 = K_{1/p, 1/p} $, and $ H_2 = (2/p)K_1 $, we obtain two graphs $ H $ and $ H' $ with $ \s(K_{\omega}, H) \asymp \frac{\ell^{\omega-2}}{p^2} $ and $\s(K_{\omega}, H')=0$, and $\TV(P_{\tH}, P_{\tH'}) \leq c < 1 $ for all $p \leq 1/2$.
By \prettyref{thm:mainlb} with $M=s/(\ell^{\omega-2}/p^2)$, we obtain the lower bound $ R \gtrsim \frac{s\ell^{\omega-2}}{p^2} $. Let $ \ell = cd $ if $ \frac{d^{\omega-2}}{p^2} \leq s $ and $ \ell = c(p^2s)^{\frac{1}{\omega-2}} $ if $ \frac{d^{\omega-2}}{p^2} > s $, for some small constant $ c $. In either case, we find that $\s(K_{\omega}, H) \leq s $, $ \s(K_{\omega}, H') \leq s $, and $ R \asymp \frac{s\ell^{\omega-2}}{p^2} \asymp \frac{s d^{\omega-2}}{p^2} \wedge s^2 $.

\paragraph{Case II: $ p \leq 1/d $.}
We use a different construction.
Let $\ell= c(d \wedge s^{1/\omega})$ for some small constant $c$. Let $ H $ and $ H' $ be the two graphs from \prettyref{lmm:general_clique_neighborhood} such that $ \s(K_{\omega}, H)- \s(K_{\omega}, H') \geq 1 $ and  $ \N(h, H) = \N(h, H') $ for all neighborhood subgraphs $ h $ with $ \sfv_b(h) \leq \omega-2 $.
By \prettyref{lmm:counts}, we have $ \TV(P_{\tH}, P_{\tH'}) = O_{\omega}(p^{\omega-1}) < 1 $. 
Next we amplify the gap $ |\s(K_{\omega}, H)- \s(K_{\omega}, H')|= \Omega(\ell^{\omega}) $ by expanding each vertex into an independent set, similar in what is done in the proof of \prettyref{thm:subgraph-main}. For each vertex in $H$, we associate $ \ell $ distinct isolated vertices, and connect each pair of vertices by an edge if and only if they were connected in $H$. This defines a new graph $F$ with $\ell \sfv(H)$ vertices and similarly we construct $F'$ from $H'$. In this way, the subgraph counts of $F$ and $F'$ also match up to order $ \omega-2 $, and, in view of \prettyref{lmm:counts}, $ \TV(P_{\tH}, P_{\tH'}) = O_{\omega}((\ell p)^{\omega-1}) $.
Furthermore, the number of cliques satisfies $\s(K_\omega,F)=\s(K_\omega,H) \ell^\omega$ and $\s(K_\omega,F')=\s(K_\omega,H') \ell^\omega$.
Thus, $ \s(K_{\omega}, F) \asymp \s(K_{\omega}, F') \asymp |\s(K_{\omega}, H) - \s(K_{\omega}, H')| = \ell^{\omega} $.
Applying \prettyref{thm:mainlb} with $M=s/\ell^{\omega}$ yields  $R  \gtrsim (\ell^{\omega} (\sqrt{\frac{s/\ell^{\omega}}{(p\ell)^{\omega-1}}} \wedge \frac{s}{\ell^{\omega}}))^2  \asymp \frac{s\ell}{p^{\omega-1}} \wedge s^2 \asymp \frac{s d }{p^{\omega-1}} \wedge \frac{s^{1+1/\omega}}{p^{\omega-1}} \wedge s^2 \asymp \frac{s d }{p^{\omega-1}}  \wedge s^2$, where the last step follows from the assumption that $ p \leq 1/d $.
\end{proof}

\section{Graphs with additional structures} \label{sec:structure}

In this section, we explore how estimation of motif counts can be improved by prior knowledge of the parent graph structure. In particular, for counting edges, we show that even if the parent graph is known to be a forest a priori, for neighborhood sampling, the bound in \prettyref{thm:edge} remains optimal up to a subpolynomial factor in $p$. Similarly, for subgraph sampling, we cannot improve the rate in \prettyref{thm:subgraph-main}. We also discuss some results for planar graphs. In what follows, we let $ \calF $ and $ \calP $ denote the collection of all forests and planar graphs, respectively.

The next results shows that for estimating edge counts, even if it is known a priori that the parent graph is a forest, the risk in \prettyref{thm:subgraph-main} and \prettyref{thm:edge} cannot be improved in terms of the exponents on $p$. The proofs of all the following results are given in \prettyref{app:additional_proofs}.

\begin{theorem} \label{thm:subgraph-rates-forest}
For subgraph sampling with sampling ratio $ p $,
\begin{equation}
\inf_{\widehat{\sfe}}\sup_{\substack{G\in \calF: ~\sfd(G)\leq d\\ ~~~\sfe(G) \leq m}}\Expect_G|\widehat{\sfe}-\sfe(G)|^2 \asymp \left(\frac{m}{p^2} \vee  \frac{md}{p}\right)\wedge m^2.
\end{equation}
\end{theorem} 

\begin{theorem} \label{thm:edge-forest}
For neighborhood sampling with sampling ratio $ p $,
\begin{equation*}
\inf_{\widehat{\sfe}}\sup_{\substack{G \in \calF: ~\sfd(G)\leq d\\ ~~~\sfe(G) \leq m}}\Expect_G|\widehat{\sfe}-\sfe(G)|^2 = \Omega \left( \frac{m}{p^{2+o(1)}} \wedge \frac{md}{p^{1+o(1)}} \wedge \frac{m^2}{p^{o(1)}} \right),
\end{equation*}
where $ o(1) = 1/\sqrt{\log \frac{1}{p}} $ is with respect to $ p\to0 $ and uniform in all other parameters.
\end{theorem}

For estimating the wedge count under subgraph sampling, the following result shows that the risk in \prettyref{thm:subgraph-main}  cannot be improved even if we know the parent graph is a forest.

\begin{theorem}
\label{thm:wedge-rates-forest}
For subgraph sampling with sampling ratio $ p $,
\begin{equation}
\inf_{\widehat{\sfw}}\sup_{\substack{G \in \calF: ~\sfd(G)\leq d\\ ~~~\sfw(G) \leq w}}\Expect_G|\widehat{\sfw}-\sfw(G)|^2 \asymp \left(\frac{w}{p^3} \vee  \frac{wd^2}{p}\right)\wedge w^2.
\end{equation}
\end{theorem}

On the other hand, for neighborhood sampling, the tree structure can be exploited to improve the rate. Analogous to \prettyref{eq:edge-estimator1}, we consider an estimator of the form
\begin{equation}
\widehat{\sfw} = \lambda \N(\tikz[scale=0.75,baseline=(zero.base)]{\draw (0,0) node (zero) [vertexdotsolid] {} -- (0.25,{0.5*sin(60)}) node[vertexdot] {} -- (0.5,0) node[vertexdotsolid] {};},\tG) + \alpha \N(\tikz[scale=0.75,baseline=(zero.base)]{\draw (0,0) node (zero) [vertexdotsolid] {} -- (0.25,{0.5*sin(60)}) node[vertexdotsolid] {} -- (0.5,0) node[vertexdot] {};},\tG) + \beta \N(\tikz[scale=0.75,baseline=(zero.base)]{\draw (0,0) node (zero) [vertexdotsolid] {} -- (0.25,{0.5*sin(60)}) node[vertexdotsolid] {} -- (0.5,0) node[vertexdotsolid] {};},\tG),
\label{eq:wedge-estimator1}
\end{equation}
If we weight \tikz[scale=0.75,baseline=(zero.base)]{\draw (0,0) node (zero) [vertexdotsolid] {} -- (0.25,{0.5*sin(60)}) node[vertexdot] {} -- (0.5,0) node[vertexdotsolid] {};} and \tikz[scale=0.75,baseline=(zero.base)]{\draw (0,0) node (zero) [vertexdotsolid] {} -- (0.25,{0.5*sin(60)}) node[vertexdotsolid] {} -- (0.5,0) node[vertexdot] {};} equally, i.e., $ \alpha = \lambda $, this estimator reduces to \prettyref{eq:hest} and hence inherits the same performance guarantee in \prettyref{eq:var-general-graph}, which by \prettyref{thm:broken}, is optimal. However, as will be seen in \prettyref{thm:broken-forest}, there is added flexibility by this three-parameter family of estimators that produces improved bounds when the parent graphs satisfies certain additional structure. It should also be mentioned that the alternative choices $ \lambda = \frac{5-8p}{p^2(4p-3)} $, $ \alpha = \frac{1}{p^2} $, and $ \beta = \frac{3p-2}{p^3(4p-3)} $ yield the same performance bound as in \prettyref{eq:var-general-graph}.

For this next result, we show that we can improve the performance of the wedge estimator \prettyref{eq:wedge-estimator1} if the parent graph is a forest by choosing alternate values of the parameters: $ \alpha = \frac{1}{2pq} $, $ \lambda = \frac{1}{p^2} $, and $ \beta = 0 $. These choices eliminate the largest term in the variance of \prettyref{eq:wedge-estimator1}, which is proportional to $ \n(S_4, G)(4\alpha^2p^3q^2+4\lambda\beta p^4q+\beta^2p^5+p^4q\lambda^2-1) $. We immediately get the following variance bound:
\begin{equation} \label{eq:variance-broken-forest}
\Var[\widehat{\sfw}] \lesssim \frac{\sfw(G)}{p^2} \vee \frac{\sfw(G)d}{p}.
\end{equation}

Note also that $ \s(P_3, G) = \sum_u \binom{d_u}{2} $ whenever $G$ is a forest. Hence another estimator we can use is $ \sum_u \frac{b_u}{p}\binom{d_u}{2} $ which has variance of order $ \frac{\sfw(G)d^2}{p} $. Putting this all together, we obtain the following result.
\begin{theorem} \label{thm:wedge-forest}
For neighborhood sampling with sampling ratio $ p $,
\begin{equation*}
\inf_{\widehat{\sfw}}\sup_{\substack{G\in \calF: ~\sfd(G)\leq d\\ ~~~\sfw(G) \leq w}}\Expect_G|\widehat{\sfw}-\sfw(G)|^2 \lesssim \left(\frac{w}{p^{2}} \vee \frac{wd}{p} \right) \wedge \left(\frac{wd^2}{p}\right) \wedge w^2.
\end{equation*}
\end{theorem}

The next theorem shows that the minimax bound from \prettyref{thm:wedge-forest} is optimal.

\begin{theorem}\label{thm:broken-forest}
For neighborhood sampling with sampling ratio $ p $ and $ w \geq d $,
\begin{equation*}
\inf_{\widehat{\sfw}}\sup_{\substack{G \in \calF: ~\sfd(G)\leq d\\ ~~~\sfw(G) \leq w}}\Expect_G|\widehat{\sfw}-\sfw(G)|^2 = \Omega \left(\left(\frac{w}{p^2} \vee \frac{wd}{p} \right) \wedge \left(\frac{wd^2}{p}\right) \wedge w^2 \right).
\end{equation*}
\end{theorem}

In the context of estimating triangles, the next set of results show that planarity improves the rates of estimation for both sampling models. Despite the smaller risk however, for subgraph sampling, the optimal estimator is still the Horvitz-Thompson type.

\begin{theorem} \label{thm:triangle-planar-subgraph}
For subgraph sampling with sampling ratio $ p $,
\begin{equation*}
\inf_{\widehat{\sft}}\sup_{\substack{G \in \calP: ~\sfd(G)\leq d\\ ~~~\sft(G) \leq t}}\Expect_G|\widehat{\sft}-\sft(G)|^2 \asymp \left(\frac{t}{p^3} \vee \frac{td}{p^2} \right) \wedge t^2.
\end{equation*}
\end{theorem}

\begin{theorem} \label{thm:triangle-planar}
For neighborhood sampling with sampling ratio $ p $,
\begin{equation*}
\left(\left( \frac{t}{p^{7/3}} \wedge \frac{td}{p^2} \right)\vee \frac{td}{p} \right) \wedge t^2 \lesssim \inf_{\widehat{\sft}}\sup_{\substack{G \in \calP: ~\sfd(G)\leq d\\ ~~~\sft(G) \leq t}}\Expect_G|\widehat{\sft}-\sft(G)|^2 \lesssim \left(\left(\frac{t}{p^3} \wedge \frac{td}{p^2} \right)\vee \frac{td}{p} \right) \wedge t^2.
\end{equation*}
\end{theorem}

\section{Numerical experiments} \label{sec:experiments}

We perform our experiments on both synthetic and real-world data. For the synthetic data, we take as our parent graph $ G $ a realization of an Erd\"os-R\'enyi graph $ \calG(N, \delta) $ for various choices of parameters. For the real-world experiment, we study the social networks of survey participants using a Facebook app \cite{snapnets}. This dataset contains 10 ego-networks (the closed neighborhood of a focal vertex (``ego") and any edges between vertices in its neighborhood) of various sizes, although we only use three of them as our parent graphs $ G $. The error bars in the following figures show the variability of the relative error of edges, triangles, and wedges over 10 independent experiments of subgraph and neighborhood sampling on a fixed parent graph $ G $. The solid black horizontal line shows the sample average and the whiskers show the mean $ \pm $ the standard deviation. 

Specifically, for subgraph sampling, we always use the HT estimator \prettyref{eq:HT}. 
For neighborhood sampling, for counting triangles or wedges, we use the estimator \prettyref{eq:hest} with choice of parameters given in \prettyref{thm:neighborhood_clique} and for counting edges we use the adaptive estimator in \prettyref{thm:adaptive}. 
The relative error for estimating the number of edges, triangles, and wedges are given in \prettyref{fig:edge}-- \prettyref{fig:wedge}, respectively.

	%
%
	%
	%
	%

As predicted by the variance bounds, the estimators based on neighborhood sampling perform better than subgraph sampling. Furthermore, there is markedly less variability across the 10 independent experiments in neighborhood sampling. In all plots, however, this variability decreases as $ p $ grows.
Furthermore, in accordance with our theory, counting bigger motifs (involving more vertices) is subject to more variability, which is evidenced in the plots for triangles and wedges by the wider spread in the whiskers.

\begin{figure}[!ht]
	\centering
	\subfigure[Facebook network (subgraph sampling).]
	{\label{fig:sub5} \includegraphics[width=0.4\textwidth]{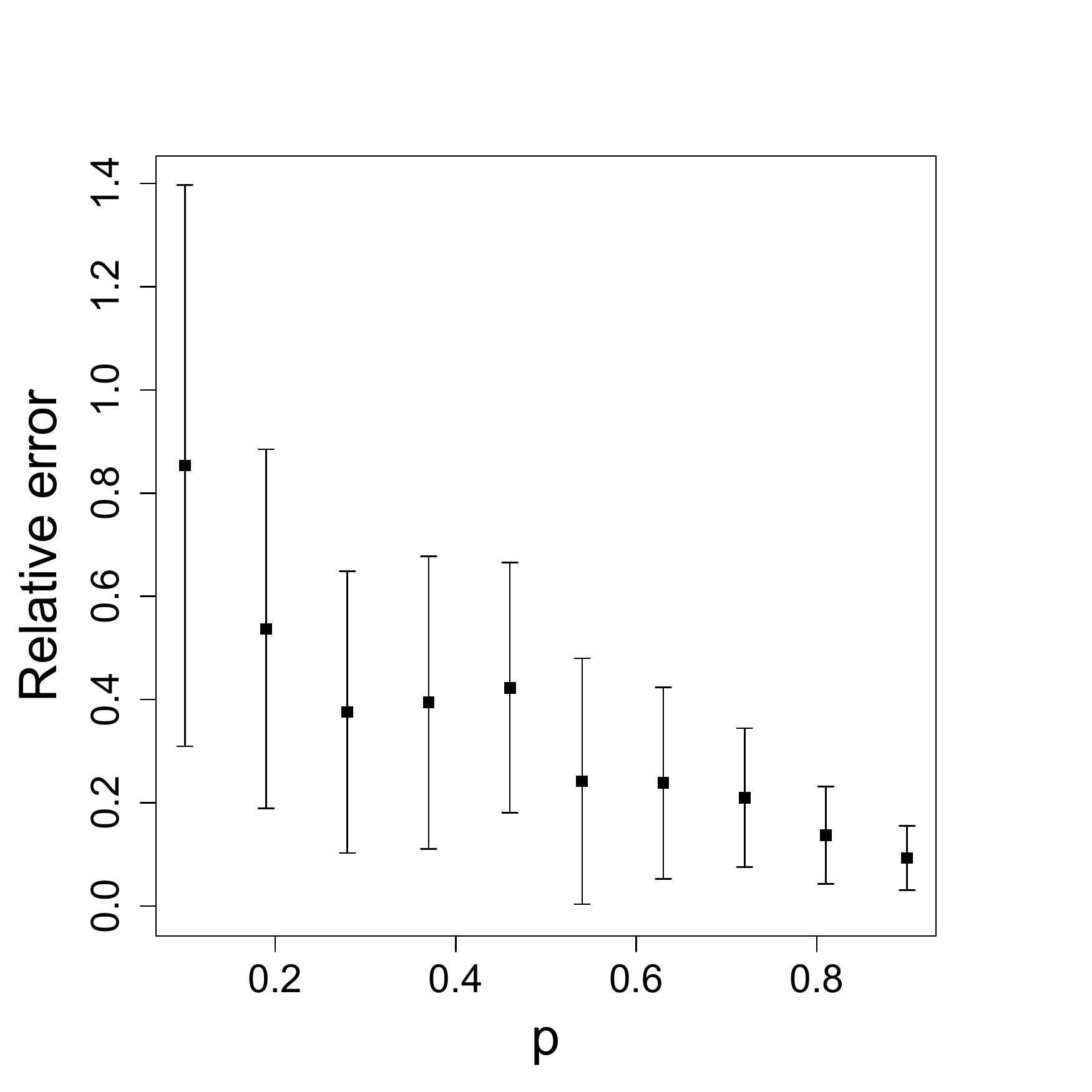}}
	~~
	\subfigure[Facebook network (neighborhood sampling).]%
	{\label{fig:sub6} 	\includegraphics[width=0.4\textwidth]{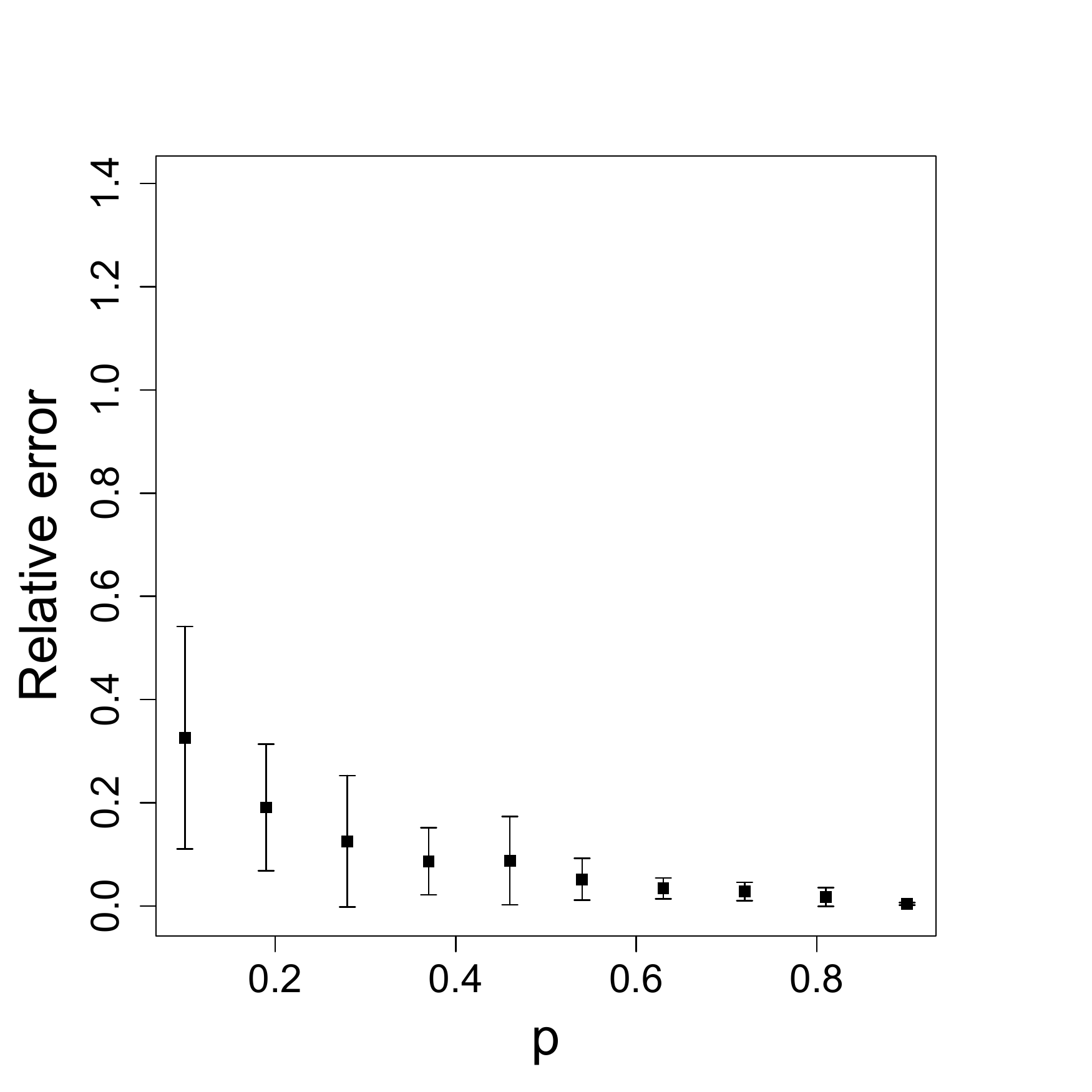}}
	\centering
	\subfigure[Erd\"os-R\'enyi graph (subgraph sampling).]
	{\label{fig:sub7} \includegraphics[width=0.4\textwidth]{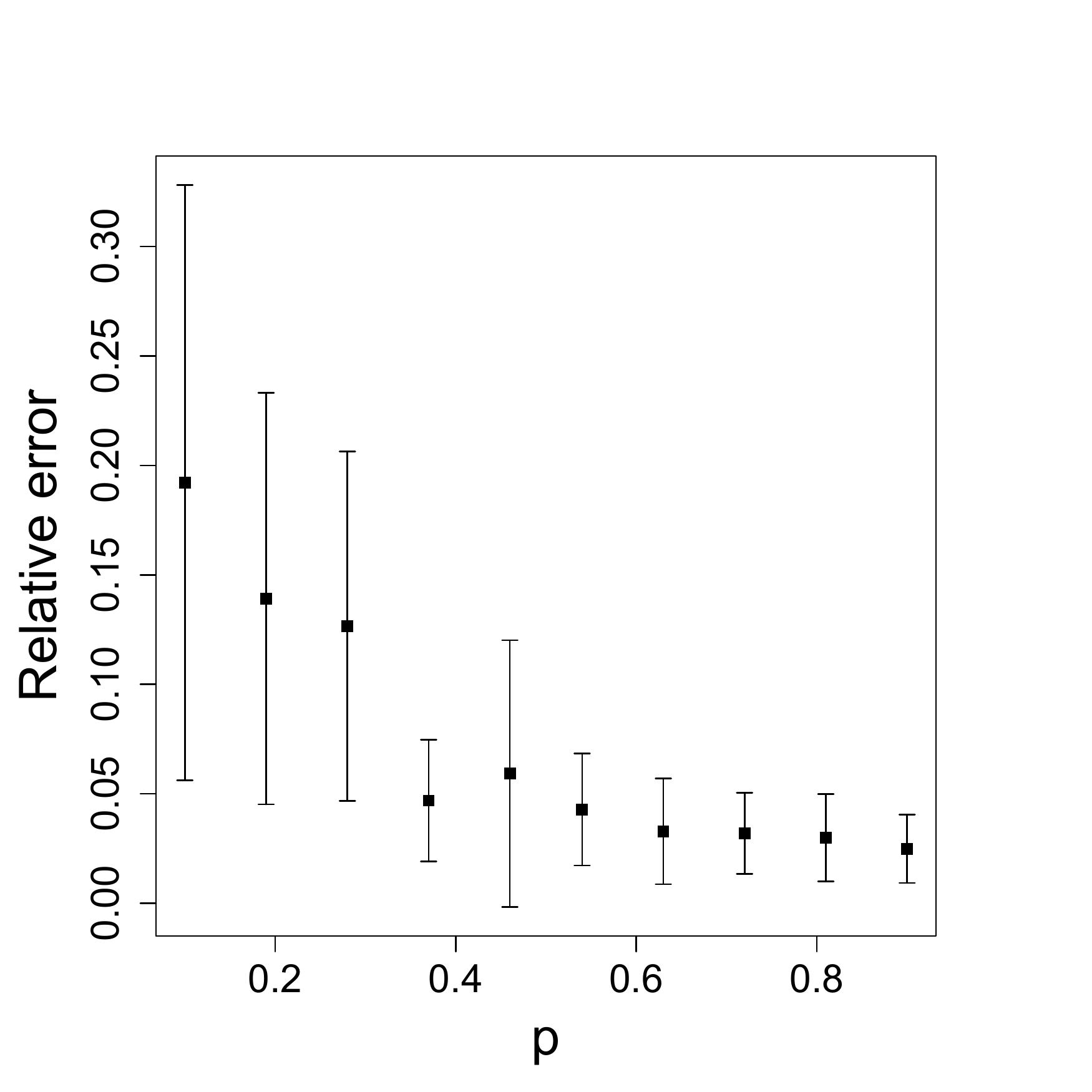}}
	~~
	\subfigure[Erd\"os-R\'enyi graph (neighborhood sampling).]%
	{\label{fig:sub8} 	\includegraphics[width=0.4\textwidth]{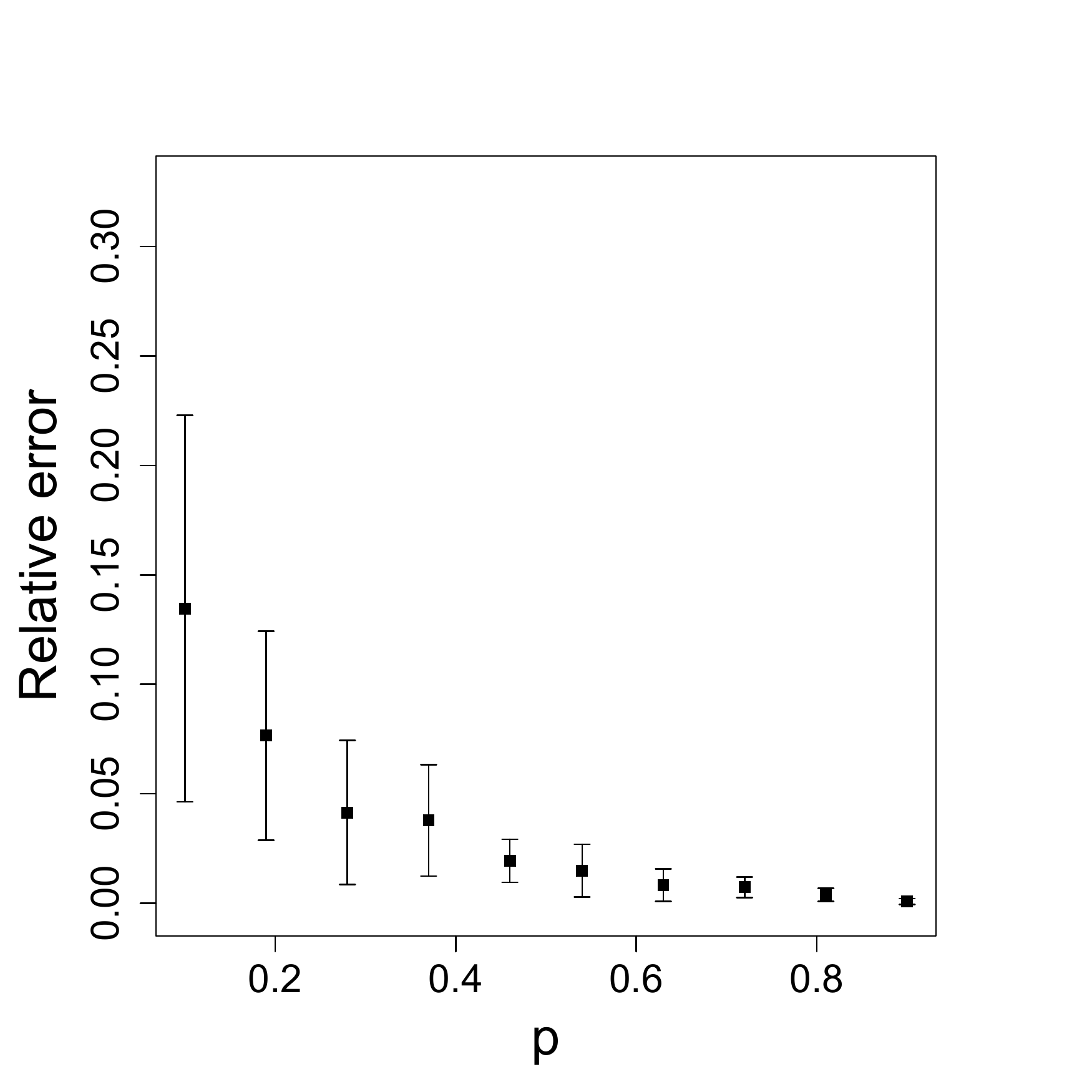}}
\caption{Relative error of estimating the edge count. In \prettyref{fig:sub5} and \prettyref{fig:sub6}, the parent graph $G$ is the Facebook network with $ d = 77 $, $ \sfv(G) = 333 $, $ \sfe(G) = 2519 $. In \prettyref{fig:sub7} and \prettyref{fig:sub8}, $G$ is a realization of the Erd\"os-R\'enyi graph $ \calG(1000, 0.05) $ with $ d = 12 $, and $ \sfe(G) = 2536 $.}
\label{fig:edge}
\end{figure}

\begin{figure}[!ht]
	\centering
	\subfigure[Facebook network (subgraph sampling).]
	{\label{fig:sub9} \includegraphics[width=0.4\textwidth]{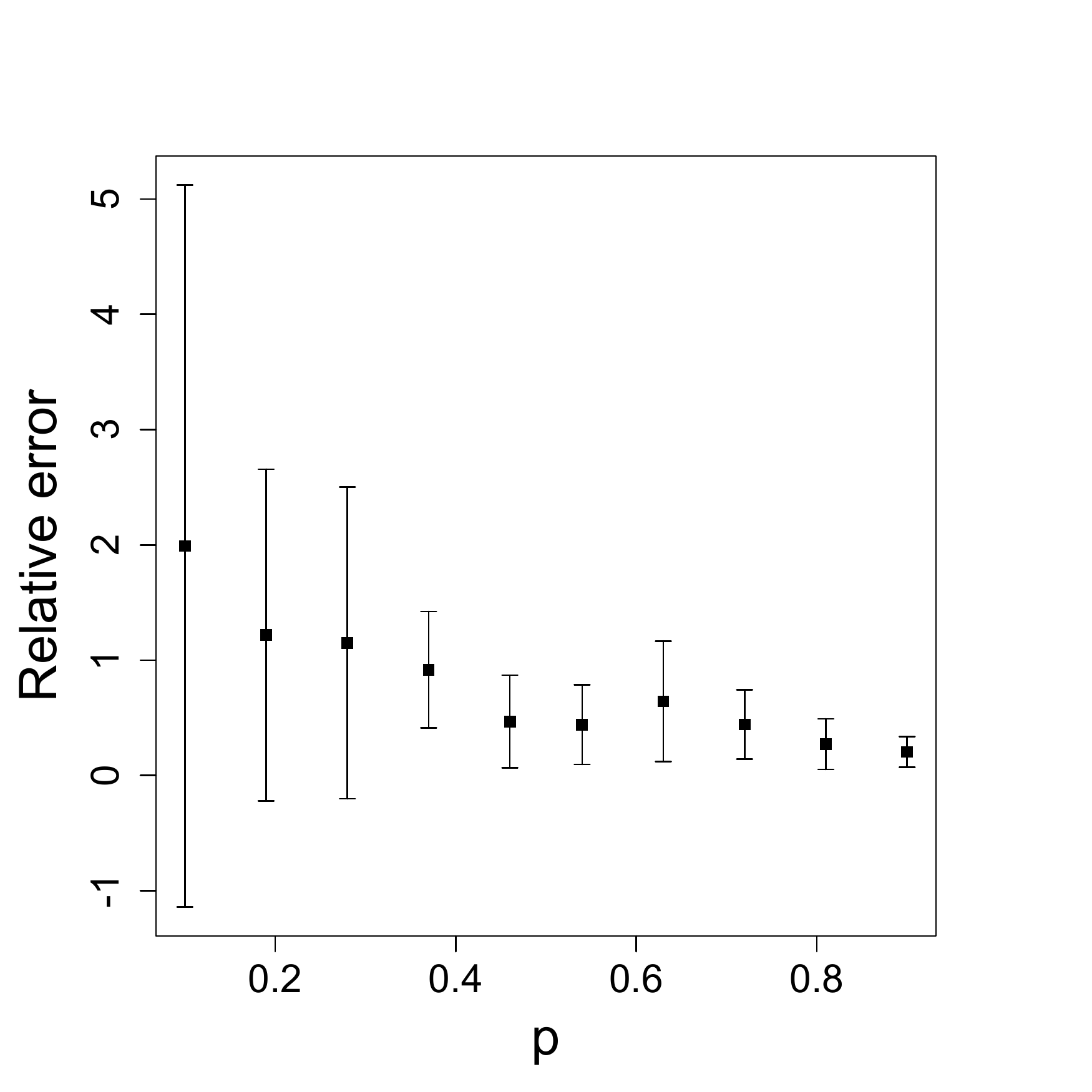}}
	~~
	\subfigure[Facebook network (neighborhood sampling).]%
	{\label{fig:sub10} 	\includegraphics[width=0.4\textwidth]{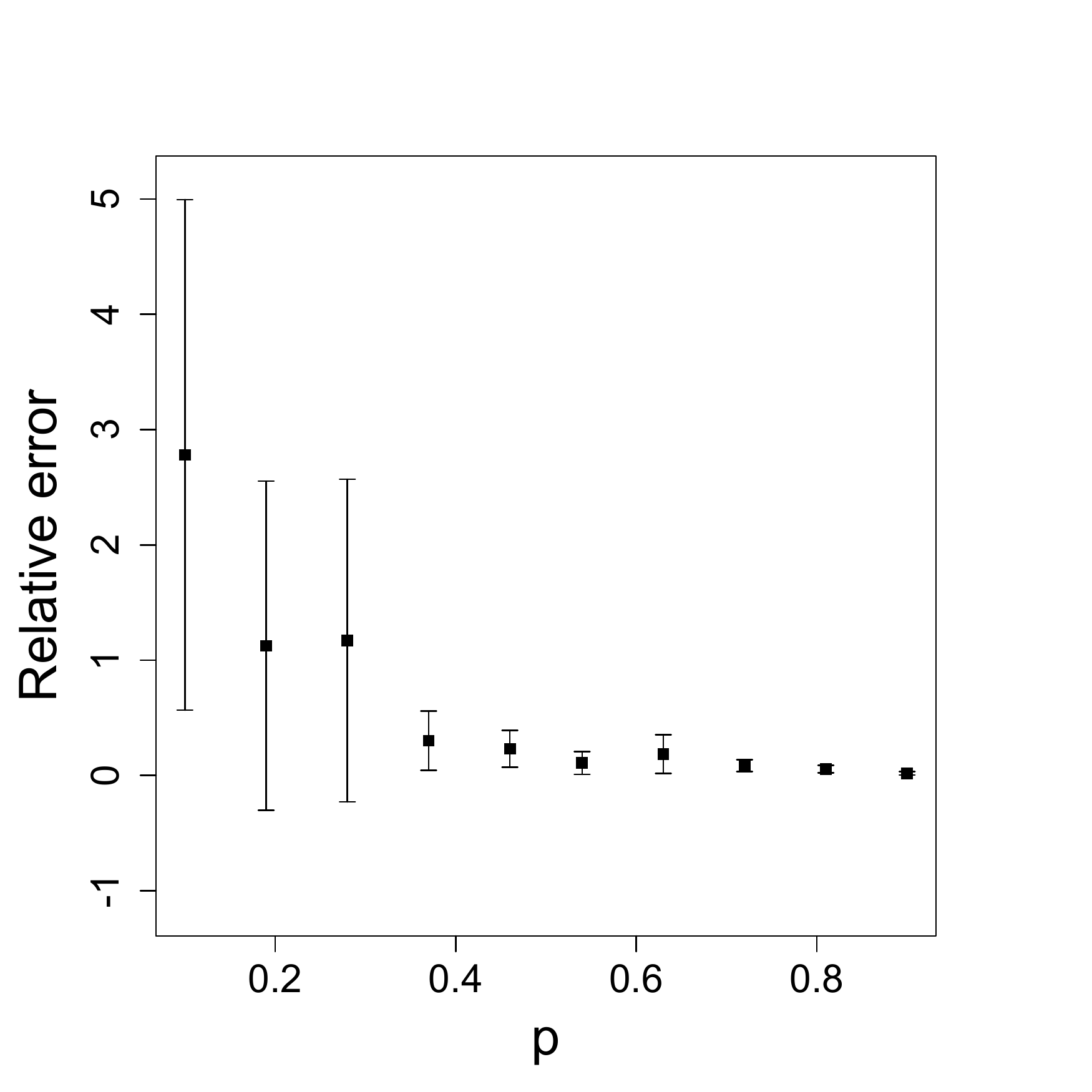}}
	\centering
	\subfigure[Erd\"os-R\'enyi graph (subgraph sampling).]
	{\label{fig:sub11} \includegraphics[width=0.4\textwidth]{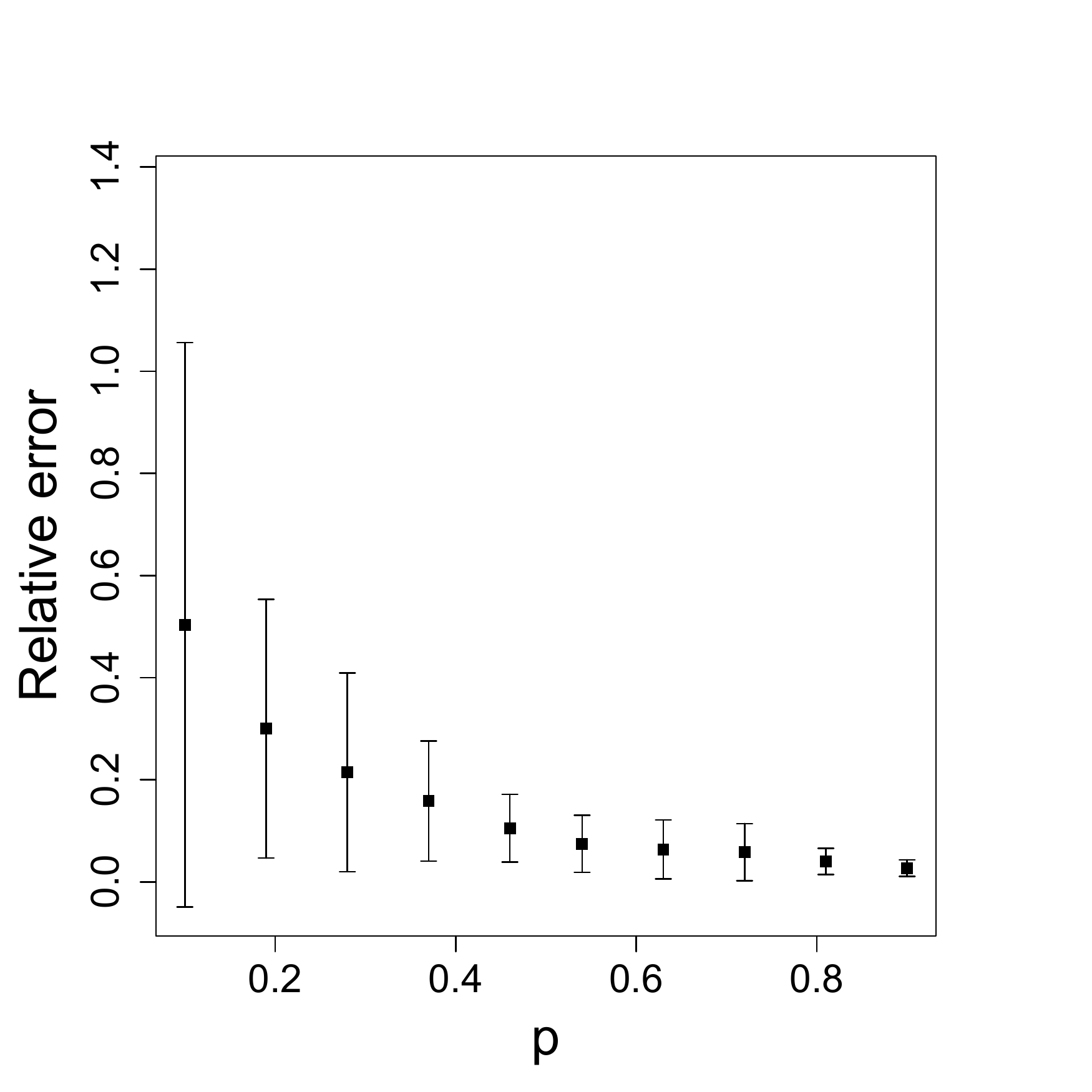}}
	~~
	\subfigure[Erd\"os-R\'enyi graph (neighborhood sampling).]%
	{\label{fig:sub12} 	\includegraphics[width=0.4\textwidth]{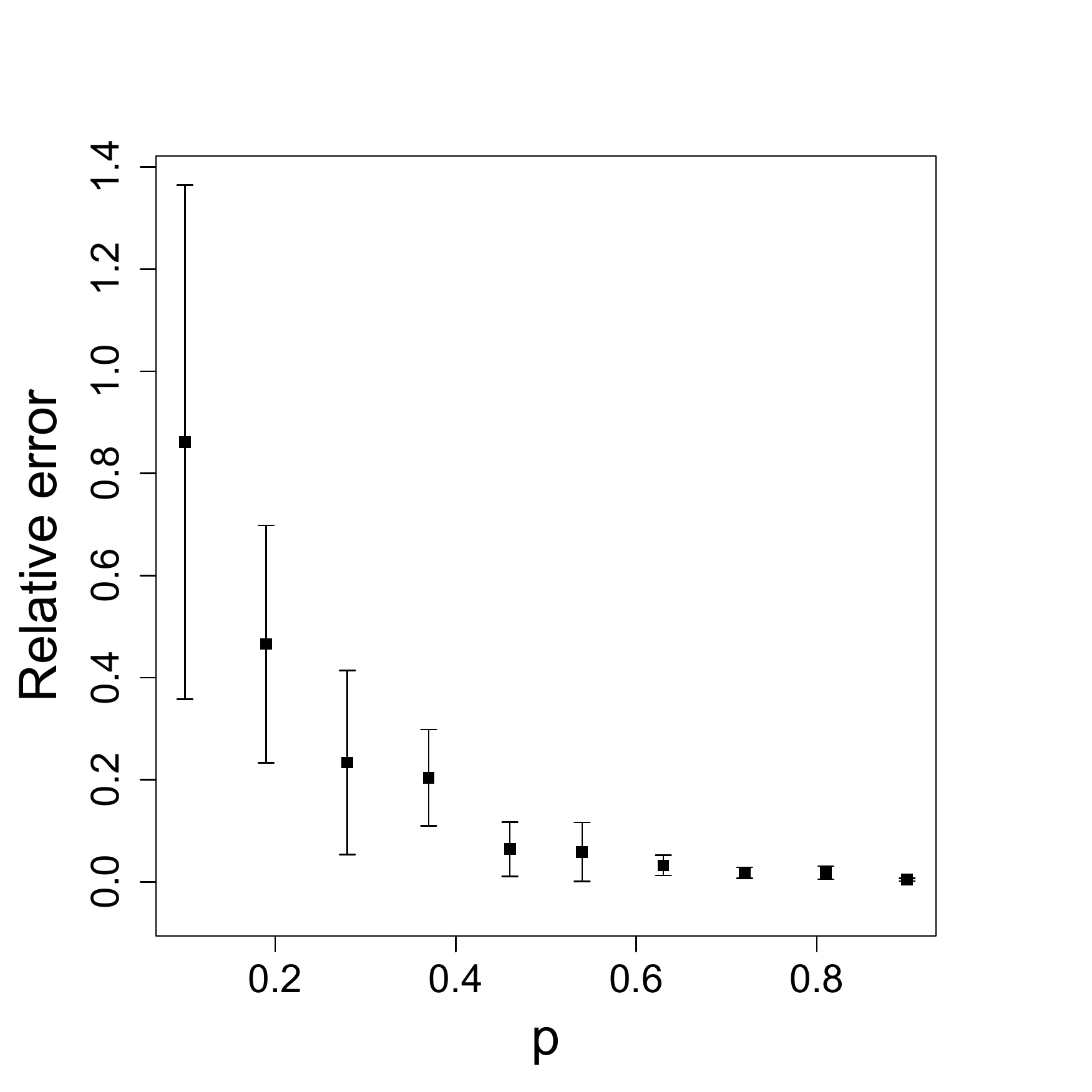}}
\caption{Relative error of counting triangles. In \prettyref{fig:sub9} and \prettyref{fig:sub10}, the parent graph is the Facebook network with $ d = 77 $, $ \sfv(G) = 168 $, $ \sft(G) = 7945 $. In \prettyref{fig:sub11} and \prettyref{fig:sub12}, the parent graph is a realization of the Erd\"os-R\'enyi graph $ \calG(1000, 0.02) $ with $ d = 35 $, and $ \sft(G) = 1319 $.}
\label{fig:triangle}
\end{figure}

\begin{figure}[!ht]
	\centering
	\subfigure[Facebook network (subgraph sampling).]
	{\label{fig:sub13} \includegraphics[width=0.4\textwidth]{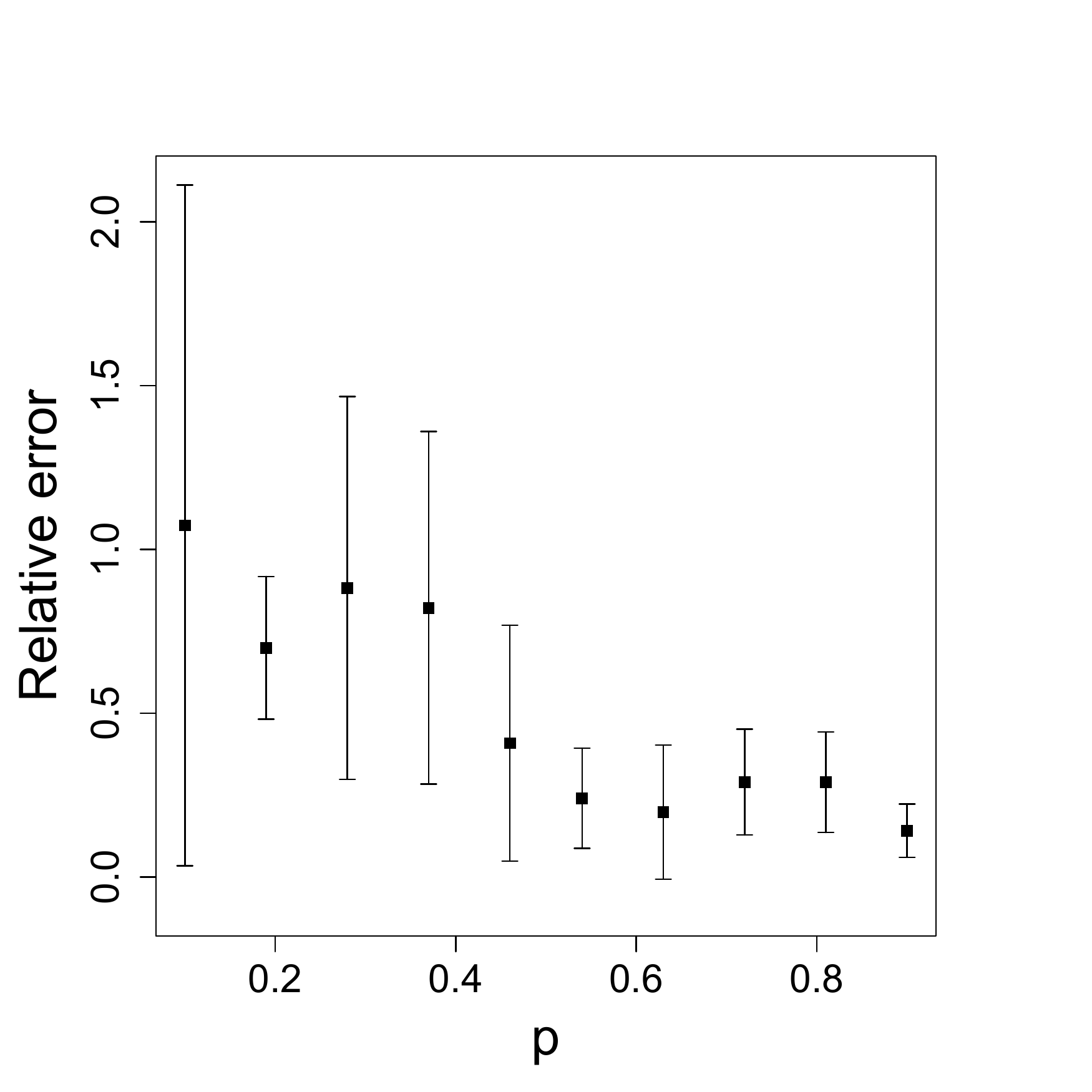}}
	~~
	\subfigure[Facebook network (neighborhood sampling).]%
	{\label{fig:sub14} 	\includegraphics[width=0.4\textwidth]{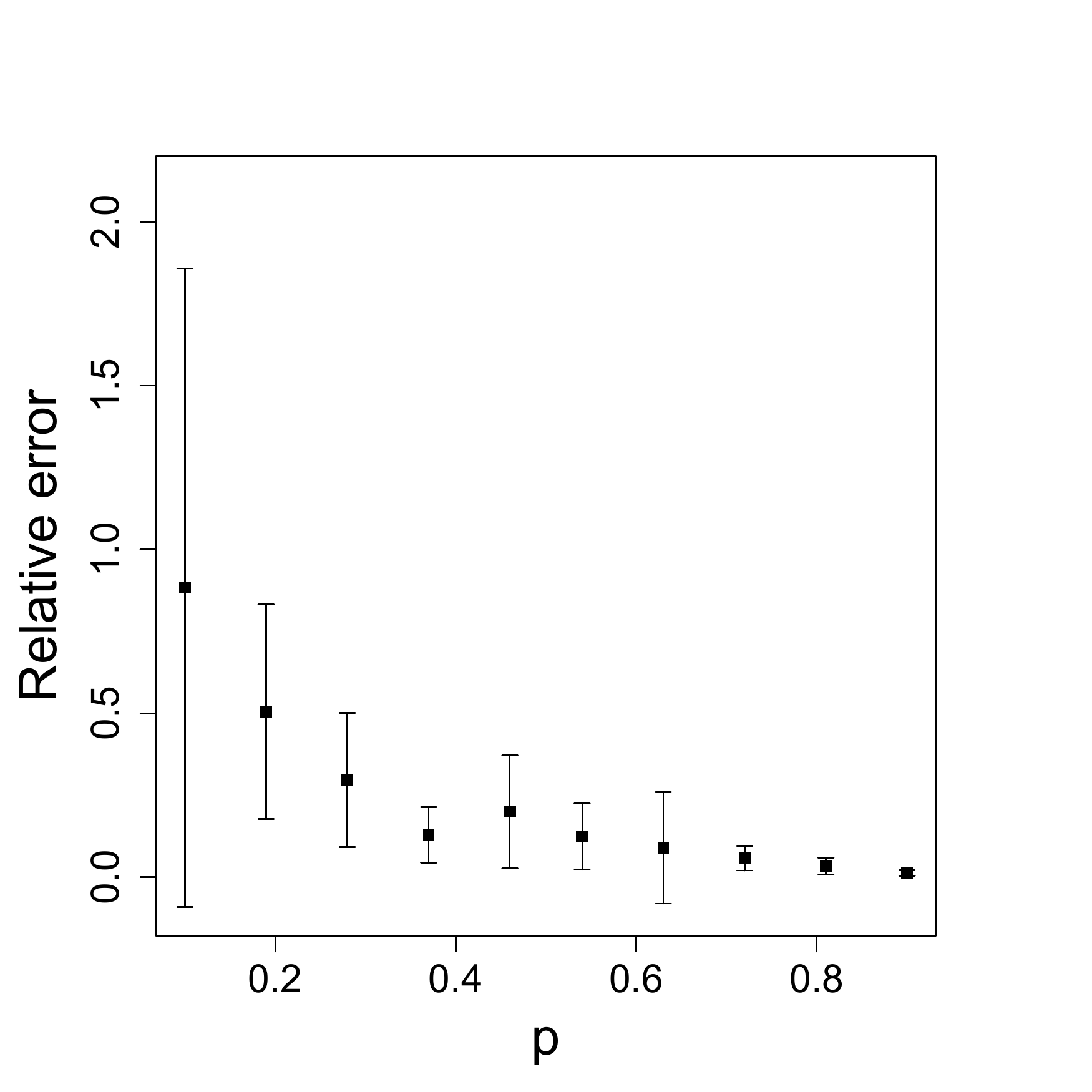}}
	\centering
	\subfigure[Erd\"os-R\'enyi graph (subgraph sampling).]
	{\label{fig:sub15} \includegraphics[width=0.4\textwidth]{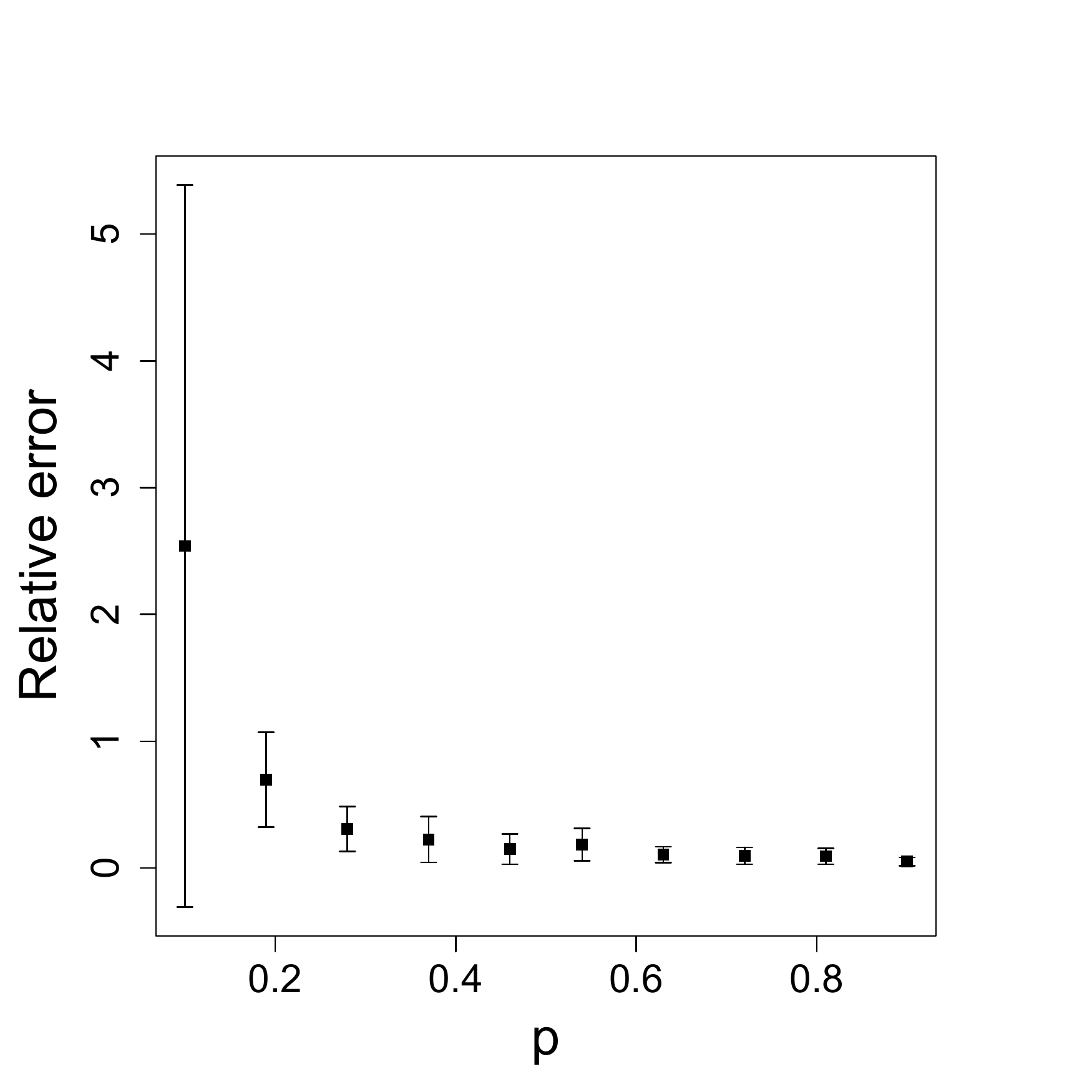}}
	~~
	\subfigure[Erd\"os-R\'enyi graph (neighborhood sampling).]%
	{\label{fig:sub16} 	\includegraphics[width=0.4\textwidth]{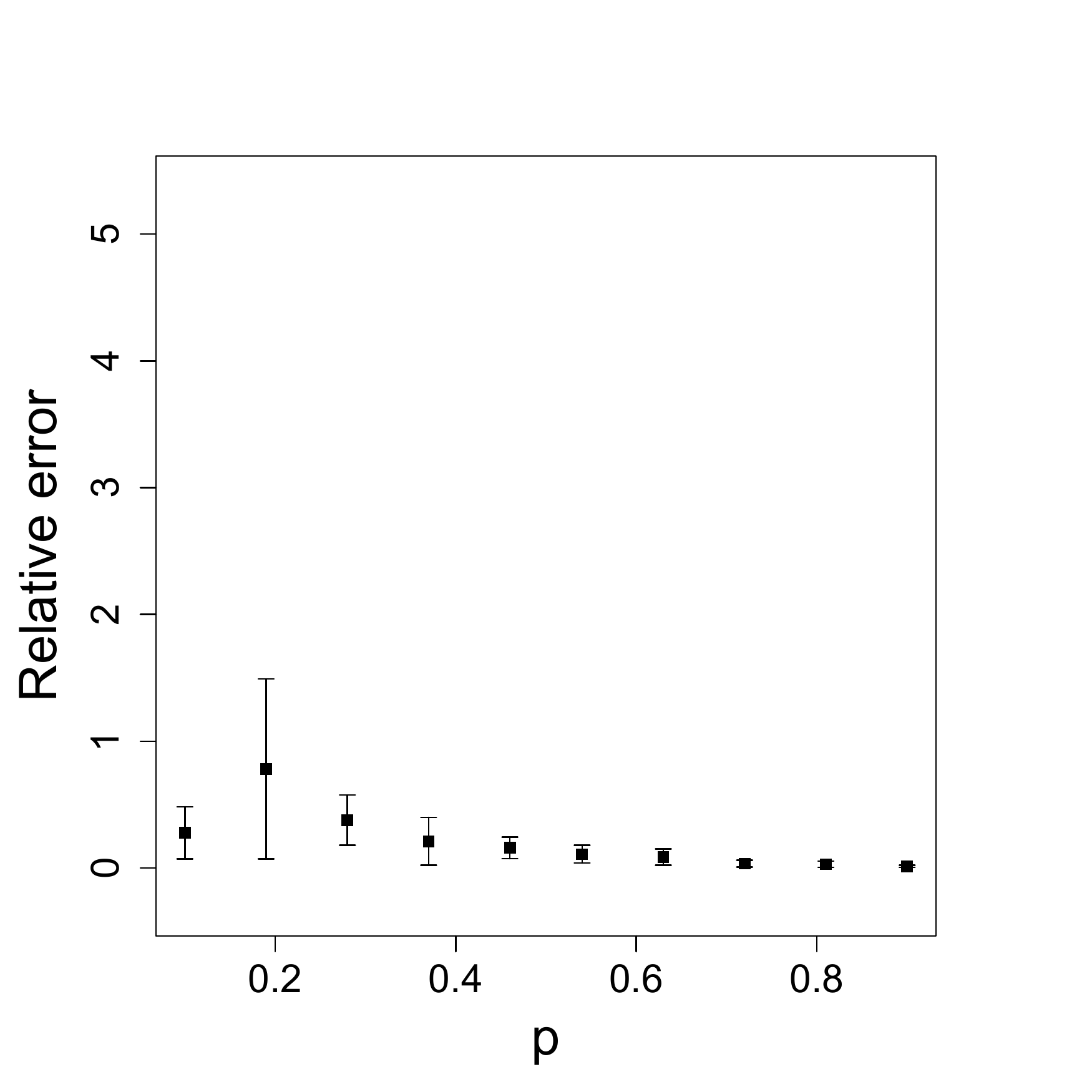}}
\caption{Relative error of counting wedges. In \prettyref{fig:sub13} and \prettyref{fig:sub14}, the parent graph is a Facebook network with $ d = 29 $, $ \sfv(G) = 61 $, $ \sfw(G) = 1039 $. In \prettyref{fig:sub15} and \prettyref{fig:sub16}, the parent graph is a realization of the Erd\"os-R\'enyi graph $ \calG(1000, 0.001) $ with $ d = 7 $, and $ \sfw(G) = 514 $.}
\label{fig:wedge}
\end{figure}

\section{Discussion}

We conclude the paper by mentioning a number of interesting questions that remain open:
\begin{itemize}
	\item As mentioned in the introduction, a more general (and powerful) version of the neighborhood sampling model is to observe a labeled radius-$r$ ball rooted at a randomly chosen vertex \cite{Lovasz12}. 
The current paper focuses on the case of $r=1$.	For $ r = 2 $, we note for example that a triangle could be observed simply by sampling only one of its vertices, i.e., \tikz[scale=0.75,baseline=(zero.base)]{\draw (0,0) node (zero) [vertexdot] {} -- (0.5,0) node[vertexdot] {} -- (0.25,{0.5*sin(60)}) node[vertexdotsolid] {} -- cycle;}. Thus, a Horvitz-Thompson type of estimator is $ \frac{1}{3p} \N(\tikz[scale=0.75,baseline=(zero.base)]{\draw (0,0) node (zero) [vertexdot] {} -- (0.5,0) node[vertexdot] {} -- (0.25,{0.5*sin(60)}) node[vertexdotsolid] {} -- cycle;}, \tG) $ and the variance scales as $1/p$. When $ p $ is small, this outperforms the neighborhood sampling counterpart ($ r = 1 $) in \prettyref{thm:neighborhood_clique}, where the variance scales as $1/p^2$.
Understanding the statistical limits of $r$-hop neighborhood sampling is an interesting and challenging research direction. In particular, the lower bound will potentially involve more complicated graph statistics as opposed neighborhood subgraph counts.
	
	\item In this paper we have focused on counting motifs as induced subgraphs. As shown in \prettyref{eq:injind}, subgraph counts can be expressed linear combinations of induced subgraph counts. However, this does not necessarily mean their sample complexity are the same. Although we do not have a systematic understanding so far, here is a concrete example that demonstrates this: consider estimating the number of (not necessarily) 4-cycles with neighborhood sampling. 
Note that to observe a $C_4$ one only need to sample the two diagonal vertices. Thus, a simple unbiased estimator is $\frac{1}{2p^2} \n(\tikz[scale=0.75,baseline=(zero.base)]{\draw (0,0) node (zero) [vertexdotsolid] {} -- (0.5,0) node[vertexdot] {} -- (0.5,0.5) node[vertexdotsolid] {} -- (0,0.5) node[vertexdot] {} -- cycle; }, \tG)$, whose variance scales as $O(1/p^2)$ and is much smaller than the best error rate for estimating induced $C_4$'s which scales as $1/p^3$, as given by \prettyref{thm:neighborhood_main}. 
The explanation for this phenomenon is that although we have the deterministic relationship $\n(\Square,G) = \s(\Square,G) + \s(\Diamond,G) + \s(\Kfour,G)$ and each of the three subgraph counts can be estimated at the rate of $p^{-3}$, the statistical errors cancel each other and result in a faster rate. 
	
\end{itemize}


\appendix
\section{Auxiliary lemmas}\label{app:appendix}

\begin{lemma}[Kocay's Edge Theorem for Colored Graphs] \label{lmm:nbhd}
Let $ h $ be a bicolored disconnected graph. Then $ \N(h, G) $ can be expressed as a polynomial, independent of $ G $, in $ \N(g, G) $, where $ g $ is bicolored, connected, and $ \sfv_b(g) \leq \sfv_b(h) $. Moreover, if $ \prod_{g\in\calG}\N(g, G)  $ is a term in the polynomial, then $ \sum_{g\in\calG}\sfv_b(g) \leq \sfv_b(h) $ and the corresponding coefficient is bounded by $ 3^{[\sfv_b(h)]^2} $. The number of terms in the polynomial representation is bounded by the number of $ \sfv_b(h) $-tuples $ (g_1, \dots, g_{\sfv_b(h)}) $ of all bicolored neighborhood subgraphs such that $ \sum_{i=1}^{\sfv_b(h)} \sfv_b(g_i) \leq \sfv_b(h) $ and $ \N(g_i, h) \neq 0 $.
\end{lemma}

\begin{proof}
For a disconnected graph $ g' $, note that $ g' $ can be decomposed into two graphs $ g'_1 $ and $ g'_2 $, where $ g'_1 $ is connected and $ \sfv_b(g'_2) \leq \sfv_b(g') - 1 $. Then,
\begin{equation}
\N(g'_1, G)\N(g'_2, G) = \sum_{g} a_g\N(g, G),
\end{equation}
where the sum runs over all graphs $ g $ with $ \sfv_b(g) \leq \sfv_b(g'_1) + \sfv_b(g'_2) = \sfv_b(g') $ and $ a_g $ is the number of decompositions of $ V(g) $ into $ V(g'_1) \cup V(g'_2) $ and $ V_b(g) $ into $ V_b(g'_1) \cup V_b(g'_2) $ (not necessarily disjoint) such that $ g\{ V_b(g'_1) \} \cong g'_1 $ and $ g\{ V_b(g'_2) \} \cong g'_2 $.


The only disconnected graph satisfying the above decomposition property for $ \sfv_b(g) = \sfv_b(g') $ is $ g \cong g' $, and hence
\begin{equation} \label{eq:base}
\N(g', G) =  \frac{1}{a_{g'}}\left[\N(g'_1, G)\N(g'_2, G) - \sum_{g} a_g\N(g, G) \right],
\end{equation}
where $ \sfv_b(g'_2) \leq \sfv_b(g')-1 $ and the sum ranges over all $ g $ that are either connected and $ \sfv_b(g) \leq \sfv_b(g') $ or disconnected and $ \sfv_b(g) \leq \sfv_b(g')-1 $. Furthermore, each $ a_g $ can be bounded by the number of ways of decomposing a set of size $ \sfv_b(g') $ into two sets (with possible overlap), or $ 3^{\sfv_b(g')} $.

We will now prove the following claim using induction.
Let $ h $ be a bicolored disconnected graph. For each $ k < \sfv_b(h) $,
\begin{equation} \label{eq:poly-rep1}
\N(h, G) = \sum_{\calG}c_{\calG}\prod_{g\in\calG}\N(g, G),
\end{equation}
where $ \calG $ contains at least one disconnected $ g' $ for which $ \sfv_b(g') \leq \sfv_b(h)-k $, $ \sum_{g\in \calG} \sfv_b(g) \leq \sfv_b(h) $, $ |c_{\calG}| \leq 3^{k\sfv_b(h)} $, and the number of terms is bounded by the number of $ k $-tuples $ (g_1, \dots, g_k) $ of all bicolored neighborhood graphs such that $ \sum_{i=1}^{k} \sfv_b(g_i) \leq \sfv_b(h) $ and $ \N(g_i, h) \neq 0 $.

The base case $ k = 1$ is established by decomposing $ h $ into two graphs $ h_1 $ and $ h_2 $ with $ h_1 $ connected and $ \sfv_b(h_2) \leq \sfv_b(h) - 1 $ and applying \prettyref{eq:base} with $ g' \cong h $, $ g'_1 \cong h_1 $, and $ g'_2 \cong h_2 $.

Next, suppose \prettyref{eq:poly-rep1} holds. Then applying \prettyref{eq:base} to each disconnected $ g' $, we have
\begin{align}
\N(h, G) & = \sum_{\calG} c_{\calG} \N(g', G)\prod_{g}\N(g, G) \nonumber \\
& = \sum_{\calG} \frac{c_{\calG}}{c_{g'}} [\N(g'_1, G)\N(g'_2, G) - \sum_{h'} c_{h'}\N(h', G)]\prod_{g}\N(g, G) \nonumber \\
& = \sum_{\calG} \frac{c_{\calG}}{c_{g'}} \N(g'_1, G)\N(g'_2, G)\prod_{g}\N(g, G) - \sum_{\calG}\sum_{h'} \frac{c_{\calG}c_{h'}}{c_{g'}}\N(h', G)\prod_{g}\N(g, G). \label{eq:poly-rep2}
\end{align}
Note that $ \sfv_b(g'_2) \leq \sfv_b(g') -1 \leq \sfv_b(h)-(k+1) $ and if $ h' $ is disconnected, then $ \sfv_{b}(h') \leq \sfv_b(g') -1 \leq \sfv_b(h)-(k+1) $. Finally, we observe that \prettyref{eq:poly-rep2} has the form
\begin{equation} \label{eq:poly-rep}
\sum_{\widetilde{\calG}} c_{\widetilde{\calG}} \prod_{g}\N(g, G),
\end{equation}
where $ \widetilde{\calG} $ contains at least one disconnected $ g' $ for which $ \sfv_b(g') \leq \sfv_b(h)-(k+1) $, $ \sfv_b(g') \leq \sfv_b(h)-(k+1) $, $ \sum_{g\in \widetilde{\calG}} \sfv_b(g) \leq \sfv_b(h) $, and $ |c_{\widetilde{\calG}}| \leq \left|\frac{c_{\calG}}{c_{g'}}\right| \vee \left|\frac{c_{\calG}c_{h'}}{c_{g'}} \right| \leq 3^{(k+1)\sfv_b(h)} $. The number of terms is bounded by the number of $ (k+1) $-tuples $ (g_1, \dots, g_{k+1}) $ of all bicolored neighborhood graphs such that $ \sum_{i=1}^{k+1} \sfv_b(g_i) \leq \sfv_b(h) $ and $ \N(g_i, h) \neq 0 $. Repeat this until $ k = \sfv_b(h) $ and so that the right hand side of \prettyref{eq:poly-rep1} contains no disconnected $ g $ in its terms.
\end{proof}

\begin{lemma} \label{lmm:subgraph_bound}
Let $ H $ and $ H' $ be two graphs on $ M $ vertices. Suppose there exists a constant $ B > 0 $ and positive integer $ k $ such that for each connected subgraph $ h $,
\begin{equation*}
|\N(h, H)-\N(h, H')|  \leq BM^{\sfv_{b}(h)-k}.
\end{equation*}
Then for each subgraph $ h $,
\begin{equation*}
|\N(h, H)-\N(h, H')|  \leq BQ_h\sfv_{b}(h)3^{[\sfv_{b}(h)]^2}M^{\sfv_{b}(h)-k},
\end{equation*}
where $ Q_h $ is the number of $ \sfv_{b}(h) $-tuples $ (g_1, \dots, g_{\sfv_{b}(h)}) $ of all bicolored neighborhood graphs such that $ \sum_{i=1}^{\sfv_{b}(h)} \sfv_b(g_i) \leq \sfv_{b}(h) $ and $ \N(g_i, H) \neq 0 $ or $ \N(g_i, H') \neq 0 $.
\end{lemma}

\begin{proof}
Let $ h $ be a disconnected subgraph. By \prettyref{lmm:nbhd}, $ \N(h, H) = \sum_{\calG_h} c_{\calG_h}\prod_{g\in\calG_h}\N(g, H) $, where $ \sum_{g\in\calG_h}\sfv_b(g) \leq \sfv_{b}(h) $, $ g $ is connected,  $ |c_{\calG_h}| \leq 3^{[\sfv_{b}(h)]^2} $, and the number of terms is bounded by $ Q_h $. Thus, using the fact that if $ x_1, \dots, x_n $ and $ y_1, \dots, y_n $ are positive real numbers, $ |x_1\cdots x_n - y_1\cdots y_n| \leq \sum_{i=1}^n|x_i - y_i |x_1\cdots x_{i-1}y_{i+1}\cdots y_n $, we have
\begin{align*}
|\N(h, H)-\N(h, H')|
& \leq \sum_{\calG_h}|c_{\calG_h}|\left|\prod_{g\in\calG_h}\N(g, H)-\prod_{g\in\calG_h}\N(g, H')\right| \\
& \leq \sum_{\calG_h}|c_{\calG_h}|\sum_{i}|\N(g_i, H)-\N(g_i, H')|\prod_{j\leq i-1}\N(g_j, H) \prod_{j\geq i+1}\N(g_j, H'),
\end{align*}
where $ \{ g_i \} $ is an ordering of $ \{ g \}_{g\in\calG_h} $. Next, we use the fact that $ \max\{\N(g, H), \N(g, H') \} \leq \binom{M}{\sfv_b(g)} \leq M^{\sfv_b(g)} $ to bound
\begin{equation*}
\sum_{i}|\N(g_i, H)-\N(g_i, H')|\prod_{j\leq i-1}\N(g_j, H) \prod_{j\geq i+1}\N(g_j, H') \leq B|\calG_h| M^{\sum_{g\in\calG_h} \sfv_b(g)-k}
\end{equation*}
Since $ |\calG_h| \leq \sum_{g\in\calG_h}\sfv_b(g) \leq \sfv_{b}(h) $, the above is further bounded by $ B\sfv_{b}(h)M^{\sfv_{b}(h)-k} $. Thus,
\begin{align*}
|\N(h, H)-\N(h, H')| & \leq B\sfv_{b}(h)M^{\sfv_{b}(h)-k}\sum_{\calG_h}|a_{\calG_h}| \\
& \leq BQ_h\sfv_{b}(h)3^{[\sfv_{b}(h)]^2}M^{\sfv_{b}(h)-k}.
\end{align*}
\end{proof}

Next we present two results on the total variation that will be used in the regime of $p>\frac{1}{d}$.
The main idea is the following: if a subset $T$ of vertices are not sampled, for subgraph sampling, in the observed graph we delete all edges incident to $T$, i.e., the edge set of $G\{T\}$, and for neighborhood sampling, we delete all edges within $T$, that is, the edge set of $G[T]$. Therefore, for two parent graphs, if missing $T$ leads to isomorphic graphs, then by a natural coupling, the total variation between the sampled graphs is at most the probability that $T$ is not completely absent in the sample.

\begin{lemma} \label{lmm:lower-star}
Let $ G_{\theta} = K_{A, \Delta-\theta} + K_{B, \Delta+\theta} $ for integer $ \theta $ between zero and $ \Delta $. 
Consider the neighborhood sampling model with sampling ratio $p$. Suppose $ |\theta - \theta'| \asymp \sqrt{\frac{\Delta}{p}} $ and both $ A $ and $ B $ are at most $ 1/p $. For neighborhood sampling with sampling ratio $ p $, there exists $ 0 < c < 1 $ such that
\begin{equation*}
\TV(P_{\tG_{\theta}}, P_{\tG_{\theta'}}) \leq  c.
\end{equation*}
\end{lemma}
\begin{proof}
Note that $ G_{\theta} $ is the union of two complete bipartite graphs. Suppose that none of the $ A + B $ ``left" side vertices are sampled. Then $ G_{\theta} $ can be described by $ K_{A, X} + K_{B, Y} + (2\Delta-(X+Y))K_1 $, where $ (X, Y) \sim \Binom(\Delta-\theta, p) \otimes \Binom(\Delta+\theta, p) $. Thus, if $ (X', Y') \sim \Binom(\Delta-\theta', p) \otimes \Binom(\Delta+\theta', p) $, then
\begin{equation*}
\TV(P_{\tG_{\theta}}, P_{\tG_{\theta'}}) \leq 1-q^{A+B} + q^{A+B}\TV(P_{(X, Y)}, P_{(X', Y')}).
\end{equation*}
Furthermore, observe that
\begin{equation*}
\TV(P_{(X, Y)}, P_{(X', Y')}) \leq \TV(P_X, P_{X'}) + \TV(P_Y, P_{Y'}),
\end{equation*}
where 
\begin{align*}
\TV(P_X, P_{X'}) & = \TV(\Binom(\Delta-\theta, p), \Binom(\Delta-\theta', p)), \\
\TV(P_Y, P_{Y'}) & = \TV(\Binom(\Delta+\theta, p), \Binom(\Delta+\theta', p)).
\end{align*}
This shows that if $ |\theta - \theta'| \asymp \sqrt{\frac{\Delta}{p}} $ and both $ A $ and $ B $ are $O(\frac{1}{p}) $, then $ \TV(P_{\tG_{\theta}}, P_{\tG_{\theta'}}) $ is less than a constant less than one.
\end{proof}

\begin{lemma} \label{lmm:lower-complete}
Let $ G $, $ H_1 $, and $ H_2 $ be an arbitrary graphs and let $ H = G \vee H_1 $ for and $ H' = G \vee H_2 $. If $ v = \sfv(H_1) = \sfv(H_2) \leq 1/p $, then for neighborhood sampling with sampling ratio $p$,
\begin{equation}
\TV(P_{\tH}, P_{\tH'}) \leq 1 - q^{v} \leq 1-q^{1/p},
\label{eq:tvHH}
\end{equation}
More generally, for $H=(V,E)$ and $H'=(V,E')$ defined on the same set $V$ of vertices, if $T\subset V$ is such that $(V\backslash T, E\backslash E(H[T]))$ and $(V\backslash T, E'\backslash E(H'[T]))$ are isomorphic, then \prettyref{eq:tvHH} holds with $v=|T|$.
\end{lemma}
\begin{proof}
Suppose that none of the $ v $ vertices in $ H_1 $ or $ H_2 $ are sampled. Then $ H_1 $ and $ H_2 $ are isomorphic to each other. Thus,
\begin{equation*}
\TV(P_{\tH}, P_{\tH'}) \leq \prob{\text{at least one vertex in} \;   H_1 \; \text{or} \; H_2 \; \text{is sampled}} = 1 - q^{v}.
\end{equation*}
The second claim follows from the same argument.
\end{proof}

The following lemma, which was used in the proof of Theorems \ref{thm:edge-forest} and \ref{thm:wedge-forest}, relies on a number-theoretic fact:
\begin{lemma} \label{lmm:moments}
There exist two sequences of integers $(\alpha_1,\ldots,\alpha_{k+1})$ and $(\beta_1,\ldots,\beta_{k+1})$ such that
\begin{equation*}
\sum_{x\in [k+1]} x^i \alpha_x = 0 \quad i = 0, 2, 3, \dots, k,
\end{equation*}
\begin{equation*}
\sum_{x=1}^{k+1} x^i \beta_x = 0 \quad i = 0,1, 3, \dots, k,
\end{equation*}
and
\begin{equation*}
\sum_{x\in [k+1]} x \alpha_x = \lcm(1, \dots, k+1),
\end{equation*}
\begin{equation*}
\sum_{x\in [k+1]} x^2 \beta_x  = \lcm^2(1, \dots, k+1),
\end{equation*}
where $\lcm$ stands for the  least common multiple.
Moreover, there exists universal constants $ A $ and $ B $ such that 
\begin{equation} \label{eq:size}
\sum_{x\in [k+1]} |\alpha_x| \leq A^k, \qquad \sum_{x\in [k+1]} |\beta_x| \leq B^k.
\end{equation}
\end{lemma}

\begin{proof}
We first introduce the quantity
\begin{equation*}
\gamma_i = \sum_{x=1}^{k}\frac{(-1)^{x+1}}{x^i}\binom{k}{x}.
\end{equation*}
The key observation is that $ \sum_{x=0}^{k+1}(-1)^x \binom{k+1}{x}D(x) = 0 $ for all polynomials $ D $ with degree less than or equal to $ k $. Hence we can set
\begin{align*}
\alpha_x = \left(\gamma_1-\frac{1}{x}\right)(-1)^x\binom{k+1}{x}\lcm(1, \dots, k+1)
\end{align*}
and
\begin{align*}
\beta_x = \left(\gamma^2_1-\gamma_2 - \frac{\gamma_1}{x} + \frac{1}{x^2}\right)(-1)^x\binom{k+1}{x}\lcm^2(1, \dots, k+1),
\end{align*}
where $ x = 1, 2, \dots, k+1 $.
A well-known number theoretic fact is that the least common multiple of the $k$ integers is in fact significantly smaller than their product. In fact, we have the estimates \cite{Nair1982}, \cite{Hanson1972}
\begin{equation*}
2^{k-1} \leq \lcm(1, \dots, k) \leq 3^k, \; \text{for all} \; k \geq 1,
\end{equation*} 
which shows \prettyref{eq:size}.
\end{proof}

\begin{lemma} \label{lmm:tv-edge-forest}
For the two graphs $ H $ and $ H' $ from \prettyref{thm:edge-forest} constructed with $(\alpha_1, \dots, \alpha_{k+1})$ from \prettyref{lmm:moments}, we have for neighborhood sampling with sampling ratio $ p $,
\begin{equation*}
\TV(P_{\tH}, P_{\tH'}) = O(pA^k + (p\ell A^k)^{k}),
\end{equation*}
provided $ p\ell A^k < 1 $.
\end{lemma}

\begin{proof}
There are four types of connected subgraphs of $ H $ and $ H' $: edge with one black vertex, edge with two black vertices, $ S_u $, $u > 1$ with white center, $ S_u $, $u > 1$ with black center. If $ g $ is an edge with one black vertex $ \N(g, H) = 2\ell\alpha + \ell\sum_{x=1}^{k+1} xw_x $ and $ \N(g, H') = 2\ell\alpha' + \ell\sum_{x=1}^{k+1} xw'_x $. If $ g $ is an edge with two black vertices $ \N(g, H) = \ell\alpha $ and $ \N(g, H') = \ell\alpha' $. If $ g \cong S_u $ with white center, then $ \N(g, H) = \sum_{x=1}^{k+1}w_x\binom{\ell x}{\sfv_b(g)} $ and $ \N(g, H') = \sum_{x=1}^{k+1} w'_x\binom{\ell x}{\sfv_b(g)} $ and furthermore,
\begin{align*}
|\N(g, H)-\N(g, H')| & =
\left|\sum_{x=1}^{k+1}w_x \binom{\ell x}{\sfv_b(g)} - \sum_{x=1}^{k+1}w'_x \binom{\ell x}{\sfv_b(g)}\right| \\
& = \frac{\ell}{\sfv_b(g)}\left|\sum_{x=1}^{k+1}xw_x - \sum_{x=1}^{k+1}xw'_x\right|.
\end{align*}
If $ g \cong S_u $ with black center, then 
\begin{align*}
\N(g, H)= & ~   \sum_{x=1}^{k+1}w_x\binom{\ell x}{\sfv_b(g)-1}\Indc\{ \ell x = u \}\\
\N(g, H') = & ~ 	\sum_{x=1}^{k+1} w'_x\binom{\ell x}{\sfv_b(g)-1}\Indc\{ \ell x = u \} 
\end{align*}
We find that $ |\N(g, H)-\N(g, H')| \leq 2a^k(\ell(k+1))^{\sfv_b(g)-1} $ and $ |\N(g, H)| \leq 2a^k(\ell(k+1))^{\sfv_b(g)} $. 

Let $ v = \sfv(H) = \sfv(H') \leq (\ell(k+1)+1)a^k $. Then
\begin{align*}
\TV(P_{\tH}, P_{\tH'}) \leq \frac{1}{2}\sum_{h:\sfv_{b}(h)\leq k}| \N(h, H) - \N(h, H')|p^{\sfv_{b}(h)}q^{v-\sfv_{b}(h)} + \prob{\Binom(v, p) \geq k+1},
\end{align*}
where the sum runs over all bicolored graphs with at most $ k $ black vertices. By \prettyref{lmm:subgraph_bound}, for each subgraph $ h $,
\begin{equation*}
|\N(h, H)-\N(h, H')|  \leq \sfv_{b}(h)3^{[\sfv_{b}(h)]^2}(2\sfv_{b}(h)a^k(k+3))^{\sfv_{b}(h)}(\ell(k+1))^{\sfv_{b}(h)-1},
\end{equation*}
where we used the bound $ Q_h \leq [\sfv_{b}(h)(k+3)]^{\sfv_{b}(h)} $.
Hence,
\begin{align*}
\TV(P_{\tH}, P_{\tH'}) & \leq \frac{1}{2}\sum_{h:1 \leq \sfv_{b}(h)\leq k}| \N(h, H) - \N(h, H')|p^{\sfv_{b}(h)}q^{v-\sfv_{b}(h)} + \prob{\Binom(v, p) \geq k+1} \\
& \leq \frac{1}{2}\sum_{h:1 \leq \sfv_{b}(h)\leq k}\sfv_{b}(h)3^{[\sfv_{b}(h)]^2}(2\sfv_{b}(h)a^k(k+3))^{\sfv_{b}(h)}(\ell(k+1))^{\sfv_{b}(h)-1}p^{\sfv_{b}(h)}q^{v-\sfv_{b}(h)}  + \\ & \qquad \prob{\Binom((\ell(k+1)+1)a^k, p) \geq k+1} \\
& \leq (pA^k) \sum_{v=0}^k (p\ell A^k)^{v} + \sum_{v=k+1}^{\infty} (p\ell A^k)^{v} \\
& = O(pA^k  + (p\ell A^k)^{k+1}),
\end{align*}
for some constant $ A > 0 $ and provided $ p\ell A^k < 1 $.
\end{proof}

\begin{lemma} \label{lmm:tv-broken-forest}
For the two graphs $ H $ and $ H' $ from \prettyref{thm:broken-forest} constructed with $(\beta_1, \dots, \beta_{k+1})$ from \prettyref{lmm:moments}, we have for neighborhood sampling with sampling ratio $ p $,
\begin{equation*}
\TV(P_{\tH}, P_{\tH'}) = O(pA^k + (p\ell A^k)^2 + (p\ell A^k)^{k}),
\end{equation*}
provided $ p\ell A^k < 1 $.
\end{lemma}

\begin{proof}
There are two types of connected subgraphs of $ H $ and $ H' $: $ S_u $, $u > 1$ with white center and $ S_u $, $u > 1$ with black center. If $ g \cong S_u $ with white center, then $ \N(g, H) = \sum_{x=1}^{k+1}w_x\binom{\ell x}{\sfv_b(g)} $ and $ \N(g, H') = \sum_{x=1}^{k+1} w'_x\binom{\ell x}{\sfv_b(g)} $ and furthermore, since $ \sum_{x=1}^{k+1}x^iw_x = \sum_{x=1}^{k+1}x^iw'_x $ for $ i = 0, 1, 3, \dots, \sfv_b(g) $,
\begin{align*}
|\N(g, H) - \N(g, H')| & = \left|\sum_{x=1}^{k+1}w_x \binom{\ell x}{\sfv_b(g)} - \sum_{x=1}^{k+1}w'_x \binom{\ell x}{\sfv_b(g)}\right|  \\
& = \frac{\ell^2}{\sfv_b(g)(\sfv_b(g)-1)}\left|\sum_{x=1}^{k+1}x^2w_x - \sum_{x=1}^{k+1}x^2w'_x\right|.
\end{align*}
If $ g \cong S_u $ with black center, then 
\begin{align*}
\N(g, H) = & ~  \sum_{x=1}^{k+1}w_x\binom{\ell x}{\sfv_b(g)-1}\Indc\{ \ell x = u \}\\
\N(g, H')= & ~ 	\sum_{x=1}^{k+1} w'_x\binom{\ell x}{\sfv_b(g)-1}\Indc\{ \ell x = u \}
\end{align*}
 We find that $ |\N(g, H)-\N(g, H')| \leq 2a^k(\ell(k+1))^{\sfv_b(g)-1} $ and $ |\N(g, H)| \leq a^k(\ell(k+1))^{\sfv_b(g)} $. 

Let $ v = \sfv(H) = \sfv(H') \leq (\ell(k+1)+1)a^k $. Then
\begin{align*}
\TV(P_{\tH}, P_{\tH'}) \leq \frac{1}{2}\sum_{h:\sfv_{b}(h)\leq k}| \N(h, H) - \N(h, H')|p^{\sfv_{b}(h)}q^{v-\sfv_{b}(h)} + \prob{\Binom(v, p) \geq k+1},
\end{align*}
where the sum runs over all bicolored graphs with at most $ k $ black vertices. By \prettyref{lmm:subgraph_bound}, for each subgraph $ h $ with $ \sfv_{b}(h) \neq 2 $,
\begin{equation*}
|\N(h, H)-\N(h, H')|  \leq \sfv_{b}(h)3^{[\sfv_{b}(h)]^2}(2\sfv_{b}(h)a^k(k+3))^{\sfv_{b}(h)}(\ell(k+1))^{\sfv_{b}(h)-1},
\end{equation*}
where we used the bound $ Q_h \leq [\sfv_{b}(h)(k+3)]^{\sfv_{b}(h)} $.
Hence,
\begin{align*}
\TV(P_{\tH}, P_{\tH'}) & \leq \frac{1}{2}\sum_{h:1 \leq \sfv_{b}(h)\leq k}| \N(h, H) - \N(h, H')|p^{\sfv_{b}(h)}q^{v-\sfv_{b}(h)} + \prob{\Binom(v, p) \geq k+1} \\
& \leq \frac{1}{2}\sum_{h:\sfv_{b}(h)\neq 2,\; \sfv_{b}(h) \leq k}\sfv_{b}(h)3^{[\sfv_{b}(h)]^2}(2\sfv_{b}(h)a^k(k+3))^{\sfv_{b}(h)}(\ell(k+1))^{\sfv_{b}(h)-1}p^{\sfv_{b}(h)}q^{v-\sfv_{b}(h)}  + \\ & \qquad a^k\ell^2p^2 + \prob{\Binom((\ell(k+1)+1)a^k, p) \geq k+1} \\
& \leq (pA^k) \sum_{v=0}^k (p\ell A^k)^{v} + (p\ell A^k)^2 + \sum_{v=k+1}^{\infty} (p\ell A^k)^{v} \\
& = O(pA^k  + (p\ell A^k)^2 + (p\ell A^k)^{k+1}),
\end{align*}
for some constant $ A > 0 $ and provided $ p\ell A^k < 1 $.
\end{proof}

\begin{lemma} \label{lmm:lower-planar}
There exists two planar graphs $ H $ and $ H' $ on order $ \ell $ vertices with matching degree sequences and maximum degree equal to $ \ell + 1 $ such that for neighborhood sampling with sampling ratio $ p $, $ \TV(P_{\tH}, P_{\tH'}) = O(p^2 + p^3\ell^3) $ and $ |\sfw(H) - \sfw(H')| = 3|\sft(H) - \sft(H')| \asymp \ell $ provided $ p = O(1/\ell) $.
Furthermore, there exists two planar graphs $ H $ and $ H' $ on order $ \ell $ vertices such that for neighborhood sampling with sampling ratio $ p $, $ \TV(P_{\tH}, P_{\tH'}) = O(p) $ and $ |\sft(H) - \sft(H')| \asymp \ell $.
\end{lemma}

\begin{proof}
The proof follows from an examination of the two graphs below. Note that $ \N(h, H) = \N(h, H') $ for all connected $ h $ with $ \sfv_b(h) = 1 $ and since $ |\N(h, H) - \N(h, H')| = O(1) $ for all connected $ h $ with $ \sfv_b(h) = 2 $, it follows from \prettyref{lmm:subgraph_bound} with $ k = 2 $ that $ |\N(h, H) - \N(h, H')| = O(1) $ for all $ h $ with $ \sfv_b(h) = 2 $. Thus,
\begin{equation*}
\TV(P_{\tH}, P_{\tH'}) = \sum_h | \N(h, H) - \N(h, H')|p^{\sfv_b(h)}q^{v - \sfv_b(h)} = O(p^2 + \sum_{k=3}^{\infty} \ell^kp^k) = O(p^2 + p^3\ell^3),
\end{equation*}
provided $ p = O(1/\ell) $.
The identity $ |\sfw(H) - \sfw(H')| = 3|\sft(H) - \sft(H')| = \ell-2 $ follows from the fact that $ H $ and $ H' $ have matching degree sequences (corresponding to matching subgraphs from neighborhood sampling with one vertex).

\begin{table} [ht]
\centering
\begin{minipage}{.4\linewidth}
\centering
\caption{The graph $ H $ with $ \ell = 5 $ and $ \sfd(H) = \ell + 1 = 6 $}
\begin{tabular}
{c|c} \hline Copies & Components \\
\hline 1 & \includegraphics[width=0.15\columnwidth]{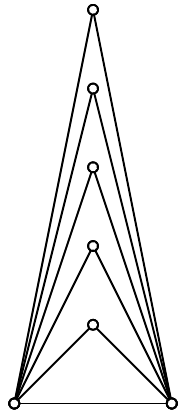} \\
\hline 2 & \includegraphics[width=0.5\columnwidth]{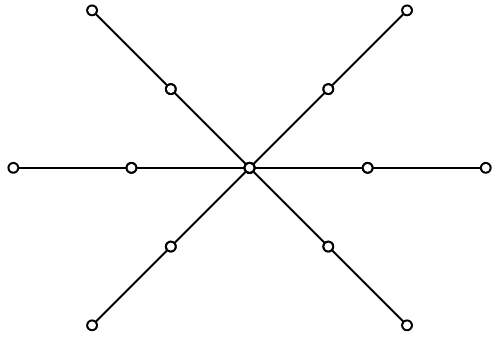} \\
\hline 2 & \includegraphics[width=0.3\columnwidth]{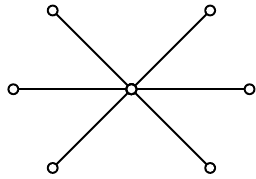} \\
\hline $ \frac{\ell+1}{3} $ & \includegraphics[width=0.2\columnwidth]{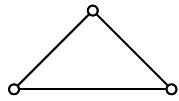} \\
\hline $ 2(\ell+1) $ & \includegraphics[width=0.45\columnwidth]{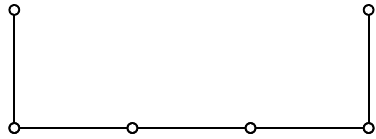} \\
\hline
\end{tabular}
\end{minipage} %
\qquad
\begin{minipage}{.4\linewidth}
\centering
\caption{The graph $ H' $ with $ \ell = 5 $ and $ \sfd(H') = \ell + 1 = 6 $}
\begin{tabular}
{c|c} \hline Copies & Components \\
\hline 1 & \includegraphics[width=0.35\columnwidth]{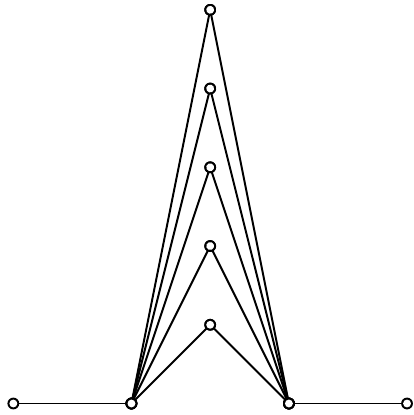} \\
\hline 2 & \includegraphics[width=0.25\columnwidth]{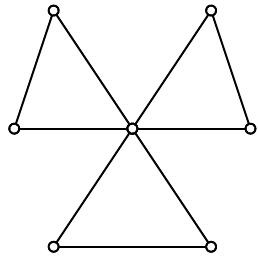} \\
\hline 1 & \includegraphics[width=0.45\columnwidth]{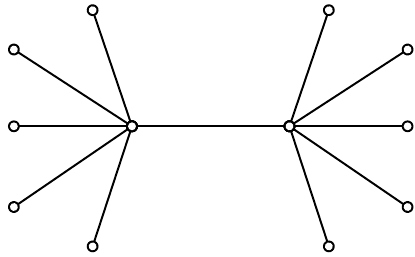} \\
\hline $ \ell+1 $ & \includegraphics[width=0.3\columnwidth]{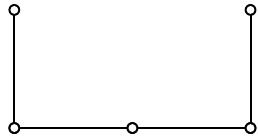} \\
\hline $ 2(\ell+1) $ & \includegraphics[width=0.2\columnwidth]{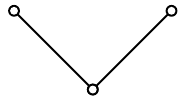} \\
\hline
\end{tabular}
\end{minipage}
\end{table}

For the second statement, consider two planar graphs $H$ and $H'$ on $\ell+2$ vertices, where $H$ consists of $\ell$ triangles sharing a common edge, and $H$ consists of $\ell$ wedges sharing a pair of non-adjacent vertices; see \prettyref{fig:trianglewedge} for an illustration for $\ell=5$.
\begin{figure}[ht]%
\centering
\includegraphics[width=0.1\columnwidth]{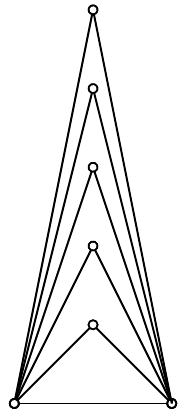}%
~~~~~
	\includegraphics[width=.1\columnwidth]{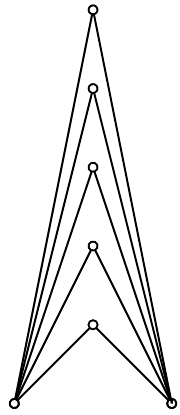}
	\caption{Example of $H$ and $ H' $ with $ \ell = 5 $}%
	\label{fig:trianglewedge}
\end{figure}
Note that if neither of the two highest-degree vertices in each graph (degree $ \ell+1 $ in $ H $ and degree $ \ell $ in $ H' $) are sampled and all incident edges removed, the two graphs are isomorphic. This shows that $ \TV(P_{\tH}, P_{\tH'}) \leq 1-q^2 = O(p) $. Also, note that $ \sft(H)=\ell$ and $\sft(H') = 0$.
\end{proof}


\section{Additional proofs} \label{app:additional_proofs}

\begin{proof}[Proof of Theorems \ref{thm:subgraph-rates-forest}, \ref{thm:wedge-rates-forest}, and \ref{thm:triangle-planar-subgraph}]
The upper bounds are achieved by Horvitz-Thompson estimation as in \prettyref{thm:subgraph-main}. However, for \prettyref{thm:wedge-rates-forest}, we are able to achieve a smaller variance because $ \n(\Bowtie, G) $ is of order $ td $ for planar $ G $ instead of $ td^2 $ and hence $ \Var[\widehat{\sft}_{\sf{HT}}] \lesssim \frac{\n(\Triangle, G))}{p^3} + \frac{\n(\Diamond, G)}{p^2} + \frac{\n(\Bowtie, G)}{p} \lesssim \frac{t}{p^3} + \frac{td}{p^2} + \frac{td}{p} \asymp \frac{t}{p^3} \vee \frac{td}{p^2}  $. For the lower bound, the proof follows the same lines as \prettyref{sec:lb-subgraph} in that we use two different constructions depending on whether $ p \leq 1/d $ or $ p > 1/d $. \\

For edges, let $ H = S_{\ell} $ and $ H' = (\ell+1)S_1 $ with $ \ell = c(d \wedge m) $ for some small constant $ c > 0 $. Then $ \TV(P_{\tH}, P_{\tH'}) \leq p(1-q^{\ell}) \leq p \wedge (\ell p^2) $. \\

For wedges, when $ p \leq 1/d $, let $ H = P_4 + K_1 $ and $ H' = P_3 + P_2 $. Then $ \TV(P_{\tH}, P_{\tH'}) \leq O(p^3) $. When $ p > 1/d $, let $ H = S_{\ell} $ and $ H' = (\ell+1)K_1 $. Then $ \TV(P_{\tH}, P_{\tH'}) \leq p $. Finally set $ \ell = c(d \wedge w) $ for some universal constant $ c > 0 $. \\

Finally, for triangles, let $ H $ be the graph which consists of $ \ell $ triangles that share the same edge plus $ \ell $ isolated vertices. Let $ H' $ be the graph which consists of two $ S_{\ell} $ star graphs with an edge between their roots. Choose $ \ell = c(d \wedge t) $ for some small universal constant $ c > 0 $. Then $ \TV(P_{\tH}, P_{\tH'}) \leq p^2(1-q^{\ell}) \leq p^2 \wedge (p^3 \ell ) $.
\end{proof}

\begin{proof}[Proof of \prettyref{thm:edge-forest}]
Let $ (w_1, \dots, w_{k+1}) $ and $ (w'_1, \dots, w'_{k+1}) $ be two sequences of integers defined by 
$w_x=\max\{\alpha_x,0\}$ and $w_x'=\max\{-\alpha_x,0\}$, where $ (\alpha_1, \dots, \alpha_{k+1}) $ is as in \prettyref{lmm:moments}. 
Consider the disjoint union of stars
\begin{equation*}
H \simeq \sum_{x=1}^{k+1} w_x S_{\ell x} + \ell\alpha S_1 \qquad \text{and} \qquad  H' \simeq \sum_{x=1}^{k+1} w'_x S_{\ell x} + \ell\alpha'  S_1,
\end{equation*}
for integer $ \ell > 1 $.

Note, for example, that $ \sfe(H) = \ell(\sum_{x=1}^{k+1} xw_x + \alpha) $ and $ \sfv(H) = \sfe(H) + \sum_{x=1}^{k+1} w_x + \ell\alpha = \sum_{x=1}^{k+1} (\ell x+1)w_x + 2\alpha \ell $. Thus, $ \sfe(H) \vee \sfe(H') \leq \ell a^k $ for some universal $ a > 0 $. Note that
\begin{equation*}
\sfe(H) - \sfe(H') = \ell(\alpha - \alpha') = \frac{\ell}{2}\left(\sum_{x=1}^{k+1} xw_x - \sum_{x=1}^{k+1} xw'_x\right) \geq \frac{\ell}{2},
\end{equation*}
and by \prettyref{lmm:tv-edge-forest} there exists universal $ A > 0 $ such that
\begin{equation*}
\TV(P_{\tH}, P_{\tH'}) = \frac{1}{2}\sum_h | \N(h, H) - \N(h, H') |p^{\sfv_b(h)}q^{v-\sfv_b(h)} = O(pA^k + (p\ell A^k)^k),
\end{equation*}
provided $ p\ell A^k < 1 $.

By \prettyref{thm:mainlb}, we have
\begin{equation*}
\inf_{\widehat{\sfe}}\sup_{\substack{G \in \calF: ~\sfd(G)\leq d\\ ~~~\sfe(G) \leq m}}\prob{|\widehat{\sfe}-\sfe(G)| \geq \Delta_{\ell} } \geq c.
\end{equation*}
where
\begin{align*}
\Delta_{\ell} & \gtrsim |\sfe(H) - \sfe(H')| \left(\sqrt{ \frac{m}{\sfe(H) \vee \sfe(H') \TV(P_{\tH}, P_{\tH'})}} \wedge \frac{m}{\sfe(H) \vee \sfe(H')} \right) \\ & \gtrsim \sqrt{ \frac{m \ell}{pc^k + (p\ell c^k)^{k}}} \wedge \frac{m}{c^k},
\end{align*}
for some universal constants $ c > 0 $ provided $ p\ell c^k < 1 $. Next, choose
\begin{equation}
\ell =
\begin{cases}
    \left(\frac{1}{pc^k}\right)^{1-1/k} \wedge \frac{m}{a^k} & \text{if } p > \left(\frac{1}{dc^k}\right)^{k/(k-1)} \\
    d \wedge \frac{m}{a^k} & \text{if } p \leq \left(\frac{1}{d^k}\right)^{k/(k-1)}
  \end{cases}.
\end{equation}
Taking $ k = \sqrt{\log\frac{1}{p}} $ yields the desired lower bound.
\end{proof}

\begin{proof}[Proof of \prettyref{thm:broken-forest}]
Let $ (w_1, \dots, w_{k+1}) $ and $ (w'_1, \dots, w'_{k+1}) $ be two sequences of integers defined by  $w_x=\max\{\beta_x,0\}$ and $w_x'=\max\{-\beta_x,0\}$, where $ (\beta_1, \dots, \beta_{k+1}) $ is as in \prettyref{lmm:moments}. Let 
\begin{equation*} 
H \simeq \sum_{x=1}^{k+1} w_x S_{\ell x} \qquad \text{and} \qquad H' \simeq \sum_{x=1}^{k+1} w'_x S_{\ell x},
\end{equation*}
for integer $ \ell > 1 $.
Note, for example, that $ \sfe(H) = \ell\sum_{x=1}^{k+1} xw_x $, $ \sfv(H) = \sum_{x=1}^{k+1} (\ell x+1)w_x $, and $ \sfw(H) =\sum_{x=1}^{k+1} \binom{\ell x}{2}w_x $. This means that $ \sfw(H) \vee \sfw(H') \leq \ell^2 a^{2k} $ for some universal $ a > 0 $. Note that
\begin{equation*}
\sfw(H) - \sfw(H') = \frac{\ell^2}{2}\left(\sum_{x=1}^{k+1} x^2w_x - \sum_{x=1}^{k+1} x^2w'_x\right) \geq \frac{\ell^2}{2}.
\end{equation*}
By \prettyref{lmm:tv-broken-forest}, we have that $\TV(P_{\tH}, P_{\tH'}) = O(pA^k + (p\ell A^k)^2 + (p\ell A^k)^k) $ for some universal $ A > 0 $.
By \prettyref{thm:mainlb}, we have
\begin{equation*}
\inf_{\widehat{\sfw}}\sup_{\substack{G \in \calF: ~\sfd(G)\leq d\\ ~~~\sfw(G) \leq w}}\prob{ |\widehat{\sfw}-\sfw(G)| \geq \Delta_{\ell} } \geq c.
\end{equation*}
where
\begin{align*}
\Delta_{\ell} & \gtrsim |\sfw(H) - \sfw(H')|\left(\sqrt{ \frac{w}{\sfw(H) \vee \sfw(H') \TV(P_{\tH}, P_{\tH'})}}  \wedge \frac{w}{\sfw(H) \vee \sfw(H')}\right) \\
& \gtrsim \sqrt{ \frac{w\ell^2}{pc^k + (p\ell c^k)^2 + (p\ell c^k)^{k}}}  \wedge \frac{w}{c^k},
\end{align*}
for some universal constant $ c > 0 $.
Next, choose $ k = 2 $ and $ \ell = c(d \wedge w^{1/2}) $ when $ p \leq 1/d $ for some universal constant $ c > 0 $. For $ p > 1/d $ and $ w \geq d $, we use \prettyref{lmm:lower-star} with $ A = B = 1 $ and $ \Delta = cd $.  Then $ \sfw(H) \asymp \sfw(H') \asymp d^2 $ and $ |\sfw(H) - \sfw(H')| \asymp d\sqrt{\frac{d}{p}} $, and $ \TV(P_{\tH}, P_{\tH'}) < c < 1 $. By \prettyref{thm:mainlb}, we have $ \inf_{\widehat{\sfw}}\sup_{\substack{G \in \calF: ~\sfd(G)\leq d\\ ~~~\sfw(G) \leq w}}\Expect_G|\widehat{\sfw}-\sfw(G)|^2 \gtrsim \frac{wd}{p} $.
\end{proof}

\begin{proof}[Proof of \prettyref{thm:triangle-planar}]
Let $ R $ denote the minimax risk. The bound $ R \lesssim \frac{td}{p^2} $ follows immediately from \prettyref{thm:neighborhood_main} with $ \omega = 3 $. For the other regimes, we modify the estimator \prettyref{eq:neighborhood-est} from \prettyref{thm:neighborhood_clique}. To accomplish this, observe that $ \n(\Bowtie, G) $ is of order $ td $ for planar $ G $, since the number of triangles that share a common vertex is at most $ d $. Choosing $ \alpha = \frac{1}{2qp^2} $ so that, in the notation of the proof of \prettyref{thm:neighborhood_clique},  $ c_1-1 \asymp \frac{1}{p} $ and $ c_2 - 1 = p^{-2} \qth{2\alpha^2q p^5 +  \pth{ 1 - 2q \alpha p^2}^2} \asymp \frac{1}{p} $, we have $ \Var[\widehat{s}] \lesssim \frac{t}{p^3} \vee \frac{td}{p} $. This yields the bound $ R \lesssim \frac{t}{p^3} \vee \frac{td}{p} $. Thus, $ R \lesssim \left(\frac{t}{p^3} \vee \frac{td}{p}\right) \wedge \frac{td}{p^2} = \left( \frac{t}{p^3} \wedge \frac{td}{p^2} \right)\vee \frac{td}{p} $.
For the lower bound, consider two cases:
\paragraph{Case I: $ p \leq 1/d $.} By \prettyref{lmm:lower-planar}, there exists two planar graphs $ H $ and $ H' $ on order $ \ell $ vertices such that $ \TV(P_{\tH}, P_{\tH'}) = O(p^2 + p^3\ell^3) $ and $ \sft(H) \asymp \sft(H') \asymp |\sft(H) - \sft(H')| \asymp \ell $ provided $ p = O(1/\ell) $. We choose $ \ell = p^{-1/3} \wedge t $ if $ p > 1/d^3 $. Otherwise, if $ p \leq 1/d^3 $, we choose $ \ell = d \wedge t $. By \prettyref{thm:mainlb}, this produces a lower bound of $ R \gtrsim \left( \frac{t}{p^{7/3}} \wedge \frac{td}{p^2} \right) \wedge t^2 $. 

\paragraph{Case II: $ p > 1/d $.}  We use the second statement of \prettyref{lmm:lower-planar} which guarantees the existence of two planar graphs $ H $ and $ H' $ on order $ \ell $ vertices such that $ \TV(P_{\tH}, P_{\tH'}) = O(p) $ and $ \sft(H) \asymp |\sft(H) - \sft(H')| \asymp \ell $. Choosing $ \ell = d \wedge t $ yields the lower bound $ R \gtrsim \frac{td}{p} \wedge t^2 $. 
\end{proof}

\begin{proof}[Proof of \prettyref{thm:adaptive}]
To make $ \widehat{\sfe} $ unbiased, in view of \prettyref{eq:emean}, we set
\begin{equation*}
1 = \Expect [\mathcal{K}_{A}] = pq(f(d_{u}) + f(d_{v}))+ p^2g(d_{u},d_{v}).
\end{equation*}
This determines
\begin{equation*}
g(d_{u},d_{v}) = \frac{1-pq(f(d_{u}) + f(d_{v}))}{p^2}.
\end{equation*}
An easy calculation shows that
\begin{equation*}
\Var[\mathcal{K}_{A}] = \frac{(1-pq(f(d_{u}) + f(d_{v})))^2}{p^2} + pq(f^2(d_{u}) + f^2(d_{v})) - 1
\end{equation*}
and if $ A = \{u, w\} $ and $ A' = \{w, v\} $ in $ G $, then
\begin{equation*}
\Cov[\mathcal{K}_{A}, \mathcal{K}_{A'}] = \frac{q}{p}(1-pf(d_{u}))(1-pf(d_{v})).
\end{equation*}
Otherwise, $\Cov[\mathcal{K}_{A}, \mathcal{K}_{A'}] = 0$ if $ A $ and $ A' $ do not intersect.
Thus,
\begin{align}
\Var[\widehat{\sfe}] & = \frac{q}{p}\sum_{u \neq v}d_{u v}(1-pf(d_{u}))(1-pf(d_{v})) \nonumber \\ & \qquad  +\sum_{\{u, v\} \in E(G) } \left[\frac{(1-pq(f(d_{u}) + f(d_{v})))^2}{p^2} + pq(f^2(d_{u}) + f^2(d_{v})) - 1\right], \label{eq:risk}
\end{align}
where $ d_{u v} $ denotes the cardinality of $ N_G(u) \cap N_G(v) $.
To gain a better idea for how to choose $ f $, we first suppose that $ f \equiv \alpha $.
Thus, \prettyref{eq:risk} reduces to the mean square error of \prettyref{eq:edge-estimator} or
\begin{equation*}
\Var[\widehat{\sfe}] = \frac{2q}{p}\n(P_3, G)(1-p\alpha)^2+ \sfe(G)\frac{q}{p^2}(1+p(1-2\alpha((p-2)p\alpha+2)))
\end{equation*}
Next, let us minimize the above expression over all $ \alpha $. Doing so with 
\begin{equation*}
\alpha' = \left( \frac{1}{p}\right)\frac{p\sfe(G)+p\n(P_3, G)}{p\n(P_3, G)+\sfe(G)(2-p)} + \left( \frac{1}{2p}\right)\frac{2q\sfe(G)}{p\n(P_3, G)+\sfe(G)(2-p)}.
\end{equation*}
yields
\begin{equation} \label{eq:minimum}
\Var[\widehat{\sfe}] = \frac{q^2}{p}\frac{\sfe(G)(\sfe(G)+\n(P_3, G))}{(2-p)\sfe(G)+p\n(P_3, G)}.
\end{equation}
Note that $ \alpha' $ is a convex combination of $ \frac{1}{p} $ and $ \frac{1}{2p} $. These are the values that yield the risk bound for the non-adaptive estimator \prettyref{eq:edge-estimator} in \prettyref{thm:edge}, viz.,
\begin{equation*}
\alpha = \left( \frac{1}{p}\right)\Indc\left\{ d > \frac{1}{p} \right\} + \left( \frac{1}{2p}\right)\Indc \left\{ d \leq \frac{1}{p} \right\}.
\end{equation*}
Of course, this choice of $ \alpha' $ is not feasible since it depends on the unknown quantities $ \sfe(G) $ and $ \n(P_3, G) $. However, noting that $ \sfe(G) = \sum_{u} d_{u}/2 $ and $ \n(P_3, G) = \sum_{u} \binom{d_{u}}{2} $ inspires us to define
\begin{align*}
f(d_{u}) & = \left( \frac{1}{p}\right)\frac{p(\frac{d_{u}}{2})+p\binom{d_{u}}{2}}{p\binom{d_{u}}{2}+(\frac{d_{u}}{2})(2-p)} + \left( \frac{1}{2p}\right)\frac{2q(\frac{d_{u}}{2})}{p\binom{d_{u}}{2}+(\frac{d_{u}}{2})(2-p)} \\
& = \left( \frac{1}{2p}\right)\frac{2pd_{u}}{p(d_{u}-1) +(2-p)} + \left( \frac{1}{2p}\right)\frac{2q}{p(d_{u}-1) +(2-p)} \\
& = \frac{pd_{u}+q}{p(pd_{u}+2q)}.
\end{align*}
With this choice of $ f $, we will verify that the variance and covariance terms in \prettyref{eq:risk} also yield the rate \prettyref{eq:edge-var}. Note that
\begin{align*} 
\frac{q}{p}\sum_{u \neq v}d_{u v}\left[\frac{q}{pd_u+2q}\right]\left[\frac{q}{pd_v+2q}\right]
\leq \frac{q}{p}\sum_{u \neq v}\frac{d_{u v}}{pd_u+2q}
\leq \frac{dq}{p}\sum_{u}\frac{d_u}{pd_u+2q}
& \\ \leq \frac{Ndq}{p}\frac{\sfe(G)}{p\sfe(G)+qN} 
\leq \frac{Nd}{p^2} \wedge \frac{\sfe(G)d}{p},
\end{align*}
where the second last inequality follows from the concavity of $ x \mapsto \frac{x}{px + 2q} $ for $ x \geq 0 $.
The variance term has the bound
\begin{equation*}
\sum_{\{u, v\} \in E(G) } \left[\frac{(1-pq(f(d_{u}) + f(d_{v})))^2}{p^2} + pq(f^2(d_{u}) + f^2(d_{v})) - 1\right] \lesssim \sfe(G) \left(\left( \frac{1}{p^2} \wedge d^2 \right) \vee \frac{1}{p} \right),
\end{equation*}
which follows from $ \frac{(1-pq(f(d_{u}) + f(d_{v})))^2}{p^2}  \lesssim \frac{1}{p^2} \wedge (d^2_{u} + d^2_{v} ) $ and $ pq(f^2(d_{u}) + f^2(d_{v})) \lesssim \frac{1}{p} $.
\end{proof}

\section{Neighborhood sampling without colors}
\label{app:nocolor}
In this appendix we demonstrate the usefulness of the color information (namely, which vertices are sampled) in neighborhood sampling by showing that without observing the colors, the performance guarantees in \prettyref{thm:edge} are no longer unattainable in certain regimes.

\begin{theorem} \label{thm:unlabelled}
Let $ \calF $ denote the collection of all forests. Consider the neighborhood sampling model without observing the colors $ \{ b_v : v\in V \} $. Then
\begin{equation} \label{eq:nolabel}
\inf_{\widehat{\sfe}}\sup_{\substack{G \in \calF: ~\sfd(G)\leq d\\ ~~~\sfe(G) \leq m}}\Expect_G|\widehat{\sfe}-\sfe(G)|^2 \gtrsim mp( d\wedge m).
\end{equation}
\end{theorem}
\begin{proof}

Let $ M = m/k $, where $ k = d \wedge m $ and set $ \mathcal{F}_0 = \{ G_{\underline{\theta}} : G_{\underline{\theta}} = S_{\theta_1} + \dots + S_{\theta_M}, \; \underline{\theta} = (\theta_1,\dots,\theta_M) \in [k]^M \} $. Note that for each $ \underline{\theta} \in [k]^M $, $ \sfe(G_{\underline{\theta}}) = \|\underline{\theta}\|_1 $. Thus, if $ \underline{X} = (X_1, \dots, X_M) $, where $ \{X_i\} $ are independent and $ X_i \sim p \delta_{\theta_i} + q \Binom(\theta_i,p) $ for $ i \in [M] $, then
\begin{equation*}
\inf_{\widehat{\sfe}}\sup_{\substack{G \in \calF: ~\sfd(G)\leq d\\ ~~~\sfe(G) \leq m}}\Expect_G|\widehat{\sfe}-\sfe(G)|^2 \geq \inf_{g}\sup_{ \underline{\theta} \in [d]^M }\Expect_{\underline{\theta}}|\|\underline{\theta}\|_1 - g(\underline{X})|^2.
\end{equation*}
By the minimax theorem,
\begin{align*}
\inf_{g}\sup_{ \underline{\theta} \in [k]^m }\Expect_{\underline{\theta}}|\|\underline{\theta}\|_1 - g(\underline{X})|^2 
& = \sup_{ \underline{\theta}\in \underline{\pi} }\inf_{g}\Expect_{\underline{\theta}}|\|\underline{\theta}\|_1 - g(\underline{X})|^2 = \sup_{ \underline{\theta} \in \underline{\pi} }\Expect_{\underline{X}}\Expect_{\underline{\theta}| \underline{X}}|\|\underline{\theta}\|_1 - \Expect_{\underline{\theta}|\underline{X}} \|\underline{\theta}\|_1|^2 \\
& \geq \sup_{ \underline{\theta} \in \pi^{\otimes M} }\Expect_{\underline{X}}\Expect_{\underline{\theta}| \underline{X}}|\|\underline{\theta}\|_1 - \Expect_{\underline{\theta}|\underline{X}} \|\underline{\theta}\|_1|^2
 = M\sup_{ \theta \in \pi } \Expect_X\Expect_{\theta|X}|\theta - \Expect_{\theta|X}\theta|^2 \\
& =  M\inf_{g}\sup_{ \theta \in [d \wedge m] }\Expect_\theta|\theta - g(X)|^2
\asymp m\left(p k\vee \pth{ \frac{1}{p} \wedge k } \right) \\
& \gtrsim mp(d \wedge m),
\end{align*}
where $ X \sim \delta_{\theta} + q \Binom(\theta, p) $ and the second to last line follows from \prettyref{lmm:mixbin} below. 
\end{proof}

\begin{remark}
Note that when $ p > (1/d)^{1/3} $ and $ m \geq d $, the minimax lower bound \prettyref{eq:nolabel} is strictly greater than the minimax risk in \prettyref{thm:edge-forest}, thus confirming the intuition that the knowledge of which vertices are sampled provide useful information.
On the other hand, the Horvitz-Thompson estimator \prettyref{eq:HTn-edge} can be implemented without the color information and achieve the error bound 
$O(\frac{md}{p})$ in \prettyref{eq:varHT-n}. Comparing with \prettyref{thm:subgraph-main}, we conclude that neighborhood sampling is at least as informative as subgraph sampling, even if the colors are not observed. This is intuitive because neighborhood sampling reveals more edges from the parent graph.
\end{remark}

\begin{lemma}
\label{lmm:mixbin}	
	Given $\theta\in [k]$, let $X$ be distributed according to $p \delta_\theta + q \Binom(\theta,p)$. Assume that $p\leq 1/2$. Then
\begin{equation}
\inf_g \sup_{\theta \in [k]} \Expect_\theta[|\theta-g(X)|^2] \asymp p k^2 \vee \pth{ \frac{k}{p} \wedge k^2 }.
\label{eq:mixbin}
\end{equation}	
Moreover, the minimax rate is achieved by the estimator $ \widehat{g}(X) = k\wedge \frac{X}{p} $.
\end{lemma}
\begin{proof}
	Denote the minimax risk by $R$.
	Let $\id$ denote the identity map.
	Given  any estimator $g$, without loss of generality, we assume $g: \{0,\ldots,k\} \to [0,k]$.
	Since $\Expect_\theta[(\theta-g(X))^2] = p (\theta-g(\theta))^2 +  q \Expect_{X\sim\Binom(\theta,p)}[(\theta-g(X))^2]$, we have
	\begin{equation}
	\sup_{\theta \in [k]} \Expect_\theta[|\theta-g(X)|^2] \geq p \|\id-g\|_\infty^2.
	\label{eq:lb1}
	\end{equation}
	Also,
	$(\theta-g(X))^2
	\geq - (X-g(X))^2 + (\theta-X)^2/2$, and hence
	\[
	\Expect_{X\sim\Binom(\theta,p)}[(\theta-g(X))^2] 
	\geq   - \|\id-g\|_\infty^2 +  \frac{1}{2} (q^2\theta^2+pq\theta).
	\]
	Therefore
	\begin{equation}
	\sup_{\theta \in [k]} \Expect_\theta[|\theta-g(X)|^2] \geq - q \|\id-g\|_\infty^2 +  \frac{q}{2} (q^2k ^2+pqk).	
	\label{eq:lb2}
	\end{equation}
	Combining \prettyref{eq:lb1} and \prettyref{eq:lb2}, we get
	\begin{equation}
	R \geq \frac{pq}{2} (q^2k ^2+pqk) \asymp pk^2.
	\label{eq:lb3}
	\end{equation}
	Next by the minimax theorem, 
	\[
	R = \sup_\pi \inf_g \Expect_\pi[|\theta-g(X)|^2] 
	= \sup_\pi \inf_g  \pth{p \underbrace{\Expect_{\theta\in\pi}[|\theta-g(\theta)|^2]}_{\in [0,k^2]}  + q \Expect_{\theta\in \pi,X\sim\Binom(\theta,p)} [|\theta-g(X)|^2]  }.
	\]
	We also know that
	\[
\sup_\pi \inf_g   \Expect_{\theta\in \pi,X\sim\Binom(\theta,p)}[|\theta-g(X)|^2]  
=  \inf_g \sup_{\theta \in [k]} \Expect_{\theta\in \pi,X\sim\Binom(\theta,p)}[|\theta-g(X)|^2]  \asymp \frac{k}{p} \wedge k^2. 
	\]
	Therefore we have
	\[
\frac{k}{p} \wedge k^2
	\lesssim R \lesssim pk^2 + \frac{k}{p} \wedge k^2.
	\]
	Combining with \prettyref{eq:lb3} yields the characterization \prettyref{eq:mixbin}.
\end{proof}

\section{Lower bounds for other motifs}
\label{app:othermotif}

\begin{theorem}[Wedges] \label{thm:broken}
For neighborhood sampling with sampling ratio $ p $,
\begin{equation*}
\inf_{\widehat{\sfw}}\sup_{\substack{G: ~\sfd(G)\leq d\\ ~~~\sfw(G) \leq w}}\Expect_G|\widehat{\sfw}-\sfw(G)|^2\asymp \frac{wd}{p^2} \wedge w^2.
\end{equation*}
\end{theorem}

\begin{proof}
For the lower bound, consider two cases:

\paragraph{Case I: $ p \leq 1/d $.}
Let $ h = P_5 $ and $ h' = K_3 + K_2 $. For each node in the original graph, we associate $ \ell $ distinct isolated vertices and connect each pair of vertices by an edge if and only if they were connected in the original graph. Call these expanded graphs $ H $ and $ H' $. Note that $ H $ and $ H' $ that have matching degree sequences $(2,2,2,1,1)$ and hence $ \TV(P_{\tH}, P_{\tH'}) = O(\ell^2p^2) $. Furthermore, $ \s(P_3, H) \asymp \s(P_3, H') \asymp |\s(P_3, H) - \s(P_3, H')| \asymp \ell^3 $. If $ \ell = c( d \wedge w^{1/3}) $, then by \prettyref{thm:mainlb} with $ M = w/\ell^3 $, $ \inf_{\widehat{\sfw}}\sup_{G\in\calG(w, d) }\Expect_G|\widehat{\sfw}-\sfw(G)|^2 \gtrsim \frac{w\ell}{p^2} \wedge w^2 \asymp \frac{wd}{p^2} \wedge w^2 $.

\paragraph{Case II: $ p > 1/d $.}
We use \prettyref{lmm:lower-complete} with $ G = K_{\ell} $, $ H_1 = K_{1/p} + K_{1/p} $, and $ H_2 = K_{2/p} $. This gives us two graphs $ H $ and $ H' $ with $ \s(P_3, H) = |\s(P_3, H) - \s(P_3, H')| \asymp \ell/p^2 $. 
 By \prettyref{thm:mainlb} with $ M = w/(\ell/p^2) $, $ \inf_{\widehat{\sfw}}\sup_{G\in\calG(w, d) }\Expect_G|\widehat{\sfw}-\sfw(G)|^2 \gtrsim \frac{w\ell}{p^2} \wedge w^2 $. Let $ \ell = cd $ if $ \frac{d}{p^2} \leq w $ and $ \ell = cp^2w $ if $ \frac{d}{p^2} > w $, for some small constant $ c $. In either case, we find that $\sfw(H) \leq w $, $ \sfw(H') \leq w $, and $ \inf_{\widehat{\sfw}}\sup_{G\in\calG(w, d) }\Expect_G|\widehat{\sfw}-\sfw(G)|^2 \asymp \frac{wd}{p^2} \wedge w^2 $.
\end{proof}

\paragraph{Lower bound for motifs of size four}

It remains to show that holds the result in \prettyref{thm:neighborhood_main} that holds for $K_4$, namely,
\begin{equation}
\inf_{\widehat{\s}} \sup_{\substack{G: ~\sfd(G)\leq d\\ ~~~\s(h,G) \leq s}} 
\Expect_G|\widehat{\s}-\s(K_{\omega}, G)|^2 = \Theta \left( \frac{s d}{p^3} \wedge  \frac{s d^2}{p^2} \wedge s^2 \right)
\label{eq:allfour}
\end{equation}
continues to hold for $h =\Square, \Paw, \Diamond$ and $\Claw$. For the case of $p< 1/d$, the construction for $K_4$ in \prettyref{eq:HH-nhbd4} works simultaneously for all motifs, because each motif is contained in one of $H$ and $H'$ and not the other. Next we consider the case of $p > 1/d$. The construction is ad hoc and similar to those in \prettyref{thm:subgraph-main} and \prettyref{thm:neighborhood_clique}.
\begin{itemize}
	\item 
For $h=\Claw$, we use the clique construction:
label the root as $v_1$ and the leaves as $v_2,v_3,v_4$. Define the graph $H$ as follows: Expand $v_1$ into a clique $S_1$ of size $\ell$, and for $i=2,3,4$, expand each $v_i$ into a clique $S_i$ of size $1/p$. Connect each pair of vertices $u_i\in S_i$ and $u_j\in S_j$ for $i\neq j$ if and only if $v_i$ and $v_j$ are connected in the motif $h$. This defines a graph $H$ on $\ell+3/p$ vertices. Repeat the same construction with $h$ replaced by \Paw, where the degree-one vertex is $v_1$. Note that if we remove the edges between the set of vertices $T\triangleq S_2\cup S_3\cup S_4$, for $H$ and $H'$ the resulting graph is isomorphic. Thus by Lemma \prettyref{lmm:lower-complete}, we have $\TV(P_{\tH},P_{\tH'}) \leq 1 - (1-p)^{3/p} \leq 0.9$ if $p \leq 1/2$. Furthermore, note that $\s(\Claw,H')=0$ and $\s(\Claw,H)=\ell/p^3$. Finally, taking $\ell = c(d \wedge \frac{s}{\ell/p^3})$ for some small constant $c$ and invoking \prettyref{thm:mainlb}, we obtain the desired lower bound $ \frac{s d}{p^3}  \wedge s^2 $ in \prettyref{eq:allfour}.

\item For $h=\Paw$, use the same construction as above with $H$ and $H'$ swapped.

\item For $h=\Diamond$, we repeat the clique construction of $H$ with $v_1$ being any of the degree-three vertices in $h$, and of $H'$ with $h'=\Kfour$; in other words, we simply have $H'=K_{\ell+3/p}$.

\item For $h=\Square$, we repeat the clique construction of $H$ with $v_1$ being any vertex in $h$, and of $H'$ with $h'=
\tikz[scale=0.75,baseline=(zero.base)]{\draw (0,0) node (zero) [vertexdot] {} -- (0,0.5) node[vertexdot] {} -- (0.5,0.5) node[vertexdot] {}; \draw (0.5,0) node[vertexdot] {};}$, with $v_1$ being the degree-two vertices.

%
%
%

\end{itemize}

%% file: motif.bbl
\begin{thebibliography}{10}

\bibitem{aliakbarpour2017sublinear}
Maryam Aliakbarpour, Amartya~Shankha Biswas, Themis Gouleakis, John Peebles,
  Ronitt Rubinfeld, and Anak Yodpinyanee.
\newblock Sublinear-time algorithms for counting star subgraphs via edge
  sampling.
\newblock {\em Algorithmica}, pages 1--30.

\bibitem{Apicella2012}
Coren~L. Apicella, Frank~W. Marlowe, James~H. Fowler, and Nicholas~A.
  Christakis.
\newblock Social networks and cooperation in hunter-gatherers.
\newblock {\em Nature}, 481(7382):497--501, 01 2012.

\bibitem{Bickel2011}
Peter~J. Bickel, Aiyou Chen, and Elizaveta Levina.
\newblock The method of moments and degree distributions for network models.
\newblock {\em Ann. Statist.}, 39(5):2280--2301, 2011.

\bibitem{Biggs1978}
Norman Biggs.
\newblock On cluster expansions in graph theory and physics.
\newblock {\em Quart. J. Math. Oxford Ser. (2)}, 29(114):159--173, 1978.

\bibitem{CL11}
T.T. Cai and M.~G. Low.
\newblock Testing composite hypotheses, {Hermite} polynomials and optimal
  estimation of a nonsmooth functional.
\newblock 39(2):1012--1041, 2011.

\bibitem{Capobianco72}
Michael Capobianco.
\newblock Estimating the connectivity of a graph.
\newblock In Y.~Alavi, D.~R. Lick, and A.~T. White, editors, {\em Graph Theory
  and Applications}, pages 65--74, Berlin, Heidelberg, 1972. Springer Berlin
  Heidelberg.

\bibitem{Chandrasekhar2011}
Arun Chandrasekhar and Randall Lewis.
\newblock Econometrics of sampled networks.
\newblock 2011.

\bibitem{chen2015graph}
Hao Chen, Nancy Zhang, et~al.
\newblock Graph-based change-point detection.
\newblock {\em The Annals of Statistics}, 43(1):139--176, 2015.

\bibitem{chu2017asymptotic}
Lynna Chu and Hao Chen.
\newblock Asymptotic distribution-free change-point detection for modern data.
\newblock {\em arXiv preprint arXiv:1707.00167}, 2017.

\bibitem{Duffield2014}
Graham Cormode and Nick Duffield.
\newblock Sampling for big data: a tutorial.
\newblock In {\em Proceedings of the 20th ACM SIGKDD International Conference
  on Knowledge Discovery and Data Mining}, pages 1975--1975. ACM, 2014.

\bibitem{Eden2015}
Talya Eden, Amit Levi, Dana Ron, and C.~Seshadhri.
\newblock Approximately counting triangles in sublinear time.
\newblock In {\em 2015 {IEEE} 56th {A}nnual {S}ymposium on {F}oundations of
  {C}omputer {S}cience---{FOCS} 2015}, pages 614--633. IEEE Computer Soc., Los
  Alamitos, CA, 2015.

\bibitem{Erdos1979}
Paul Erd\"os, L\'aszl\'o Lov\'asz, and Joel Spencer.
\newblock Strong independence of graphcopy functions.
\newblock In {\em Graph theory and related topics ({P}roc. {C}onf., {U}niv.
  {W}aterloo, {W}aterloo, {O}nt., 1977)}, pages 165--172. Academic Press, New
  York-London, 1979.

\bibitem{Feige2006}
Uriel Feige.
\newblock On sums of independent random variables with unbounded variance and
  estimating the average degree in a graph.
\newblock {\em SIAM J. Comput.}, 35(4):964--984, 2006.

\bibitem{GaoLafferty2017}
Chao Gao and John Lafferty.
\newblock Testing network structure using relations between small subgraph
  probabilities.
\newblock {\em arXiv preprint arXiv:1704.06742}, 2017.

\bibitem{Goldreich2008}
Oded Goldreich and Dana Ron.
\newblock Approximating average parameters of graphs.
\newblock {\em Random Structures Algorithms}, 32(4):473--493, 2008.

\bibitem{Gonen2011}
Mira Gonen, Dana Ron, and Yuval Shavitt.
\newblock Counting stars and other small subgraphs in sublinear-time.
\newblock {\em SIAM J. Discrete Math.}, 25(3):1365--1411, 2011.

\bibitem{Han2005}
Jing-Dong~J Han, Denis Dupuy, Nicolas Bertin, Michael~E Cusick, and Marc Vidal.
\newblock Effect of sampling on topology predictions of protein-protein
  interaction networks.
\newblock {\em Nature Biotechnology}, 23(7):839, 2005.

\bibitem{Handcock2010}
Mark~S. Handcock and Krista~J. Gile.
\newblock Modeling social networks from sampled data.
\newblock {\em Ann. Appl. Stat.}, 4(1):5--25, 2010.

\bibitem{Hanson1972}
Denis Hanson.
\newblock On the product of the primes.
\newblock {\em Canad. Math. Bull.}, 15:33--37, 1972.

\bibitem{HT52}
Daniel~G Horvitz and Donovan~J Thompson.
\newblock A generalization of sampling without replacement from a finite
  universe.
\newblock {\em Journal of the American Statistical Association},
  47(260):663--685, 1952.

\bibitem{JVHW15}
Jiantao Jiao, Kartik Venkat, Yanjun Han, and Tsachy Weissman.
\newblock Minimax estimation of functionals of discrete distributions.
\newblock {\em IEEE Transactions on Information Theory}, 61(5):2835--2885,
  2015.

\bibitem{KlusowskiWu2017-cc}
Jason~M. Klusowski and Yihong Wu.
\newblock Estimating the number of connected components in a graph via subgraph
  sampling.
\newblock {\em arXiv preprint arXiv:1801.04339}, 2018.

\bibitem{Kocay1982}
W.~L. Kocay.
\newblock Some new methods in reconstruction theory.
\newblock In {\em Combinatorial mathematics, {IX} ({B}risbane, 1981)}, volume
  952 of {\em Lecture Notes in Math.}, pages 89--114. Springer, Berlin-New
  York, 1982.

\bibitem{Kolaczyk2009}
Eric~D Kolaczyk.
\newblock {\em Statistical Analysis of Network Data: Methods and Models}.
\newblock Springer Science \& Business Media, 2009.

\bibitem{Kolaczyk2017}
Eric~D. Kolaczyk.
\newblock {\em Topics at the Frontier of Statistics and Network Analysis:
  (Re)Visiting the Foundations}.
\newblock SemStat Elements. Cambridge University Press, 2017.

\bibitem{LNS99}
Oleg Lepski, Arkady Nemirovski, and Vladimir Spokoiny.
\newblock On estimation of the {$L_r$} norm of a regression function.
\newblock 113(2):221--253, 1999.

\bibitem{Leskovec2006}
Jure Leskovec and Christos Faloutsos.
\newblock Sampling from large graphs.
\newblock In {\em Proceedings of the 12th ACM SIGKDD International Conference
  on Knowledge Discovery and Data Mining}, pages 631--636. ACM, 2006.

\bibitem{snapnets}
Jure Leskovec and Andrej Krevl.
\newblock {SNAP Datasets}: {Stanford} large network dataset collection.
\newblock \url{http://snap.stanford.edu/data/egonets-Facebook.html}, June 2014.

\bibitem{Lovasz12}
L{\'a}szl{\'o} Lov{\'a}sz.
\newblock {\em Large Networks and Graph Limits}, volume~60.
\newblock American Mathematical Society, 2012.

\bibitem{LS10}
L{\'a}szl{\'o} Lov{\'a}sz and Bal{\'a}zs Szegedy.
\newblock The graph theoretic moment problem.
\newblock {\em arXiv preprint arXiv:1010.5159}, 2010.

\bibitem{Alon2002}
Ron Milo, Shai Shen-Orr, Shalev Itzkovitz, Nadav Kashtan, Dmitri Chklovskii,
  and Uri Alon.
\newblock Network motifs: simple building blocks of complex networks.
\newblock {\em Science}, 298(5594):824--827, 2002.

\bibitem{MosselNeemanSly2015}
Elchanan Mossel, Joe Neeman, and Allan Sly.
\newblock Reconstruction and estimation in the planted partition model.
\newblock {\em Probab. Theory Related Fields}, 162(3-4):431--461, 2015.

\bibitem{Nair1982}
M.~Nair.
\newblock On {C}hebyshev-type inequalities for primes.
\newblock {\em Amer. Math. Monthly}, 89(2):126--129, 1982.

\bibitem{Prvzulj2004}
Natasa Pr{\v{z}}ulj, Derek~G Corneil, and Igor Jurisica.
\newblock Modeling interactome: scale-free or geometric?
\newblock {\em Bioinformatics}, 20(18):3508--3515, 2004.

\bibitem{uetz2000comprehensive}
Peter Uetz, Loic Giot, Gerard Cagney, Traci~A Mansfield, Richard~S Judson,
  James~R Knight, Daniel Lockshon, Vaibhav Narayan, Maithreyan Srinivasan,
  Pascale Pochart, et~al.
\newblock A comprehensive analysis of protein--protein interactions in
  saccharomyces cerevisiae.
\newblock {\em Nature}, 403(6770):623, 2000.

\bibitem{wasserman1994social}
Stanley Wasserman and Katherine Faust.
\newblock {\em Social network analysis: Methods and applications}, volume~8.
\newblock Cambridge university press, 1994.

\bibitem{west-book}
Douglas~B. West.
\newblock {\em Introduction to graph theory}.
\newblock Prentice Hall, Inc., Upper Saddle River, NJ, 1996.

\bibitem{Whitney1932}
Hassler Whitney.
\newblock The coloring of graphs.
\newblock {\em Ann. of Math. (2)}, 33(4):688--718, 1932.

\bibitem{WY14}
Yihong Wu and Pengkun Yang.
\newblock Minimax rates of entropy estimation on large alphabets via best
  polynomial approximation.
\newblock {\em IEEE Transactions on Information Theory}, 62(6):3702--3720,
  2016.

\end{thebibliography}
